\documentclass[11pt,twoside]{article}
\usepackage{amssymb,amsmath,amsthm,amsfonts,mathrsfs,hyperref}
\usepackage{times}
\usepackage{enumerate}
\usepackage{cite,titletoc}
\usepackage{color}
\usepackage[toc,page,title,titletoc,header]{appendix}

\allowdisplaybreaks

\def\cB{\mathcal B}
\def\cC{\mathcal C}

\def\cK{\mathcal K}
\def\cL{\mathcal L}

\def\cN{\mathcal N}

\def\t{\tilde}
\def\cQ{\mathcal Q}

\def\cV{\mathcal V}

\def\cX{\mathcal X}
\def\cY{\mathcal Y}

\def\sF{\mathscr F}

\def\f{\frac}

\def\sV{\mathscr V}
\def\N{\mathop{\mathbb N\kern 0pt}\nolimits}
\def\Z{\mathop{\mathbb Z\kern 0pt}\nolimits}
\def\Q{\mathop{\mathbb Q\kern 0pt}\nolimits}
\def\R{\mathop{\mathbb R\kern 0pt}\nolimits}
\def\T{\mathop{\mathbb T\kern 0pt}\nolimits}
\def\C{\mathop{\mathbb C\kern 0pt}\nolimits}

\def\ds{\displaystyle}
\def\supp{\mathop{\rm supp}\nolimits}

\def\p{\partial}
\def\eps{\epsilon}
\def\ve{\varepsilon}

\def\ls{\lesssim}
\def\gt{\gtrsim}
\def\med{\mathop{\rm med}\nolimits}

\newcommand{\w}[1]{\langle {#1} \rangle}
\def\id{\textbf{1}}

\hypersetup{colorlinks=true,linkcolor=blue,citecolor=red,urlcolor=cyan}

\theoremstyle{plain}
\newtheorem{theorem}{Theorem}[section]
\newtheorem{proposition}[theorem]{Proposition}
\newtheorem{lemma}[theorem]{Lemma}
\newtheorem{corollary}[theorem]{Corollary}

\theoremstyle{definition}

\newtheorem{remark}{Remark}[section]

\numberwithin{equation}{section}

\title{Almost global solutions of 1D nonlinear Klein-Gordon equations with small weakly decaying initial data}

\author{Hou Fei$^{1,*}$ \qquad Tao Fei$^{2,*}$ \qquad
  Yin Huicheng$^{3, } $\footnote{Hou Fei (\texttt{fhou$@$nju.edu.cn}), Tao Fei (\texttt{1458527731@qq.com}) and
    Yin Huicheng (\texttt{huicheng$@$nju.edu.cn}, \texttt{05407$@$njnu.edu.cn}) are supported by the NSFC (No.11731007, No.12101304).
    Hou Fei is also supported by the NSF of Jiangsu Province (No. BK20210170).
    In addition, Yin Huicheng is supported by the National key research and development program of China (No.2020YFA0713803).}\\
    [12pt] {\small 1. Department of Mathematics, Nanjing University, Nanjing 210093, China}\\
  {\small 2. School of Science, Nanjing University of Posts and Telecommunications,}\\
  {\small Nanjing 210023, China}\\
  {\small 3. School of Mathematical Sciences and Mathematical Institute, }\\
  {\small Nanjing Normal University, Nanjing 210023, China}}

\begin{document}

\date{}
\maketitle
\thispagestyle{empty}

\begin{abstract}

It has been known that if the initial data decay sufficiently fast at space infinity, then 1D Klein-Gordon equations
with quadratic nonlinearity admit classical solutions up to time $e^{C/\epsilon^2}$
while $e^{C/\epsilon^2}$ is also the upper bound of the lifespan,
where $C>0$ is some suitable constant and $\epsilon>0$ is the size of the initial data.
In this paper, we will focus on the 1D nonlinear Klein-Gordon equations with weakly decaying initial data.
It is shown that if the $H^s$-Sobolev norm with $(1+|x|)^{1/2+}$ weight of the initial data is small,
then the almost global solutions exist; if the initial $H^s$-Sobolev norm with $(1+|x|)^{1/2}$ weight
is small, then for any $M>0$, the solutions exist on $[0,\epsilon^{-M}]$.
Our proof is based on the dispersive estimate with a suitable $Z$-norm and a delicate analysis on the phase function.

\vskip 0.2 true cm


\vskip 0.2 true cm
\noindent
\end{abstract}

\vskip 0.5 true cm

\tableofcontents

\section{Introduction}

Consider the Cauchy problem of the following semilinear Klein-Gordon equation
\begin{equation}\label{KG-0}
\left\{
\begin{aligned}
&\Box u+u=F(u,\p u),\quad(t,x)\in[0,\infty)\times\R^d,\\
&(u,\p_tu)(0,x)=(u_0,u_1)(x),
\end{aligned}
\right.
\end{equation}
where $\Box=\p_t^2-\Delta$, $\Delta=\ds\sum_{j=1}^d\p_j^2$, $x=(x^1,\cdots,x^d)\in\R^d$, $d\ge1$,
$\p_0=\p_t$, $\p_j=\p_{x^j}$ for $j=1,\cdots,d$, $\p_x=(\p_1,\cdots,\p_n)$, $\p=(\p_0,\p_x)$,
$u$ is real valued, $(u_0,u_1)\in H^{s+1}(\Bbb R^d)\times H^s(\Bbb R^d)$ with $s>\f{d}{2}$
being suitably large numbers,  $\ve=\|u_0\|_{H^{s+1}(\Bbb R^d)}+\|u_1\|_{H^s(\Bbb R^d)}>0$ is sufficiently
small, and the smooth nonlinearity $F(u,\p u)$ is quadratic on $(u,\p u)$.

Under the assumption of null condition for $F(u,\p u)$, the authors in \cite{DF00}
prove that the solution $u\in C([0,T_\varepsilon), H^{s+1}(\Bbb R^d))\cap  C^1([0,T_\varepsilon), H^{s}(\Bbb R^d))$
of \eqref{KG-0} exists, where $T_\varepsilon\ge Ce^{C\varepsilon^{-\mu}}$ for $\mu=1$ if $d\ge3$, and $\mu=2/3$ if $d=2$.
In addition, for $d=1$, the lifespan $T_{\ve}\ge \frac{C}{\ve^4 |\ln\ve|^6}$  of \eqref{KG-0} is shown in \cite{Delort97}.
Recently, without the restriction of null condition for $F(u,\p u)$, the authors in \cite{HY23} have established
that the existence time of the solution $u\in C([0,T_\varepsilon), H^{s+1}(\Bbb R^d))
\cap  C^1([0,T_\varepsilon), H^{s}(\Bbb R^d))$ to \eqref{KG-0} can be improved to $T_\varepsilon=+\infty$ if $d\ge3$,
$T_\varepsilon\ge e^{C\varepsilon^{-2}}$ if $d=2$ and $T_\varepsilon\ge \frac{C}{\ve^4}$ if $d=1$.
Moreover, for $d=2$ and any fixed number $\beta>0$, if
\begin{equation}\label{initial:d=2}
\t\ve=\|u_0\|_{H^{N+1}(\R^2)}+\|u_1\|_{H^{N}(\R^2)}
+\|(1+|x|)^\beta u_0\|_{L^2(\R^2)}+\|(1+|x|)^\beta u_1\|_{L^2(\R^2)}
\end{equation}
is sufficiently small, where $N\ge12$, then it is proved in \cite{HY23}
that \eqref{KG-0} has a  global small classical solution $u\in C([0,\infty),H^{N+1}(\R^2))$
$\cap C^1([0,\infty),H^{N}(\R^2))$.
In the present paper, we are concerned with the 1D case of \eqref{KG-0}, that is,
\begin{equation}\label{KG}
\left\{
\begin{aligned}
&\p_t^2u-\p_x^2u+u=F(u,\p u),\quad(t,x)\in[0,\infty)\times\R,\\
&(u,\p_tu)(0,x)=(u_0,u_1)(x).
\end{aligned}
\right.
\end{equation}
Our main results can be stated as follows.

\begin{theorem}\label{thm1}
Let $N\ge27$ and $\alpha\in(0,1/2]$.
There are two positive constants $\ve_0$ and $\kappa_0$ such that if $(u_0,u_1)$ satisfies
\begin{equation}\label{initial:data}
\ve:=\|u_0\|_{H^{N+1}(\R)}+\|u_1\|_{H^N(\R)}
+\|(\Lambda u_0,u_1)\|_{Z_\alpha}\le\ve_0,
\end{equation}
where $\Lambda:=(1-\p_x^2)^{1/2}$ and $\|\cdot\|_{Z_\alpha}$ is defined by \eqref{Znorm:def} below, then \eqref{KG}
has a unique classical solution $u\in C([0,T_{\alpha,\ve}],H^{N+1}(\R))\cap C^1([0,T_{\alpha,\ve}],H^N(\R))$ with
\begin{equation}\label{lifespan}
T_{\alpha,\ve}=\left\{
\begin{aligned}
&e^{\kappa_0/\ve^2}-1,\qquad&&\alpha=1/2,\\
&\frac{\kappa_0}{\ve^{\frac{2}{1-2\alpha}}},&&\alpha\in(0,1/2).
\end{aligned}
\right.
\end{equation}
Moreover, there is a positive constant $C$ such that
\begin{equation}\label{thm1:disp}
\|(\Lambda u,\p_tu)(t,\cdot)\|_{L^\infty(\R)}\le C\ve(1+t)^{-\alpha}.
\end{equation}
\end{theorem}

\begin{corollary}\label{coro1}
Let $N\ge27$.
There are two positive constants $\eps_1$ and $\kappa_1$ such that for any $\beta>1/2$, if $(u_0,u_1)$ satisfies
\begin{equation*}
\eps:=\|u_0\|_{H^{N+1}(\R)}+\|u_1\|_{H^N(\R)}
+\|\w{x}^\beta\Lambda^{14}(\Lambda u_0,u_1)\|_{L^2(\R)}\le\eps_1,
\end{equation*}
where $\w{x}=\sqrt{1+x^2}$, then \eqref{KG} has a unique classical solution $u\in C([0,e^{\kappa_1/\eps^2}-1],H^{N+1}(\R))\cap C^1([0,e^{\kappa_1/\eps^2}-1],H^N(\R))$.
\end{corollary}

\begin{corollary}\label{coro2}
Let $N\ge27$.
For any $M>0$, there is $\eps_2>0$, such that if $(u_0,u_1)$ satisfies
\begin{equation*}
\eps:=\|u_0\|_{H^{N+1}(\R)}+\|u_1\|_{H^{N}(\R)}
+\|\w{x}^{1/2}\Lambda^{14}(\Lambda u_0,u_1)\|_{L^2(\R)}\le\eps_2,
\end{equation*}
then \eqref{KG} has a unique classical solution $u\in C([0,\eps^{-M}],H^{N+1}(\R))\cap C^1([0,\eps^{-M}],H^N(\R))$.
\end{corollary}

\begin{remark}
For the Cauchy problem
\begin{equation}\label{KG:blowup}
\left\{
\begin{aligned}
&\p_t^2u-\p_x^2u+u=(\p_tu)^2\p_xu,\\
&(u,\p_tu)(0,x)=\ve(\tilde u_0,\tilde u_1)(x),
\end{aligned}
\right.
\end{equation}
where $(\tilde u_0,\tilde u_1)\in C_0^\infty([-R,R])$,
\cite[Proposition 7.8.8]{Hormander97book} proved that
the lifespan $T_\ve\le R(e^{\frac{2}{\sigma\ve^2}}-1)$ holds if
$\sigma=\int_{\R}\tilde u_0'(x)\tilde u_1(x)dx>0$.
Note that problem \eqref{KG} contains the case \eqref{KG:blowup}, then
the upper bound $T_{1/2,\ve}=e^{\kappa_0/\ve^2}-1$ in Theorem \ref{thm1}
and $T_{\ve}=e^{\kappa_1/\eps^2}-1$ in Corollary \ref{coro1} are optimal.
\end{remark}

\begin{remark}
Although the lifespan $T_{\alpha,\ve}$ in Theorem \ref{thm1} may be not optimal
for $\alpha\in(0,1/2)$, it suffices to obtain Corollary \ref{coro2}.
\end{remark}

\begin{remark}
By the definition of $Z_{\alpha}$-norm in \eqref{Znorm:def} below, there exists some positive constant $C>0$
such that
\begin{equation}\label{Znorm:rough}
\text{$\|f\|_{Z_{1/2}}\le C\|(1+|x|)^{1/2+}\Lambda^{14}f\|_{L^2}$ and $\|f\|_{Z_{\alpha}}\le C\|(1+|x|)^{1/2}\Lambda^{14}f\|_{L^2}$
for $\alpha\in(0, 1/2)$.}
\end{equation}
One can see the details in the proofs for Corollaries \ref{coro1} and \ref{coro2} of $\S 6$.
\end{remark}

\begin{remark}
When the small data $(u_0,u_1)(x)$ decay sufficiently fast,
the analogous result to Corollary \ref{coro1} has been obtained for problem \eqref{KG} in  \cite{MTT97} by the vector field method.
It is pointed out that our Corollary \ref{coro1} only requires the smallness of
$H^s$-Sobolev norm with $\w{x}^{1/2+}$
weights of $(u_0,u_1)$, which leads to the failure of vector field method since
$\|x\p_x(u_0,u_1)\|_{L^2(\Bbb R)}$ can become infinite.
\end{remark}

\begin{remark}
Consider 1D quasilinear Klein-Gordon equation
\begin{equation}\label{KG-00}
\left\{
\begin{aligned}
&\p_t^2v-\p_x^2v+v=P(v,\p v,\p^2_{tx}v, \p^2_xv),\quad(t,x)\in[0,\infty)\times\R,\\
&(v,\p_tv)(0,x)=\delta (v_0,v_1)(x),
\end{aligned}
\right.
\end{equation}
where $\delta>0$ is small,  $P(v,\p v,\p^2_{tx}v, \p^2_xv)$ is smooth
on its arguments and linear with respect to $(\p^2_{tx}v, \p^2_xv)$, moreover, $P$ vanishes at least at order 2 at 0.
In \cite{Delort01}, under the null condition of $P(v,\p v,\p^2_{tx}v, \p^2_xv)$ and $(v_0,v_1)(x)\in C_0^{\infty}(\Bbb R)$,
the author shows that \eqref{KG-00} has a global small solution.
When $P(v,\p v,\p^2_{tx}v, \p^2_xv)$ is a homogeneous polynomial of degree 3 in $(v,\p v,\p^2_{tx}v, \p^2_xv)$,
affine in $(\p^2_{tx}v, \p^2_xv)$, if there exists an integer $s$ sufficiently large such that
\begin{equation}\label{YH-0}
\|v_0\|_{H^{s+1}(\Bbb R)}+\|v_1\|_{H^{s}(\Bbb R)}+\|xv_0\|_{H^2(\Bbb R)}+\|xv_1\|_{H^1(\Bbb R)}\le 1,
\end{equation}
it is proved in \cite{Stingo18} that \eqref{KG-00} admits a global small solution under the null condition
of $P(v,\p v,\p^2_{tx}v, \p^2_xv)$.
By \eqref{YH-0}, $(v_0,v_1)$ decays as $\w{x}^{-1}$ at infinity, which implies that the method of Klainerman vector fields
can be applied in \cite{Stingo18}.
\end{remark}

\begin{remark}
When $d\ge2$, it is well known that problem \eqref{KG-0} with rapidly decaying and small
initial data $(u_0, u_1)$ has a global
smooth solution, see \cite{Klainerman85,OTT96,Shatah85,ST93}.
\end{remark}

\begin{remark}
For 1D or 2D irrotational Euler-Poisson systems, when the $H^s$-Sobolev norms with $1+|x|$ weight of initial data are small,
the authors in \cite{GHZ17} or \cite{LW14} have proved the global existence of small solutions, respectively.
In this paper, we prove the almost global existence of problem \eqref{KG} with quadratic nonlinearity and small $H^s$-Sobolev norm
with lower order $\w{x}^{1/2+}$ weight. It is expected that 1D  or 2D irrotational Euler-Poisson systems
still have global solutions  when the corresponding initial data with the lower order weight $\w{x}^{1/2+}$ or $\w{x}^{0+}$ are small.
\end{remark}

We now give some comments and illustrations on the proof of Theorem \ref{thm1}.
Note that the vector field method in \cite{Klainerman85,MTT97,OTT96} will produce quite high order $\langle x\rangle$ weight
in the resulting Sobolev norm of the initial data, which is not suitable for the proof of Theorem \ref{thm1} with the
initial data of lower order $\w{x}^{1/2+}$ weight. Motivated by
the Fourier analysis methods as in \cite{GHZ17,IP13,LW14, Shatah85}, at first,
we will transform the quadratic nonlinearity of \eqref{KG} into the cubic nonlinearity.
For this end, we set
\begin{equation*}
U:=(\p_t+i\Lambda)u.
\end{equation*}
Then \eqref{KG} can be reduced to the following half Klein-Gordon equation
\begin{equation}\label{halfKG}
(\p_t-i\Lambda)U=\cN(U),
\end{equation}
where $\cN(U)$ is at least quadratic in $U$.
Denote the profile
\begin{equation}\label{intro:profile:def}
V:=V_+=e^{-it\Lambda}U,\quad V_-:=\overline{V}.
\end{equation}
Applying Fourier transformation to \eqref{halfKG} yields
\begin{equation}\label{intro:profile1}
\hat V(t,\xi)=\hat V(0,\xi)+\sum_{\mu_1,\mu_2=\pm}\int_0^t\int_{\xi_1+\xi_2=\xi}
e^{is\Phi_{\mu_1\mu_2}}m_2(\xi_1,\xi_2)\hat V_{\mu_1}(s,\xi_1)
\hat V_{\mu_2}(s,\xi_2)d\xi_1ds+\mathrm{other~terms},
\end{equation}
where $\hat V(t,\xi)=(\sF_xV(t,x))(t,\xi)$, $m_2(\xi_1,\xi_2)$ is some Fourier multiplier and
\begin{equation*}
\Phi_{\mu_1\mu_2}=\Phi_{\mu_1\mu_2}(\xi_1,\xi_2):=-\Lambda(\xi_1+\xi_2)+\mu_1\Lambda(\xi_1)
+\mu_2\Lambda(\xi_2),\quad\Lambda(\xi)=\sqrt{1+\xi^2},\quad\xi\in\R.
\end{equation*}
Note that $\Phi_{\mu_1\mu_2}\neq0$ for equation \eqref{KG}. Then one can integrate
by parts in time $s$ in \eqref{intro:profile1} and utilize \eqref{halfKG} to obtain
\begin{equation}\label{intro:profile2}
\begin{split}
\hat V(t,\xi)=\hat V(0,\xi)+\sum_{\substack{(\mu_1,\mu_2,\mu_3)\in\{(+++),\\
(++-),(+--),(---)\}}}\int_0^t\iint_{\xi_1+\xi_2+\xi_3=\xi}
e^{is\Phi_{\mu_1\mu_2\mu_3}}m_3(\xi_1,\xi_2,\xi_3)\hat V_{\mu_1}(s,\xi_1)\\
\times\hat V_{\mu_2}(s,\xi_2)\hat V_{\mu_3}(s,\xi_3)d\xi_1d\xi_2ds+\mathrm{other~terms},
\end{split}
\end{equation}
where $m_3(\xi_1,\xi_2,\xi_3)$ is the resulting Fourier multiplier and
\begin{equation}\label{intro:3phase}
\Phi_{\mu_1\mu_2\mu_3}(\xi_1,\xi_2,\xi_3):=-\Lambda(\xi_1+\xi_2+\xi_3)+\mu_1\Lambda(\xi_1)
+\mu_2\Lambda(\xi_2)+\mu_3\Lambda(\xi_3).
\end{equation}
Through the normal form transformation (see details in Section \ref{1st:nf}), one can simply
consider problem \eqref{KG} with the cubic nonlinearity.
Based on this, applying the standard energy method, one can obtain that there are some
positive constants $C$ and $N'$ such that
\begin{equation}\label{intro:energy}
\frac{d}{dt}\|U(t)\|_{H^N(\R)}\le C\|U(t)\|^2_{W^{N',\infty}(\R)}\|U(t)\|_{H^N(\R)}.
\end{equation}
To derive the sufficient time-decay of $\|U(t)\|_{W^{N',\infty}}$,
we firstly consider the following corresponding linear problem of \eqref{KG}
\begin{equation}\label{KG:ln}
\left\{
\begin{aligned}
&\p_t^2u_{lin}-\p_x^2u_{lin}+u_{lin}=0,\quad(t,x)\in[0,\infty)\times\R,\\
&(u_{lin},\p_tu_{lin})(0,x)=(u_0,u_1)(x).
\end{aligned}
\right.
\end{equation}
The solution to \eqref{KG:ln} can be expressed as
\begin{equation}\label{KG:ln:solu}
u_{lin}(t)=\frac{(e^{it\Lambda}+e^{-it\Lambda})u_0}{2}
+\frac{(e^{it\Lambda}-e^{-it\Lambda})\Lambda^{-1}u_1}{2i}.
\end{equation}
Note that by the standard dispersive estimate  of $e^{\pm it\Lambda}$ (see \eqref{disp:estimate} below), one has
\begin{equation}\label{stan:disp}
\|e^{\pm it\Lambda}f\|_{L^\infty(\R)}\le C(1+t)^{-1/2}\|\Lambda^{3/2+}f\|_{L^1(\R)}.
\end{equation}
Under the weakly decaying initial data of Theorem \ref{thm1},
it is necessary to employ the $Z_{\alpha}$-norm
instead of the $L^1(\R)$ norm on the right hand side of \eqref{stan:disp},
which has the form
\begin{equation}\label{ln:solu:decay}
\|u_{lin}(t)\|_{W^{N',\infty}(\R)}\le C(1+t)^{-\alpha}\|(u_0, \Lambda^{-1}u_1)\|_{Z_\alpha},
\quad\alpha\in(0,1/2].
\end{equation}
Similarly, for the solution $u(t)$ to the nonlinear problem \eqref{KG}, we can arrive at
\begin{equation}\label{solu:decay}
\|U(t)\|_{W^{N',\infty}(\R)}\le C(1+t)^{-\alpha}\|V(t)\|_{Z_\alpha},
\quad\alpha\in(0,1/2],
\end{equation}
where $V$ is defined in \eqref{intro:profile:def}.
The remaining task is to control $\|V(t)\|_{Z_\alpha}\le C\ve$.
Inspired by \cite{IP13,LW14}, we will give a precise analysis on the related cubic nonlinearity and perform a suitable
normal form transformation once again. Note that for $(\mu_1,\mu_2,\mu_3)\in\{(+++),(+--),(---)\}$,
the phase $\Phi_{\mu_1\mu_2\mu_3}$ does not vanish and the cubic nonlinearity can be further transformed into
a quartic one. Then for the bad cubic nonlinearity $\hat V_+(s,\xi_1)\hat V_+(s,\xi_2)\hat V_-(s,\xi_3)$,
the corresponding phase in \eqref{intro:profile2} is
\begin{equation}\label{bad:phase}
\begin{split}
&\Phi_{bad}(\xi,\eta,\zeta)=\Phi_{++-}(\xi_1,\xi_2,\xi_3)
=-\Lambda(\xi)+\Lambda(\xi-\eta)+\Lambda(\eta-\zeta)-\Lambda(\zeta),\\
&\xi_1=\xi-\eta,\qquad\xi_2=\eta-\zeta,\qquad\xi_3=\zeta.
\end{split}
\end{equation}
To handle the situation of bad phase, we write \eqref{intro:profile2} in the physical space as
\begin{equation}\label{intro:profile3}
\begin{split}
V(t,x)=V(0,x)+\frac{1}{(2\pi)^{3}}\int_0^t\iiint_{\R^3}K_{bad}(x-x_1,x-x_2,x-x_3)V_+(s,x_1)
V_+(s,x_2)\\
\times V_-(s,x_3)dx_1dx_2dx_3ds+\mathrm{other~terms},
\end{split}
\end{equation}
where the Schwartz kernel $K_{bad}$ is given by
\begin{equation}\label{bad:kernel}
\begin{split}
&K_{bad}(x-x_1,x-x_2,x-x_3)=\iiint_{\R^3}e^{i\Psi_{bad}}\times\{\mathrm{other~terms}\}d\xi d\eta d\zeta,\\
&\Psi_{bad}=s\Phi_{bad}(\xi,\eta,\zeta)+\xi(x-x_1)+\eta(x_1-x_2)+\zeta(x_2-x_3).
\end{split}
\end{equation}
Therefore, in order to estimate $\|V(t)\|_{Z_\alpha}$, the key points  are  to analyze the phase $\Psi_{bad}$
and further to treat the Schwartz kernel $K_{bad}$.
For this purpose, according to the relations of $\xi_1=\xi-\eta$, $\xi_2=\eta-\zeta$ and $\xi_3=\zeta$,
the following cases are distinguished:
\begin{equation}\label{case:intro}
\begin{array}{cccc}
  &\xi-\eta & \eta-\zeta & \zeta \\
  {\rm case~(LLH)} & low & low & high \\
  {\rm case~(HLL)} & high & low & low \\
  {\rm case~(LHL)} & low & high & low \\
  {\rm case~(HLH)} & high & low & high\\
  {\rm case~(Oth)} & & other~cases &
\end{array}
\end{equation}

In the case~(LLH), one has $|\xi-\eta|,|\eta-\zeta|\ll|\zeta|$ and $\Phi_{bad}\neq0$.
Then the related cubic nonlinearity can be transformed into the quartic one.

For the cases of~(HLL), (LHL), (HLH) and (Oth), it is required to precisely compute
the critical points of $\Psi_{bad}$.
However, this is a hard task since $\p_{\xi,\zeta}\Psi_{bad}$ depends on the space-time locations
as well as the frequencies:
\begin{equation}\label{grad:bad:phase}
\begin{split}
\p_\xi\Psi_{bad}&=x-x_1+s(\Lambda'(\xi-\eta)-\Lambda'(\xi))
=x-x_1-s\eta\Lambda''(\xi-r_1\eta),\\
\p_\zeta\Psi_{bad}&=x_2-x_3+s(\Lambda'(\zeta-\eta)-\Lambda'(\zeta))
=x_2-x_3-s\eta\Lambda''(\zeta-r_2\eta),
\end{split}
\end{equation}
where $r_1,r_2\in[0,1]$ and $\Lambda''(y)=(1+y^2)^{-3/2}$ with $y\in\R$.
On the other hand, in order to analyze the critical points of $\Psi_{bad}$ in \eqref{grad:bad:phase},
the Littlewood-Paley decompositions both in the physical and frequency spaces are applied,
which leads to the introduction of the related $Z_{\alpha}$-norm.
Note that by a careful discussion on the relations between $s\eta$ and other factors in \eqref{grad:bad:phase},
a suitable classification will be taken in terms of the relative size of the space-time locations and the frequencies.
Roughly speaking, the classification includes: near the possible critical points and away from the critical points
of $\Psi_{bad}$. Near the possible critical points, the $Z_{\alpha}$-norm estimate of the cubic nonlinearity can be
treated by the dispersive estimate \eqref{solu:decay} with a bootstrap assumption on $\|V(t)\|_{Z_\alpha}$.
Away from the critical points, the stationary phase method is performed.
Nevertheless, many  involved  and technical computations are needed.
For examples, in the case~(HLL) with $|\eta|\ll|\xi|$, by the observation $\Lambda''(\xi-r_1\eta)\approx(1+|\xi|)^{-3}$,
the $L_x^\infty$ norm of some related high frequency term can be obtained;
in the case~(HLH), due to the different distances from the zero points of $\p_\xi\Psi_{bad}$,
three cases including the high-frequency, intermediate-frequency and low-frequency in
the kernel of $K_{bad}$ are separately treated: with respect to  the parts of the high-frequency and low-frequency,
since the corresponding frequencies are away from the zero points of $\p_\xi\Psi_{bad}$, the stationary phase argument
with respect to the $\xi$ variable can be implemented. For the part of intermediate frequency,
the zero points of $\p_\xi\Psi_{bad}$ and $\p_\zeta\Psi_{bad}$ will be considered simultaneously
so that the space-decay rate of $K_{bad}$ can be obtained. Next we explain why
some technical analysis on the related phase $\Psi_{bad}$ in the 2D case of \cite{IP13}
is difficult to be utilized directly by us. For the  2D case, such a faster time-decay estimate
than \eqref{solu:decay} in 1D case is obtained
\begin{equation}\label{solu:decay:2d}
\|U(t)\|_{W^{N',\infty}(\R^2)}\le C(1+t)^{-1}\|V(t)\|_{Z_1}.
\end{equation}
Due to \eqref{Znorm:rough}, the estimate of $\|V(t)\|_{Z_1}$ in \eqref{solu:decay:2d} roughly comes down to
that of $\|(1+|x|)^{1+}\Lambda^{\upsilon} V\|_{L^2(\R^2)}$ for some suitable number $\upsilon >0$.
To this end, two kinds of regions for $|x|\ge s^\theta$ and $|x|\le s^\theta$ with $\theta\in(0,1)$
are divided, respectively. For $|x|\le s^\theta$, the authors in \cite{IP13} obtain that for $\theta\in(0,1)$,
\begin{equation}\label{IP13:InCone}
\begin{split}
\|(1+|x|)^{1+}\Lambda^{\upsilon} V(t)\|_{L^2(|x|\le s^\theta)}&\le C\int_0^t(1+s)^{\theta^{+}}
\|U(s)\|^2_{W^{N',\infty}(\R^2)}\|U(s)\|_{H^N}ds+\mathrm{other~terms}\\
&\le C\ve^3\int_0^t(1+s)^{-2+\theta^{+}}ds+\mathrm{other~terms}\\
&\le C\ve^3+\mathrm{other~terms},
\end{split}
\end{equation}
which yields the smallness estimate of  $\|V(t)\|_{Z_1}$ when $|x|\le s^\theta$.
However, in our problem \eqref{KG}, if taking the case of $\alpha=1/2$ as an instance,
by $\|U(t)\|_{W^{N',\infty}(\R)}\le C(1+t)^{-1/2}\|V(t)\|_{Z_{1/2}}$
and $\|V(t)\|_{Z_{1/2}}\le C\|(1+|x|)^{1/2+}\Lambda^{\upsilon} V\|_{L^2(\R)}$, then similarly to \eqref{IP13:InCone}, one has
that for $\theta>0$,
\begin{equation}
\begin{split}
\|(1+|x|)^{1/2+}\Lambda^{\upsilon} V(T_{1/2,\ve})\|_{L^2(|x|\le s^\theta)}&\le C\ve^3\int_0^{T_{1/2,\ve}}(1+s)^{\theta^{+}/2-1}ds+\mathrm{other~terms}\\
&\le C\ve^3(1+T_{1/2,\ve})^{\theta^{+}/2}+\mathrm{other~terms}.
\end{split}
\end{equation}
This means that $T_{1/2,\ve}\le \ve^{-\f{4}{\theta^{+}}}$ holds in order to guarantee the smallness of
$\|V(t)\|_{Z_1}$, which is too crude by comparison with $T_{1/2,\ve}\sim e^{\kappa_0/\ve^2}$ in \eqref{lifespan} of Theorem \ref{thm1}.
This is the reason that we have to give more delicate analysis on the related phase $\Psi_{bad}$  in \eqref{bad:kernel}.

Based on all the above analysis, the estimate of the $Z_{\alpha}$-norm of the
cubic nonlinearity in \eqref{intro:profile3} will be finished.
On the other hand, the treatments for the quartic nonlinearity and other terms in \eqref{intro:profile3}
are much easier. Finally, the bootstrap assumption of
$\|V(t)\|_{Z_\alpha}$ can be closed and then Theorem \ref{thm1} is proved.

\vskip 0.3cm

The paper is organized as follows.
In Section 2, some preliminaries such as the Littlewood-Paley decomposition, the definition of $Z_{\alpha}$-norm,
the linear dispersive estimate and two useful lemmas are illustrated.
By the normal form transformations, a reformulation of \eqref{KG} will be derived in Section 3.
In Section 4, some energy estimates and the continuity of the $Z_{\alpha}$-norm are established.
In Section 5, the related $Z_{\alpha}$-norm is estimated.
In Section 6, we complete the proofs of Theorem \ref{thm1} and Corollaries \ref{coro1}-\ref{coro2}.
In addition, the estimates on some resulting multilinear Fourier multipliers are given in Appendix.

\section{Preliminaries}

\subsection{Littlewood-Paley decomposition and definition of $Z_{\alpha}$-norm}

For the integral function $f(x)$ on $\R$, its Fourier transformation is defined as
\begin{equation*}
\hat f(\xi):=\sF_xf(\xi)=\int_{\R}e^{-ix\xi}f(x)dx.
\end{equation*}
Choosing a smooth cut-off function $\psi: \R\rightarrow[0,1]$, which equals 1 on $[-5/4,5/4]$ and vanishes outside $[-8/5,8/5]$,
we set
\begin{equation*}
\begin{split}
&\psi_k(x):=\psi(|x|/2^k)-\psi(|x|/2^{k-1}),\quad k\in\Z,k\ge0,\\
&\psi_{-1}(x):=1-\sum_{k\ge0}\psi_k(x)=\psi(2|x|),
\quad\psi_I:=\sum_{k\in I\cap\Z\cap[-1,\infty)}\psi_k,
\end{split}
\end{equation*}
where $I$ is any interval of $\R$.
Let $P_k$ be the Littlewood-Paley projection onto frequency $2^k$
\begin{equation*}
\sF(P_kf)(\xi):=\psi_k(\xi)\sF f(\xi),\quad k\in\Z,k\ge-1.
\end{equation*}
For any interval $I$, $P_I$ is defined by
\begin{equation*}
P_If:=\sum_{k\in I\cap\Z\cap[-1,\infty)}P_kf.
\end{equation*}
Introducing the following dyadic decomposition in the Euclidean physical space $\R$
\begin{equation*}
(Q_jf)(x):=\psi_j(x)f(x),\qquad j\in\Z,j\ge-1.
\end{equation*}
Inspired by \cite{IP13}, we define the $Z_{\alpha}$-norm of $f$ as
\begin{equation}\label{Znorm:def}
\|f\|_{Z_\alpha}:=\sum_{j,k\ge-1}2^{j\alpha+N_1k}\|Q_jP_kf\|_{L^2(\R)},
\qquad\alpha\in(0,1/2],\ N_1=12.
\end{equation}
Let
\begin{equation*}
Z_\alpha:=\{f\in L^2(\R):\|f\|_{Z_\alpha}<\infty\}
\end{equation*}
and
$\|(g,h)\|_{Z_\alpha}:=\|g\|_{Z_\alpha}+\|h\|_{Z_\alpha}$.

Through the whole paper, for non-negative quantities $f$ and $g$,
$f\ls g$ and $f\gt g$ mean $f\le Cg$ and $f\ge Cg$ with $C>0$ being a generic constant.

\subsection{Linear dispersive estimate}

\begin{lemma}[Linear dispersive estimate]\label{lem:disp}
For any function $f$, integer $k\ge-1$ and $t\ge0$, it holds that
\begin{equation}\label{disp:estimate}
\|P_ke^{\pm it\Lambda}f\|_{L^\infty(\R)}\ls2^{3k/2}(1+t)^{-1/2}\|P_kf\|_{L^1(\R)}.
\end{equation}
Moreover, for $\beta\in[0,1/2]$ and $j\ge-1$, one has
\begin{equation}\label{loc:disp}
\|P_ke^{\pm it\Lambda}Q_jf\|_{L^\infty(\R)}
\ls2^{k/2+2k\beta+j\beta}(1+t)^{-\beta}\|Q_jf\|_{L^2(\R)}.
\end{equation}
\end{lemma}
\begin{proof}
Note that
\begin{equation}\label{proj:proj}
\psi_k(x)=\psi_k(x)\psi_{[[k]]}(x),
\end{equation}
where $[[k]]:=[k-1,k+1]$.
Then one has
\begin{equation}\label{disp:estimate1}
\begin{split}
P_ke^{it\Lambda}f(x)&=(2\pi)^{-1}\int_{\R}\cK_k(t,x-y)P_kf(y)dy,\\
\cK_k(t,x)&:=\int_{\R}e^{i(x\xi+t\w{\xi})}\psi_{[[k]]}(\xi)d\xi.
\end{split}
\end{equation}
According to Corollary 2.36 and 2.38 in \cite{NS11book}, for any $t\ge1$, it holds that
\begin{equation}\label{disp1}
\|\cK_k(t,x)\|_{L^\infty(\R)}\ls2^{3k/2}t^{-1/2}.
\end{equation}
For $0\leq t\le1$, we easily have
\begin{equation*}
\|\cK_k(t,x)\|_{L^\infty(\R)}\ls\int_{\R}\psi_{[[k]]}(\xi)d\xi\ls2^k.
\end{equation*}
This, together with \eqref{disp:estimate1}, \eqref{disp1} and Young's inequality, leads to
\begin{equation*}
\|P_ke^{it\Lambda}f\|_{L^\infty(\R)}\ls\|\cK_k\|_{L^\infty(\R)}\|P_kf\|_{L^1(\R)}
\ls2^{3k/2}(1+t)^{-1/2}\|P_kf\|_{L^1(\R)}.
\end{equation*}
In addition, the estimate of $\|P_ke^{-it\Lambda}f\|_{L^\infty(\R)}$ is analogous.
Thus, \eqref{disp:estimate} is achieved.

Next we turn to the proof of \eqref{loc:disp}.
It follows from the Bernstein inequality such as \cite[Lemma 2.1]{BCD2011}
and the unitarity of $e^{\pm it\Lambda}$ that
\begin{equation*}
\|P_ke^{\pm it\Lambda}Q_jf\|_{L^\infty(\R)}
\ls2^{k/2}\|P_ke^{\pm it\Lambda}Q_jf\|_{L^2(\R)}\ls2^{k/2}\|Q_jf\|_{L^2(\R)}.
\end{equation*}
On the other hand, \eqref{disp:estimate} implies
\begin{equation*}
\|P_ke^{\pm it\Lambda}Q_jf\|_{L^\infty(\R)}
\ls2^{3k/2}(1+t)^{-1/2}\|Q_jf\|_{L^1(\R)}
\ls2^{3k/2+j/2}(1+t)^{-1/2}\|Q_jf\|_{L^2(\R)}.
\end{equation*}
Therefore,
\begin{equation*}
\begin{split}
\|P_ke^{\pm it\Lambda}Q_jf\|_{L^\infty(\R)}
&=(\|P_ke^{\pm it\Lambda}Q_jf\|_{L^\infty(\R)})^{1-2\beta}
(\|P_ke^{\pm it\Lambda}Q_jf\|_{L^\infty(\R)})^{2\beta}\\
&\ls(2^{k/2})^{1-2\beta}(2^{3k/2+j/2}(1+t)^{-1/2})^{2\beta}\|Q_jf\|_{L^2(\R)}\\
&\ls2^{k/2+2k\beta+j\beta}(1+t)^{-\beta}\|Q_jf\|_{L^2(\R)}.
\end{split}
\end{equation*}
\end{proof}

\begin{lemma}
For any function $f$, integer $k\ge-1$, $t\ge0$ and $p\in[2,+\infty]$, it holds that
\begin{equation}\label{disp:Lp}
\|P_ke^{\pm it\Lambda}Q_jf\|_{L^p(\R)}
\ls\Big(\frac{2^{3k+j}}{1+t}\Big)^{1/2-1/p}\|Q_jf\|_{L^2(\R)}.
\end{equation}
\end{lemma}
\begin{proof}
Note that
\begin{equation}\label{disp:Lp1}
\|P_ke^{\pm it\Lambda}f\|_{L^2(\R)}=\|P_kf\|_{L^2(\R)}\ls\|f\|_{L^2(\R)}.
\end{equation}
Applying the Riesz-Thorin interpolation theorem to \eqref{disp:estimate} and \eqref{disp:Lp1} yields
\begin{equation*}
\|P_ke^{\pm it\Lambda}f\|_{L^p(\R)}\ls\Big(\frac{2^{3k/2}}{\sqrt{1+t}}\Big)^{1-2/p}\|f\|_{L^{p'}(\R)},
\end{equation*}
where $\frac{1}{p'}=1-\frac1p$.
Therefore, we achieve from \eqref{proj:proj} that
\begin{equation*}
\begin{split}
\|P_ke^{\pm it\Lambda}Q_jf\|_{L^p(\R)}
&\ls\Big(\frac{2^{3k}}{1+t}\Big)^{1/2-1/p}\|Q_jf\|_{L^{p'}(\R)}\\
&\ls\Big(\frac{2^{3k}}{1+t}\Big)^{1/2-1/p}\|\psi_{[[j]]}Q_jf\|_{L^{p'}(\R)}\\
&\ls\Big(\frac{2^{3k}}{1+t}\Big)^{1/2-1/p}\|\psi_{[[j]]}\|_{L^{2p/(p-2)}(\R)}
\|Q_jf\|_{L^2(\R)}\\
&\ls\Big(\frac{2^{3k}}{1+t}\Big)^{1/2-1/p}2^{j(1/2-1/p)}\|Q_jf\|_{L^2(\R)},
\end{split}
\end{equation*}
which derives \eqref{disp:Lp}.
\end{proof}

\subsection{Two technical Lemmas}

\begin{lemma}
For $\mu_1,\mu_2,\mu_3=\pm$, define
\begin{equation}\label{phase:def}
\begin{split}
\Phi_{\mu_1\mu_2}(\xi_1,\xi_2)&:=-\Lambda(\xi_1+\xi_2)+\mu_1\Lambda(\xi_1)+\mu_2\Lambda(\xi_2),\\
\Phi_{\mu_1\mu_2\mu_3}(\xi_1,\xi_2,\xi_3)&:=-\Lambda(\xi_1+\xi_2+\xi_3))+\mu_1\Lambda(\xi_1)
+\mu_2\Lambda(\xi_2)+\mu_3\Lambda(\xi_3).
\end{split}
\end{equation}
For $\mu_1,\mu_2=\pm$ and $l\ge1$, one has
\begin{equation}\label{2phase:bdd1}
|\Phi^{-1}_{\mu_1\mu_2}(\xi_1,\xi_2)|\ls1+\min\{|\xi_1+\xi_2|,|\xi_1|,|\xi_2|\},
|\p_{\xi_1,\xi_2}^l\Phi_{\mu_1\mu_2}(\xi_1,\xi_2)|\ls\min\{1,|\Phi_{\mu_1\mu_2}(\xi_1,\xi_2)|\}
\end{equation}
and
\begin{equation}\label{2phase:bdd2}
|\p_{\xi_1,\xi_2}^l\Phi^{-1}_{\mu_1\mu_2}(\xi_1,\xi_2)|\ls|\Phi^{-1}_{\mu_1\mu_2}(\xi_1,\xi_2)|.
\end{equation}
For $(\mu_1,\mu_2,\mu_3)\in A_\Phi^{good}:=\{(+++),(+--),(---)\}$, one has
\begin{equation}\label{3phase:bdd}
|\Phi^{-1}_{\mu_1\mu_2\mu_3}(\xi_1,\xi_2,\xi_3)|\ls1+\min\{|\xi_1+\xi_2+\xi_3|,|\xi_1|,|\xi_2|,|\xi_3|\}.
\end{equation}
\end{lemma}
\begin{proof}
The proof of \eqref{2phase:bdd1} can be found in Lemma 5.1 of \cite{IP13}. Meanwhile,
\eqref{2phase:bdd2} is a consequence of \eqref{2phase:bdd1}.
For inequality \eqref{3phase:bdd}, see (4.47) in \cite{IP13}.
Note that although all these related inequalities in \cite{IP13} are derived for $\xi_1,\xi_2,\xi_3\in\R^2$,
it is easy to check that these inequalities still hold for $\xi_1,\xi_2,\xi_3\in\R$.
\end{proof}

\begin{lemma}[H\"{o}lder inequality]\label{lem:Holder}
For any functions $f_1,f_2,f_3,f_4$ on $\R$ and $p,q_1,q_2,q_3,q_4\in[1,\infty]$, one has
\begin{equation}\label{Holder}
\begin{split}
&\Big\|\iint_{\R^2}K(x-x_1,x-x_2)f_1(x_1)f_2(x_2)dx_1dx_2\Big\|_{L^p_x(\R)}\\
\le&\|K(\cdot,\cdot)\|_{L^1(\R^2)}\|f_1\|_{L^{q_1}}\|f_2\|_{L^{q_2}},
\qquad\frac{1}{p}=\frac{1}{q_1}+\frac{1}{q_2},\\
&\Big\|\iiint_{\R^3}K(x-x_1,x-x_2,x-x_3)f_1(x_1)f_2(x_2)f_3(x_3)dx_1dx_2dx_3\Big\|_{L^p_x(\R)}\\
\le&\|K(\cdot,\cdot,\cdot)\|_{L^1(\R^3)}\|f_1\|_{L^{q_1}}\|f_2\|_{L^{q_2}}
\|f_3\|_{L^{q_3}},
\qquad\frac{1}{p}=\frac{1}{q_1}+\frac{1}{q_2}+\frac{1}{q_3},\\
&\Big\|\iiiint_{\R^4}K(x-x_1,x-x_2,x-x_3,x-x_4)f_1(x_1)f_2(x_2)f_3(x_3)f_4(x_4)
dx_1dx_2dx_3dx_4\Big\|_{L^p_x(\R)}\\
\le&\|K(\cdot,\cdot,\cdot,\cdot)\|_{L^1(\R^4)}\|f_1\|_{L^{q_1}}\|f_2\|_{L^{q_2}}
\|f_3\|_{L^{q_3}}\|f_4\|_{L^{q_4}},
\qquad\frac{1}{p}=\frac{1}{q_1}+\frac{1}{q_2}+\frac{1}{q_3}+\frac{1}{q_4}.
\end{split}
\end{equation}
\end{lemma}
\begin{proof}
\eqref{Holder} can be directly derived from the Minkowski inequality and the H\"{o}lder inequality,
or see Lemma 2.3 in \cite{OTT96}.
\end{proof}

Denote
\begin{equation}\label{proj:decom}
\begin{split}
\cX_k&=\cX_k^1\cup\cX_k^2,\qquad \cY_k=\cY_k^1\cup\cY_k^2,\\
\cX_k^1&=\{(k_1,k_2)\in\Z^2: k_1,k_2\ge-1,|\max\{k_1,k_2\}-k|\le8\},\\
\cX_k^2&=\{(k_1,k_2)\in\Z^2: k_1,k_2\ge-1,\max\{k_1,k_2\}\ge k+8,|k_1-k_2|\le8\},\\
\cY_k^1&=\{(k_1,k_2,k_3)\in\Z^3: k_1,k_2,k_3\ge-1,|\max\{k_1,k_2,k_3\}-k|\le4\},\\
\cY_k^2&=\{(k_1,k_2,k_3)\in\Z^3: k_1,k_2,k_3\ge-1,k+4\le\max\{k_1,k_2,k_3\}
\le\med\{k_1,k_2,k_3\}+4\}.
\end{split}
\end{equation}
As in \cite[page 784,799]{IP13}, if $P_k(P_{k_1}fP_{k_2}g)\neq0$ and $P_k(P_{k_1}fP_{k_2}gP_{k_3}h)\neq0$,
one then has $(k_1,k_2)\in\cX_k$ and $(k_1,k_2,k_3)\in\cY_k$, respectively.

\section{Reduction}

\subsection{First normal form transformation}\label{1st:nf}
Based on \eqref{2phase:bdd1}, we are devoted to transforming the quadratic nonlinearity
in \eqref{KG} into the cubic one. Denote
\begin{equation}\label{profile:def}
U_\pm:=(\p_t\pm i\Lambda)u,\quad U:=U_+.
\end{equation}
For functions $m_2(\xi_1,\xi_2):\R^2\rightarrow\C$ and $m_3(\xi_1,\xi_2,\xi_3):\R^3\rightarrow\C$,
define the following multi-linear pseudoproduct operators:
\begin{equation}\label{m-linear:def}
\begin{split}
T_{m_2}(f,g)&:=\sF_\xi^{-1}\Big((2\pi)^{-2}\int_{\R}
m_2(\xi-\eta,\eta)\hat f(\xi-\eta)\hat g(\eta)d\eta\Big),\\
T_{m_3}(f,g,h)&:=\sF_\xi^{-1}\Big((2\pi)^{-3}\iint_{\R^2}
m_3(\xi-\eta,\eta-\zeta,\zeta)\hat f(\xi-\eta)\hat g(\eta-\zeta)\hat h(\zeta)d\eta d\zeta\Big).
\end{split}
\end{equation}
Then \eqref{KG} is reduced to
\begin{equation}\label{profile:eqn1}
(\p_t-i\Lambda)U=\cN(U),\qquad\p_tV(t,x)=e^{-it\Lambda}\cN(U),
\end{equation}
where $V=V_+$ and $V_-$ are defined in \eqref{intro:profile:def}, $\cN(U)$ is given by
\begin{equation}\label{nonlinear}
\cN(U):=\sum_{\mu_1,\mu_2=\pm}T_{a_{\mu_1\mu_2}}(U_{\mu_1},U_{\mu_2})
+\sum_{\mu_1,\mu_2,\mu_3=\pm}T_{b_{\mu_1\mu_2\mu_3}}(U_{\mu_1},U_{\mu_2},U_{\mu_3})
+\cN_4(U),
\end{equation}
here $a_{\mu_1\mu_2}=a_{\mu_1\mu_2}(\xi_1,\xi_2)$ is a linear combination of the products of
the following terms
\begin{equation}\label{symbol:a}
1,\frac{1}{\Lambda(\xi_1)},\frac{1}{\Lambda(\xi_2)},\frac{\xi_1}{\Lambda(\xi_1)},
\frac{\xi_2}{\Lambda(\xi_2)},
\end{equation}
$b_{\mu_1\mu_2\mu_3}=b_{\mu_1\mu_2\mu_3}(\xi_1,\xi_2,\xi_3)$ is a linear combination of the products of
\begin{equation}\label{symbol:b}
1,\frac{1}{\Lambda(\xi_1)},\frac{1}{\Lambda(\xi_2)},\frac{1}{\Lambda(\xi_3)},
\frac{\xi_1}{\Lambda(\xi_1)},\frac{\xi_2}{\Lambda(\xi_2)},\frac{\xi_3}{\Lambda(\xi_3)},
\end{equation}
and the nonlinearity $\cN_4(U)$ is at least quartic in $U$.

Applying the Fourier transformation to \eqref{profile:eqn1} and solving the resulting equation yield
\begin{equation}\label{profile:eqn3}
\begin{split}
\hat V(t,\xi)
&=\hat V(0,\xi)+\int_0^te^{-is\Lambda(\xi)}\widehat{\cN_4(U)}(s,\xi)ds\\
&\quad+\sum_{\mu_1,\mu_2=\pm}\int_0^t\int_{\R}e^{is\Phi_{\mu_1\mu_2}}
a_{\mu_1\mu_2}\hat V_{\mu_1}(s,\xi-\eta)\hat V_{\mu_2}(s,\eta)d\eta ds,\\
&\quad+\sum_{\mu_1,\mu_2,\mu_3=\pm}\int_0^t\iint_{\R^2}
e^{is\Phi_{\mu_1\mu_2\mu_3}}b_{\mu_1\mu_2\mu_3}\hat V_{\mu_1}(s,\xi-\eta)
\hat V_{\mu_2}(s,\eta-\zeta)\hat V_{\mu_3}(s,\zeta)d\eta d\zeta ds,
\end{split}
\end{equation}
where $\Phi_{\mu_1\mu_2}$ and $\Phi_{\mu_1\mu_2\mu_3}$ are defined by \eqref{phase:def}.

Thanks to \eqref{2phase:bdd1}, through integrating by parts in $s$ for the second line of \eqref{profile:eqn3}, we arrive at
\begin{equation*}
\begin{split}
\hat V(t,\xi)
&=\hat V(0,\xi)+\int_0^te^{-is\Lambda(\xi)}\widehat{\cN_4(U)}(s,\xi)ds\\
&\quad-i\sum_{\mu_1,\mu_2=\pm}\sF(e^{-is\Lambda}
T_{\Phi_{\mu_1\mu_2}^{-1}a_{\mu_1\mu_2}}(U_{\mu_1},U_{\mu_2}))(s,\xi)\Big|_{s=0}^t\\
&\quad+i\sum_{\mu_1,\mu_2=\pm}\int_0^t\int_{\R}e^{is\Phi_{\mu_1\mu_2}}
\Phi_{\mu_1\mu_2}^{-1}a_{\mu_1\mu_2}\frac{d}{ds}\Big(\hat V_{\mu_1}(s,\xi-\eta)
\hat V_{\mu_2}(s,\eta)\Big)d\eta ds\\
&\quad+\sum_{\mu_1,\mu_2,\mu_3=\pm}\int_0^t\iint_{\R^2}
e^{is\Phi_{\mu_1\mu_2\mu_3}}b_{\mu_1\mu_2\mu_3}
\hat V_{\mu_1}(s,\xi-\eta)\hat V_{\mu_2}(s,\eta-\zeta)\hat V_{\mu_3}(s,\zeta)d\eta d\zeta ds.
\end{split}
\end{equation*}
Returning to the physical space, one has
\begin{equation}\label{profile:eqn4}
\begin{split}
&V(t,x)=V(0,x)+\int_0^te^{-is\Lambda}\cN_4(U)ds
-i\sum_{\mu_1,\mu_2=\pm}e^{-is\Lambda}
T_{\Phi_{\mu_1\mu_2}^{-1}a_{\mu_1\mu_2}}(U_{\mu_1},U_{\mu_2})\Big|_{s=0}^t\\
&+i\sum_{\mu_1,\mu_2=\pm}\int_0^te^{-is\Lambda}\Big\{
T_{\Phi_{\mu_1\mu_2}^{-1}a_{\mu_1\mu_2}}(e^{is\mu_1\Lambda}\p_tV_{\mu_1},U_{\mu_2})
+T_{\Phi_{\mu_1\mu_2}^{-1}a_{\mu_1\mu_2}}(U_{\mu_1},e^{is\mu_2\Lambda}\p_tV_{\mu_2})\Big\}ds\\
&+\sum_{\mu_1,\mu_2,\mu_3=\pm}\int_0^te^{-is\Lambda}
T_{b_{\mu_1\mu_2\mu_3}}(U_{\mu_1},U_{\mu_2},U_{\mu_3})ds.
\end{split}
\end{equation}
Set
\begin{equation}\label{cN3:def}
\begin{split}
\cN_3(U)&=\cN_4(U)+\sum_{\mu_1,\mu_2,\mu_3=\pm}
T_{b_{\mu_1\mu_2\mu_3}}(U_{\mu_1},U_{\mu_2},U_{\mu_3}),\\ \cN_{3,+}(U)&=\cN_3(U),\qquad\cN_{3,-}(U)=\overline{\cN_3(U)}.
\end{split}
\end{equation}
For $\nu=\pm$,
\begin{equation}\label{profile:eqn5}
\p_tV_\nu=e^{-it\nu\Lambda}(\cN_{3,\nu}(U)+
\sum_{\mu_1,\mu_2=\pm}T_{a^I_{\nu\mu_1\mu_2}}(U_{\mu_1},U_{\mu_2})),
\end{equation}
where
\begin{equation}\label{symbol:a'}
a^I_{+\mu_1\mu_2}=a_{\mu_1\mu_2},\qquad
a^I_{-\mu_1\mu_2}(\xi_1,\xi_2)=\overline{a_{-\mu_1,-\mu_2}(-\xi_1,-\xi_2)}.
\end{equation}
Substituting \eqref{profile:eqn5} into \eqref{profile:eqn4} derives
\begin{equation}\label{profile:eqn6}
\begin{split}
V(t,x)
&=V(0,x)+\int_0^te^{-is\Lambda}\cN^I_4(U)ds\\
&\quad-i\sum_{\mu_1,\mu_2=\pm}e^{-is\Lambda}
T_{\Phi_{\mu_1\mu_2}^{-1}a_{\mu_1\mu_2}}(U_{\mu_1},U_{\mu_2})(s,x)\Big|_{s=0}^t\\
&\quad+\sum_{(\mu_1,\mu_2,\mu_3)\in A_\Phi}
\int_0^te^{-is\Lambda}T_{m_{\mu_1\mu_2\mu_3}}(U_{\mu_1},U_{\mu_2},U_{\mu_3})ds,
\end{split}
\end{equation}
where $A_\Phi:=\{(+++),(++-),(+--),(---)\}$,
\begin{equation}\label{cN4:def}
\cN^I_4(U)=\cN_4(U)+\sum_{\mu,\nu=\pm}
(T_{\Phi_{\mu\nu}^{-1}a_{\mu\nu}}(\cN_{3,\mu}(U),U_\nu)
+T_{\Phi_{\mu\nu}^{-1}a_{\mu\nu}}(U_\mu,\cN_{3,\nu}(U)))
\end{equation}
and
\begin{equation}\label{symbol:m}
\begin{split}
m_{\mu_1\mu_2\mu_3}&=m^I_{\mu_1\mu_2\mu_3}+m^{II}_{\mu_1\mu_2\mu_3}\\
\end{split}
\end{equation}
with
\begin{equation}\label{symbol:m-0}
\begin{split}
m^I_{+++}(\xi_1,\xi_2,\xi_3)&=b^I_{+++}(\xi_1,\xi_2,\xi_3),\\
m^I_{++-}(\xi_1,\xi_2,\xi_3)&=b^I_{++-}(\xi_1,\xi_2,\xi_3)
+b^I_{+-+}(\xi_1,\xi_3,\xi_2)+b^I_{-++}(\xi_3,\xi_2,\xi_1),\\
m^I_{+--}(\xi_1,\xi_2,\xi_3)&=b^I_{+--}(\xi_1,\xi_2,\xi_3)
+b^I_{-+-}(\xi_2,\xi_1,\xi_3)+b^I_{--+}(\xi_3,\xi_2,\xi_1),\\
m^I_{---}(\xi_1,\xi_2,\xi_3)&=b^I_{---}(\xi_1,\xi_2,\xi_3),\\
m^{II}_{+++}(\xi_1,\xi_2,\xi_3)&=b_{+++}(\xi_1,\xi_2,\xi_3),\\
m^{II}_{++-}(\xi_1,\xi_2,\xi_3)&=b_{++-}(\xi_1,\xi_2,\xi_3)
+b_{+-+}(\xi_1,\xi_3,\xi_2)+b_{-++}(\xi_3,\xi_2,\xi_1),\\
m^{II}_{+--}(\xi_1,\xi_2,\xi_3)&=b_{+--}(\xi_1,\xi_2,\xi_3)
+b_{-+-}(\xi_2,\xi_1,\xi_3)+b_{--+}(\xi_3,\xi_2,\xi_1),\\
m^{II}_{---}(\xi_1,\xi_2,\xi_3)&=b_{---}(\xi_1,\xi_2,\xi_3),\\
b^I_{\sigma\mu_1\mu_2}(\xi_1,\xi_2,\xi_3)
&=i\sum_{\mu=\pm}(\Phi_{\mu\sigma}^{-1}a_{\mu\sigma})(\xi_2+\xi_3,\xi_1)
a^I_{\mu\mu_1\mu_2}(\xi_2,\xi_3)\\
&\quad+i\sum_{\nu=\pm}(\Phi_{\sigma\nu}^{-1}a_{\sigma\nu})(\xi_1,\xi_2+\xi_3)
a^I_{\nu\mu_1\mu_2}(\xi_2,\xi_3),\quad\sigma,\mu_1,\mu_2=\pm.
\end{split}
\end{equation}

\subsection{Partial second normal form transformation}\label{2nd:nf}

We require the second normal form to transform some parts of the cubic nonlinearity in
\eqref{profile:eqn6} into the quartic one.
Note that if $\max\{k_1,k_2\}\le k_3-O(1)$ with $O(1)$ being a fixed and large enough number, one then has
\begin{equation}\label{badphase:bdd}
\begin{split}
|\Phi_{++-}(\xi_1,\xi_2,\xi_3)|&=|-\Lambda(\xi_1+\xi_2+\xi_3)+\Lambda(\xi_1)
+\Lambda(\xi_2)-\Lambda(\xi_3)|\\
&\ge\Lambda(\xi_3)/2\approx2^{k_3},
\end{split}
\end{equation}
where $|\xi_l|\approx2^{k_l},l=1,2,3$.
Acting $P_k$ to \eqref{profile:eqn6}, together with \eqref{proj:decom}, yields that
\begin{equation}\label{profile:eqn7}
\begin{split}
&P_kV(t,x)
=P_kV(0,x)+\int_0^te^{-is\Lambda}P_k\cN^I_4(U)ds
-i\sum_{\mu_1,\mu_2=\pm}e^{-is\Lambda}P_k
T_{\Phi_{\mu_1\mu_2}^{-1}a_{\mu_1\mu_2}}(U_{\mu_1},U_{\mu_2})\Big|_{s=0}^t\\
&\qquad+\sum_{(\mu_1,\mu_2,\mu_3)\in A_\Phi^{good}}\sum_{(k_1,k_2,k_3)\in\cY_k}
\int_0^te^{-is\Lambda}P_kT_{m_{\mu_1\mu_2\mu_3}}
(P_{k_1}U_{\mu_1},P_{k_2}U_{\mu_2},P_{k_3}U_{\mu_3})ds\\
&\qquad+\sum_{\substack{(k_1,k_2,k_3)\in\cY_k,\\ \max\{k_1,k_2\}\le k_3-O(1)}}
\int_0^te^{-is\Lambda}P_kT_{m_{++-}}(P_{k_1}U,P_{k_2}U,P_{k_3}U_-)ds\\
&\qquad+\sum_{\substack{(k_1,k_2,k_3)\in\cY_k,\\ \max\{k_1,k_2\}\ge k_3-O(1)}}
\int_0^te^{-is\Lambda}P_kT_{m_{++-}}(P_{k_1}U,P_{k_2}U,P_{k_3}U_-)ds,
\end{split}
\end{equation}
where $A_\Phi^{good}:=\{(+++),(+--),(---)\}$.

Analogously to \eqref{profile:eqn4}, from \eqref{3phase:bdd} and \eqref{badphase:bdd}, we
can transform the cubic nonlinearities in the second and third lines of \eqref{profile:eqn7} into
the corresponding quartic form. Then
\begin{equation}\label{profile:eqn8}
P_kV(t,x)=P_kV(0,x)+\cB_k+\int_0^t(\cC_k(s)+\cQ_k(s)+P_ke^{-is\Lambda}\cN^I_4(U))ds,
\end{equation}
where the boundary term $\cB_k$, the cubic nonlinearity $\cC_k(s)$ and the quartic nonlinearity  $\cQ_k(s)$ are
respectively
\begin{equation}\label{bdry:def}
\begin{split}
\cB_k&:=-i\sum_{\mu_1,\mu_2=\pm}e^{-is\Lambda}P_k
T_{\Phi_{\mu_1\mu_2}^{-1}a_{\mu_1\mu_2}}(U_{\mu_1},U_{\mu_2})\Big|_{s=0}^t\\
&-i\sum_{(\mu_1,\mu_2,\mu_3)\in A_\Phi^{good}}\sum_{(k_1,k_2,k_3)\in\cY_k}
e^{-is\Lambda}P_kT_{\Phi_{\mu_1\mu_2\mu_3}^{-1}m_{\mu_1\mu_2\mu_3}}
(P_{k_1}U_{\mu_1},P_{k_2}U_{\mu_2},P_{k_3}U_{\mu_3})\Big|_{s=0}^t\\
&-i\sum_{\substack{(k_1,k_2,k_3)\in\cY_k,\\ \max\{k_1,k_2\}\le k_3-O(1)}}
e^{-is\Lambda}P_kT_{\Phi_{++-}^{-1}m_{++-}}(P_{k_1}U,P_{k_2}U,P_{k_3}U_-)\Big|_{s=0}^t,
\end{split}
\end{equation}
\begin{equation}\label{cubic:def}
\cC_k(t):=\sum_{\substack{(k_1,k_2,k_3)\in\cY_k,\\ \max\{k_1,k_2\}\ge k_3-O(1)}}
e^{-it\Lambda}P_kT_{m_{++-}}(P_{k_1}U,P_{k_2}U,P_{k_3}U_-),
\end{equation}
\begin{equation}\label{quartic:def}
\begin{split}
&\cQ_k(t):=i\sum_{\substack{(k_1,k_2,k_3)\in\cY_k,\\
(\mu_1,\mu_2,\mu_3)\in A_\Phi^{good}}}
P_ke^{-it\Lambda}\Big\{T_{\Phi_{\mu_1\mu_2\mu_3}^{-1}m_{\mu_1\mu_2\mu_3}}
(e^{it\mu_1\Lambda}P_{k_1}\p_tV_{\mu_1},P_{k_2}U_{\mu_2},P_{k_3}U_{\mu_3})\\
&\quad+T_{\Phi_{\mu_1\mu_2\mu_3}^{-1}m_{\mu_1\mu_2\mu_3}}
(P_{k_1}U_{\mu_1},e^{is\mu_2\Lambda}P_{k_2}\p_tV_{\mu_2},P_{k_3}U_{\mu_3})\\
&\quad+T_{\Phi_{\mu_1\mu_2\mu_3}^{-1}m_{\mu_1\mu_2\mu_3}}
(P_{k_1}U_{\mu_1},P_{k_2}U_{\mu_2},e^{it\mu_3\Lambda}P_{k_3}\p_tV_{\mu_3})\Big\}\\
&\quad +i\sum_{\substack{(k_1,k_2,k_3)\in\cY_k,\\ \max\{k_1,k_2\}\le k_3-O(1)}}
P_k e^{-it\Lambda}\Big\{T_{\Phi_{++-}^{-1}m_{++-}}
(e^{it\Lambda}P_{k_1}\p_tV,P_{k_2}U,P_{k_3}U_-)\\
&\quad+T_{\Phi_{++-}^{-1}m_{++-}}(P_{k_1}U,e^{it\Lambda}P_{k_2}\p_tV,P_{k_3}U_-)
+T_{\Phi_{++-}^{-1}m_{++-}}(P_{k_1}U,P_{k_2}U,e^{-it\Lambda}P_{k_3}\p_tV_-)\Big\}.
\end{split}
\end{equation}

\section{Energy estimate and continuity of $Z_{\alpha}$-norm}

\subsection{Energy estimate}

\begin{lemma}\label{lem:energy}
Let $N\ge27$. Suppose that $U$ is defined by \eqref{profile:def} and $\|U(t)\|_{H^N(\R)}$ is small, one then
has that for $t\ge0$,
\begin{equation}\label{energy}
\begin{split}
\|U(t)\|_{H^N(\R)}&\ls\|U(0)\|_{H^N(\R)}+\|U(0)\|^2_{H^N(\R)}+\|U(0)\|^3_{H^N(\R)}\\
&\quad+\int_0^t\sum_{k\ge-1}2^{k(7+1/4)}\|P_kU(s)\|_{L^\infty}
\|U(s)\|_{W^{1,\infty}}\|U(s)\|_{H^N(\R)}ds.
\end{split}
\end{equation}
\end{lemma}
\begin{proof}

By \eqref{proj:decom}, \eqref{profile:eqn6} and the unitarity  of $e^{-is\Lambda}$, we have
\begin{equation}\label{energy1}
\begin{split}
\|P_k(V(t)-V(0))\|_{L^2}&\ls\sum_{(k_1,k_2)\in\cX_k}(J_{kk_1k_2}(0)+J_{kk_1k_2}(t))\\
&\quad+\int_0^t(\|P_k\cN^I_4(U)\|_{L^2}
+\sum_{(k_1,k_2,k_3)\in\cY_k}J_{kk_1k_2k_3}(s))ds,
\end{split}
\end{equation}
where
\begin{equation}\label{energy2}
\begin{split}
J_{kk_1k_2}(t)&:=\sum_{\mu_1,\mu_2=\pm}\|P_kT_{\Phi_{\mu_1\mu_2}^{-1}a_{\mu_1\mu_2}}
(P_{k_1}U_{\mu_1},P_{k_2}U_{\mu_2}))(t)\|_{L^2},\\
J_{kk_1k_2k_3}(s)&:=\sum_{(\mu_1,\mu_2,\mu_3)\in A_\Phi}
\|P_kT_{m_{\mu_1\mu_2\mu_3}}(P_{k_1}U_{\mu_1},P_{k_2}U_{\mu_2},P_{k_3}U_{\mu_3})(s)\|_{L^2}.
\end{split}
\end{equation}

\noindent\textbf{(A) Estimate of $J_{kk_1k_2}(t)$}
\vskip 0.2 true cm

It only suffices to deal with the case of $k_1\le k_2$ in $\cX_k$ for $J_{kk_1k_2}(t)$
since the treatment on the case of $k_1\ge k_2$ is completely similar.
Applying \eqref{bilinear:a} and the Bernstein inequality yields
\begin{equation*}
\begin{split}
J_{kk_1k_2}(t)
&\ls\sum_{\mu_1,\mu_2=\pm}\|T_{\Phi_{\mu_1\mu_2}^{-1}a_{\mu_1\mu_2}}
(P_{k_1}U_{\mu_1},P_{k_2}U_{\mu_2}))(t)\|_{L^2}\\
&\ls2^{5k_1}\|P_{k_1}U(t)\|_{L^\infty}\|P_{k_2}U(t)\|_{L^2}\\
&\ls2^{k_1(5+\frac12)}\|P_{k_1}U(t)\|_{L^2}\|P_{k_2}U(t)\|_{L^2}.
\end{split}
\end{equation*}
Then
\begin{equation}\label{energy3}
\begin{split}
&\quad\;\Big\|2^{kN}\sum_{(k_1,k_2)\in\cX_k}J_{kk_1k_2}(t)\Big\|_{\ell^2_k}\\
&\ls\Big\|\sum_{(k_1,k_2)\in\cX_k^1}2^{k_2N}J_{kk_1k_2}(t)\Big\|_{\ell^2_k}
+\Big\|\sum_{(k_1,k_2)\in\cX_k^2}2^{k_2(N-1/8)-k/8+k_1/4}J_{kk_1k_2}(t)\Big\|_{\ell^2_k}\\
&\ls\sum_{k_1\ge-1}2^{k_1(5+\frac12)}\|P_{k_1}U(t)\|_{L^2}\Big\|2^{k_2N}\|P_{k_2}U(t)\|_{L^2}\Big\|_{\ell^2_{k_2}}\\
&\quad+\sum_{k_1\ge-1}2^{k_1(5+\frac34)}\|P_{k_1}U(t)\|_{L^2}\|U(t)\|_{H^N}\\
&\ls\|U(t)\|^2_{H^N},
\end{split}
\end{equation}
where $\ds\|A_k\|_{\ell^p_k}=(\sum_{k\ge-1}A_k^p)^{1/p}$ with $p\ge1$.

\noindent\textbf{(B) Estimate of $J_{kk_1k_2k_3}(s)$}
\vskip 0.2 true cm

Without loss of generality,  $k_1\le k_2\le k_3$ is assumed in $J_{kk_1k_2k_3}(s)$.
It follows from \eqref{trilin} that
\begin{equation*}
J_{kk_1k_2k_3}(s)\ls2^{7k_2}\|P_{k_1}U(s)\|_{L^\infty}
\|P_{k_2}U(s)\|_{L^\infty}\|P_{k_3}U(s)\|_{L^2}.
\end{equation*}
Similarly to \eqref{energy3}, one can achieve
\begin{equation}\label{energy4}
\Big\|2^{kN}\sum_{(k_1,k_2,k_3)\in\cY_k}
J_{kk_1k_2k_3}(s)\Big\|_{\ell^2_k}
\ls\sum_{k_2\ge-1}2^{k_2(7+1/4)}\|P_{k_2}U(s)\|_{L^\infty}
\|U(s)\|_{W^{1,\infty}}\|U(s)\|_{H^N}.
\end{equation}
\noindent\textbf{(C) Estimate of $P_k\cN^I_4(U)$}
\vskip 0.2 true cm
Note that
\begin{equation}\label{energy5}
\Big\|2^{kN}\|P_k\cN^I_4(U)\|_{L^2}\Big\|_{\ell^2_k}
\ls\sum_{k_2\ge-1}2^{k_2(7+1/4)}\|P_{k_2}U(s)\|_{L^\infty}\|U(s)\|_{W^{1,\infty}}\|U(s)\|_{H^N}.
\end{equation}

It follows from \eqref{energy1}-\eqref{energy5} that
\begin{equation*}
\begin{split}
\|V(t)-V(0)\|_{H^N}&\ls\Big\|2^{kN}\|P_k(V(t)-V(0))\|_{L^2}\Big\|_{\ell^2_k}
\ls\|U(0)\|^2_{H^N(\R^d)}+\|U(t)\|^2_{H^N(\R^d)}\\
&\quad+\int_0^t\sum_{k_2\ge-1}2^{k_2(7+1/4)}\|P_{k_2}U(s)\|_{L^\infty}
\|U(s)\|_{W^{1,\infty}}\|U(s)\|_{H^N(\R^d)}ds.
\end{split}
\end{equation*}
On the other hand, the unitarity of $e^{it\Lambda}$ ensures
\begin{equation*}
\|U(t)\|_{H^N}\ls\|V(t)\|_{H^N}\ls\|V(0)\|_{H^N}+\|V(t)-V(0)\|_{H^N}.
\end{equation*}
Therefore, \eqref{energy} is proved.
\end{proof}

\subsection{Continuity of $Z_\alpha$-norm}

In order to take a continuation argument later, the following continuous property of $Z_\alpha$-norm is required.
\begin{proposition}[Continuity and boundedness of $Z_\alpha$-norm]
\label{Prop:Cont}
Assume that $u\in C([0,T_0],H^{N+1}(\R))\cap C^1([0,T_0],H^{N}(\R))$ is a solution of problem \eqref{KG}.
Define $U$ as in \eqref{profile:def} with the property $U_0=U(0)\in Z_{\alpha}$.
Then it holds that
\begin{equation}\label{Cont}
\sup_{t\in[0,T_0]}\|e^{-it\Lambda}U(t)\|_{Z_{\alpha}}\le C\Big(T_0,\|U_0\|_{Z_\alpha},\sup_{t\in[0,T_0]}\|U(t)\|_{H^N(\R)}\Big).
\end{equation}
Moreover, the mapping $t\mapsto e^{-it\Lambda}U(t)$ is continuous from $[0,T_0]$ to $Z_\alpha$.
\end{proposition}
\begin{proof}
Let $C>0$ denote the sufficiently large generic constant that depends only on $T_0$, $\|U_0\|_{Z_\alpha}$ and  $\ds\sup_{t\in[0,T_0]}\|U(t)\|_{H^N(\R)}$.

For integer $J\ge0$ and $f\in H^N(\R)$, define
\begin{equation}\label{Cont1}
\|f\|_{Z^J_\alpha}:=\sum_{j,k\ge-1}2^{\min\{j\alpha,J\}+N_1k}
\|Q_jP_kf\|_{L^2(\R)},\qquad\alpha\in(0,1/2].
\end{equation}
This obviously means that there is a constant $C_J>0$ which depends on $J$ such that
\begin{equation*}
\|f\|_{Z^J_\alpha}\le\|f\|_{Z_\alpha},\qquad\|f\|_{Z^J_\alpha}\le C_J\|f\|_{H^N(\R)}.
\end{equation*}
As in (3.20) of \cite{IP13}, we shall show that when $t,t'\in[0,T_0]$ with $0\le t'-t\le1$, for any $J\ge0$,
one has
\begin{equation}\label{Cont2}
\|e^{-it'\Lambda}U(t')-e^{-it\Lambda}U(t)\|_{Z^J_\alpha}
\le C|t'-t|\Big(1+\sup_{s\in [t,t']}\|e^{-is\Lambda}U(s)\|_{Z^J_\alpha}\Big).
\end{equation}
Note that under \eqref{Cont2}, for any $t,t'\in [0,T_0]$,
\begin{equation}\label{Cont3}
\sup_{t\in[0,T_0]}\|e^{-it\Lambda}U(t)\|_{Z^J_\alpha}\le C,\qquad \|e^{-it'\Lambda}U(t')
-e^{-it\Lambda}U(t)\|_{Z^J_\alpha}\le C|t'-t|
\end{equation}
hold uniformly in $J$.
Subsequently, letting $J\rightarrow\infty$ in \eqref{Cont1} and \eqref{Cont3} yields the results in \eqref{Cont}.

Integrating \eqref{profile:eqn1} and \eqref{nonlinear} over $[t,t']$ yields
\begin{equation}\label{Cont4}
\begin{split}
V(t')-V(t)&=\int_t^{t'}e^{-is\Lambda}\cN_4(U)ds
+\sum_{\mu_1,\mu_2=\pm}\int_t^{t'}e^{-is\Lambda}
T_{a_{\mu_1\mu_2}}(U_{\mu_1},U_{\mu_2})(s)ds\\
&\quad+\sum_{\mu_1,\mu_2,\mu_3=\pm}\int_t^{t'}e^{-is\Lambda}
T_{b_{\mu_1\mu_2\mu_3}}(U_{\mu_1},U_{\mu_2},U_{\mu_3})(s)ds.
\end{split}
\end{equation}
Since \eqref{Cont2} is equivalent to
\begin{align}\label{Cont5}
\|V(t')-V(t)\|_{Z^J_\alpha}\le C|t'-t|\Big(1+\sup_{s\in[t,t']}\|V(s)\|_{Z^J_\alpha}\Big),
\end{align}
then \eqref{Cont4}, \eqref{Cont5} as well as \eqref{Cont2} will be obtained if there
hold for $s\in[t,t']$ and $\mu_1,\mu_2,\mu_3=\pm$:
\addtocounter{equation}{1}
\begin{align}
\|e^{-is\Lambda}T_{a_{\mu_1\mu_2}}(U_{\mu_1},U_{\mu_2})\|_{Z^J_\alpha}
&\le C\Big(1+\sup_{s\in[t,t']}\|V(s)\|_{Z^J_\alpha}\Big),
\tag{\theequation a}\label{Cont6a}\\
\|e^{-is\Lambda}T_{b_{\mu_1\mu_2\mu_3}}(U_{\mu_1},U_{\mu_2},U_{\mu_3})\|_{Z^J_\alpha}
&\le C\Big(1+\sup_{s\in[t,t']}\|V(s)\|_{Z^J_\alpha}\Big),
\tag{\theequation b}\label{Cont6b}\\
\|e^{-is\Lambda}\cN_4(U)\|_{Z^J_\alpha}
&\le C\Big(1+\sup_{s\in[t,t']}\|V(s)\|_{Z^J_\alpha}\Big).
\tag{\theequation c}\label{Cont6c}
\end{align}
Next, we prove \eqref{Cont6a}.
Let $C(T_0)>0$ be a large constant to be determined later.

\vskip 0.2cm

\noindent\textbf{Case 1.} $j\le C(T_0)$

\vskip 0.1cm

We now establish
\begin{equation}\label{Cont7}
\sum_{-1\le j\le C(T_0),k\ge-1}2^{\min\{j\alpha,J\}+N_1k}
\|Q_jP_ke^{-is\Lambda}T_{a_{\mu_1\mu_2}}(U_{\mu_1},U_{\mu_2})\|_{L^2(\R)}\le C.
\end{equation}
By \eqref{proj:decom}, one has
\begin{equation*}
P_kT_{a_{\mu_1\mu_2}}(U_{\mu_1},U_{\mu_2})
=\sum_{(k_1,k_2)\in\cX_k}P_kT_{a_{\mu_1\mu_2}}(P_{k_1}U_{\mu_1},P_{k_2}U_{\mu_2}).
\end{equation*}
Without loss of generality, $k_1\ge k_2$ is assumed.
In addition, $2^k\ls2^{k_1}$ holds true.
Then it follows from \eqref{bilinear:b} and the Bernstein inequality that
\begin{equation*}
\begin{split}
&\sum_{-1\le j\le C(T_0),k\ge-1}2^{\min\{j\alpha,J\}+N_1k}
\|Q_jP_ke^{-is\Lambda}T_{a_{\mu_1\mu_2}}(U_{\mu_1},U_{\mu_2})\|_{L^2(\R)}\\
&\le C\sum_{k_1,k_2\ge-1}2^{j\alpha+N_1k_1}\|T_{a_{\mu_1\mu_2}}
(P_{k_1}U_{\mu_1},P_{k_2}U_{\mu_2})\|_{L^2}\\
&\le C\sum_{k_1,k_2\ge-1}2^{N_1k_1}\|P_{k_1}U_{\mu_1}\|_{L^2}
\|P_{k_2}U_{\mu_2}\|_{L^\infty}\\
&\le C\sum_{k_1,k_2\ge-1}2^{(N_1-N)k_1+k_2/2}\|U_{\mu_1}\|_{H^N}
\|P_{k_2}U_{\mu_2}\|_{L^2}\\
&\le C\sum_{k_1,k_2\ge-1}2^{(N_1-N)(k_1+k_2)}\|U\|^2_{H^N}
\le C,
\end{split}
\end{equation*}
which derives \eqref{Cont7}.

\vskip 0.2cm

\noindent\textbf{Case 2.} $j\ge C(T_0)$

\vskip 0.1cm

In this case, we establish
\begin{equation}\label{Cont8}
\sum_{j\ge C(T_0),k\ge-1}2^{\min\{j\alpha,J\}+N_1k}
\|Q_jP_ke^{-is\Lambda}T_{a_{\mu_1\mu_2}}(U_{\mu_1},U_{\mu_2})\|_{L^2(\R)}\le C.
\end{equation}
By virtue of \eqref{proj:proj}, one has
\begin{equation}\label{Cont9}
\begin{split}
&Q_jP_ke^{-is\Lambda}T_{a_{\mu_1\mu_2}}(U_{\mu_1},U_{\mu_2})
=\sum_{j_1,j_2\ge-1}\sum_{(k_1,k_2)\in\cX_k}J_{kk_1k_2}^{jj_1j_2},\\
&J_{kk_1k_2}^{jj_1j_2}:=Q_jP_ke^{-is\Lambda}T_{a_{\mu_1\mu_2}}
(e^{is\mu_1\Lambda}P_{[[k_1]]}Q_{j_1}P_{k_1}V_{\mu_1},
e^{is\mu_2\Lambda}P_{[[k_2]]}Q_{j_2}P_{k_2}V_{\mu_2}).
\end{split}
\end{equation}
As in Case 1, $k_1\ge k_2$ is assumed.
Note that $J_{kk_1k_2}^{jj_1j_2}$ can be written as
\begin{equation}\label{Cont10}
\begin{split}
&J_{kk_1k_2}^{jj_1j_2}(t,x)
=(2\pi)^{-2}\psi_j(x)\iint_{\R^2}K_0(x-x_1,x-x_2)Q_{j_1}P_{k_1}V_{\mu_1}(s,x_1)
Q_{j_2}P_{k_2}V_{\mu_2}(s,x_2)dx_1dx_2,\\
\end{split}
\end{equation}
where
\begin{equation}\label{Cont10-0}
\begin{split}
&K_0(x-x_1,x-x_2)=\iint_{\R^2}e^{i\Psi_0}
a_{\mu_1\mu_2}(\xi_1,\xi_2)\psi_k(\xi_1+\xi_2)
\psi_{[[k_1]]}(\xi_1)\psi_{[[k_2]]}(\xi_2)
d\xi_1d\xi_2,\\
&\Psi_0=s(-\Lambda(\xi_1+\xi_2)+\mu_1\Lambda(\xi_1)+\mu_2\Lambda(\xi_2))
+\xi_1(x-x_1)+\xi_2(x-x_2).
\end{split}
\end{equation}
If $C(T_0)>0$ is sufficiently large, when $j\ge C(T_0)$ and $s\in[0,T_0]$,
then the possible critical points of the phase $\Psi_0$
in \eqref{Cont10-0} are contained in the scope of $\max\{|j-j_1|,|j-j_2|\}\le O(1)$.
The proof of \eqref{Cont8} will be separated into such two subcases:
$\max\{|j-j_1|,|j-j_2|\}\ge O(1)$ and $\max\{|j-j_1|,|j-j_2|\}\le O(1)$.

\vskip 0.1cm
\noindent\textbf{Subcase 2.1.} $\max\{|j-j_1|,|j-j_2|\}\ge O(1)$

\vskip 0.1cm

Denote the operator $\cL_0$ and its adjoint operator $\cL_0^*$ as
\begin{equation*}
\begin{split}
&\cL_0:=-i(|\p_{\xi_1}\Psi_0|^2+|\p_{\xi_2}\Psi_0|^2)^{-1}
(\p_{\xi_1}\Psi_0\p_{\xi_1}+\p_{\xi_2}\Psi_0\p_{\xi_2}),\\
&\cL_0^*:=i\p_{\xi_1}\Big(\frac{\p_{\xi_1}\Psi_0\cdot}
{|\p_{\xi_1}\Psi_0|^2+|\p_{\xi_2}\Psi_0|^2}\Big)
+i\p_{\xi_2}\Big(\frac{\p_{\xi_2}\Psi_0\cdot}
{|\p_{\xi_1}\Psi_0|^2+|\p_{\xi_2}\Psi_0|^2}\Big).
\end{split}
\end{equation*}
Then $\cL_0e^{i\Psi_0}=e^{i\Psi_0}$.
The fact of $|\Lambda'(y)|\le1$ and the condition of $\max\{|j-j_1|,|j-j_2|\}\ge O(1)$ for $j\ge C(T_0)$ with large $C(T_0)$ lead to
\begin{equation*}
|\p_{\xi_1}\Psi_0|+|\p_{\xi_2}\Psi_0|
\gt|x-x_1|+|x-x_2|\gt2^{\max\{j,j_1,j_2\}}.
\end{equation*}
On the other hand, $|\Lambda^{(l)}(y)|\ls1$ holds for $l\ge1$, which yields
\begin{equation*}
|\p_{\xi_1}^l\Psi_0|+|\p_{\xi_2}^l\Psi_0|\ls s\ls T_0\qquad \text{for $l\ge2$}.
\end{equation*}
By the method of stationary phase, we can achieve
\begin{equation*}
\begin{split}
&\quad\;|K_0(x-x_1,x-x_2)|\\
&=\Big|\iint_{\R^2}\cL_0^4(e^{i\Psi_0})a_{\mu_1\mu_2}(\xi_1,\xi_2)
\psi_k(\xi_1+\xi_2)\psi_{[[k_1]]}(\xi_1)\psi_{[[k_2]]}(\xi_2)d\xi_1d\xi_2\Big|\\
&\ls\iint_{\R^2}|(\cL_0^*)^4[a_{\mu_1\mu_2}(\xi_1,\xi_2)
\psi_k(\xi_1+\xi_2)\psi_{[[k_1]]}(\xi_1)\psi_{[[k_2]]}(\xi_2)]|d\xi_1d\xi_2\\
&\ls2^{k_1+k_2-\max\{j,j_1,j_2\}}(1+|x-x_1|+|x-x_2|)^{-3}.
\end{split}
\end{equation*}
This, together with the H\"{o}lder inequality \eqref{Holder}, the Bernstein inequality and \eqref{Cont10}, implies
\begin{equation*}
\begin{split}
\|J_{kk_1k_2}^{jj_1j_2}\|_{L^2(\R)}
&\ls\|K_0(\cdot,\cdot)\|_{L^1(\R^2)}\|Q_{j_1}P_{k_1}V_{\mu_1}\|_{L^2}
\|Q_{j_2}P_{k_2}V_{\mu_2}\|_{L^\infty}\\
&\ls2^{k_1+k_2-\max\{j,j_1,j_2\}}
\|P_{k_1}V_{\mu_1}\|_{L^2}\|P_{k_2}V_{\mu_2}\|_{L^\infty}\\
&\ls2^{k_1(1-N)+3k_2/2-\max\{j,j_1,j_2\}}\|V_{\mu_1}\|_{H^N}\|P_{k_2}V_{\mu_2}\|_{L^2}\\
&\ls2^{(k_1+k_2)(2-N)-\max\{j,j_1,j_2\}}\|U\|^2_{H^N}.
\end{split}
\end{equation*}
Therefore, one arrives at
\begin{equation}\label{Cont11}
\sum_{j\ge C(T_0),k\ge-1}2^{\min\{j\alpha,J\}+N_1k}
\sum_{\substack{j_1,j_2\ge-1,\\ \max\{|j-j_1|,|j-j_2|\}\ge O(1)}}
\sum_{(k_1,k_2)\in\cX_k}\|J_{kk_1k_2}^{jj_1j_2}\|_{L^2(\R)}\le C.
\end{equation}

\vskip 0.1cm
\noindent\textbf{Subcase 2.2.} $\max\{|j-j_1|,|j-j_2|\}\le O(1)$

\vskip 0.1cm

Applying \eqref{bilinear:b} to $J_{kk_1k_2}^{jj_1j_2}$ in \eqref{Cont9} directly yields
\begin{equation*}
\begin{split}
\|J_{kk_1k_2}^{jj_1j_2}\|_{L^2(\R)}
&\ls\|T_{a_{\mu_1\mu_2}}(e^{is\mu_1\Lambda}P_{[[k_1]]}Q_{j_1}P_{k_1}V_{\mu_1},
e^{is\mu_2\Lambda}P_{[[k_2]]}Q_{j_2}P_{k_2}V_{\mu_2})\|_{L^2(\R)}\\
&\ls\|Q_{j_1}P_{k_1}V_{\mu_1}\|_{L^2}
\|e^{is\mu_2\Lambda}P_{[[k_2]]}Q_{j_2}P_{k_2}V_{\mu_2})\|_{L^\infty}\\
&\ls2^{k_2/2}\|Q_{j_1}P_{k_1}V_{\mu_1}\|_{L^2}\|P_{k_2}V_{\mu_2}\|_{L^2},
\end{split}
\end{equation*}
where we have used \eqref{loc:disp} with $\beta=0$.
Due to $2^k\ls2^{k_1}$ and $\max\{|j-j_1|,|j-j_2|\}\le O(1)$, then
\begin{equation*}
\begin{split}
&\sum_{j\ge C(T_0),k\ge-1}2^{\min\{j\alpha,J\}+N_1k}
\sum_{\substack{j_1,j_2\ge-1,\\ \max\{|j-j_1|,|j-j_2|\}\le O(1)}}
\sum_{(k_1,k_2)\in\cX_k}\|J_{kk_1k_2}^{jj_1j_2}\|_{L^2(\R)}\\
&\ls\sum_{j_1,k_1,k_2\ge-1}2^{\min\{j_1\alpha,J\}+N_1k_1}
\|J_{kk_1k_2}^{jj_1j_2}\|_{L^2(\R)}\\
&\ls\sum_{j_1,k_1,k_2\ge-1}2^{\min\{j_1\alpha,J\}+N_1k_1+k_2/2}
\|Q_{j_1}P_{k_1}V_{\mu_1}\|_{L^2}\|P_{k_2}V_{\mu_2}\|_{L^2}\\
&\ls\|V\|_{Z^J_\alpha}\|U\|_{H^N}.
\end{split}
\end{equation*}
This, together with \eqref{Cont9} and \eqref{Cont11}, yields \eqref{Cont8}.

In addition, \eqref{Cont6a} follows from \eqref{Cont7} and \eqref{Cont8}.
Note that only the small value solution problem \eqref{KG} is studied,
then the cubic and higher order nonlinear terms do not cause any additional difficulties.
Then the proofs of \eqref{Cont6b} and \eqref{Cont6c} are omitted here.

\end{proof}

\section{Estimate of $Z_{\alpha}$-norm}

In this section, suppose that the following bootstrap assumption holds for $\alpha\in(0,1/2]$ and $t\in[0,T_{\alpha,\ve}]$,
\begin{equation}\label{thm1BA}
\|V(t)\|_{H^N(\R)}+\|V(t)\|_{Z_\alpha}\le\ve_1.
\end{equation}
This, together with \eqref{Znorm:def}, implies
\begin{equation}\label{thm1BA1}
\sup_{k\ge-1}2^{kN}\|P_kV(t)\|_{L^2(\R)}
+\sum_{j,k\ge-1}2^{j\alpha+N_1k}\|Q_jP_kV(t)\|_{L^2(\R)}\ls\ve_1.
\end{equation}
Acting $Q_j$ to \eqref{profile:eqn8} yields
\begin{equation}\label{QjPk}
Q_jP_kV(t,x)=Q_jP_kV(0,x)+Q_j\cB_k+\int_0^tQ_j(\cC_k(s)+\cQ_k(s) +P_ke^{-is\Lambda}\cN^I_4(U))ds,
\end{equation}
where $\cB_k$, $\cC_k$, $\cQ_k$ and $\cN^I_4(U)$ are defined by \eqref{bdry:def}, \eqref{cubic:def}, \eqref{quartic:def}
and \eqref{cN4:def}, respectively.

\subsection{Estimate of the cubic nonlinearity $\cC_k(s)$}

\begin{lemma}\label{lem:cubic}
Under the bootstrap assumption \eqref{thm1BA1}, it holds that for $\alpha\in(0,1/2]$ and $t\ge0$,
\begin{equation}\label{YH-1}
\sum_{j,k\ge-1}2^{j\alpha+N_1k}\|Q_j\cC_k(t)\|_{L^2(\R)}
\ls\ve_1^3(1+t)^{-2\alpha}.
\end{equation}
\end{lemma}
We point out that the key point for  proving \eqref{YH-1} is  to analyze the corresponding Schwartz kernel
of $\cC_k(s)$ according to the space-time locations and the frequencies.
For this purpose, by \eqref{proj:proj} and \eqref{cubic:def}, we rewrite $Q_j\cC_k(t)$ as
\begin{equation}\label{QjPkCubic}
\begin{split}
Q_j\cC_k(t)=\sum_{j_1,j_2,j_3\ge-1}
\sum_{\substack{(k_1,k_2,k_3)\in\cY_k,\\ \max\{k_1,k_2\}\ge k_3-O(1)}}
I_{kk_1k_2k_3}^{jj_1j_2j_3},
\end{split}
\end{equation}
where
\begin{equation}\label{QjPkCubic1}
\begin{split}
I_{kk_1k_2k_3}^{jj_1j_2j_3}&:=Q_jP_ke^{-it\Lambda}T_{m_{++-}}
(e^{it\Lambda}P_{[[k_1]]}\cV_1,e^{it\Lambda}P_{[[k_2]]}\cV_2,
e^{-it\Lambda}P_{[[k_3]]}\cV_3),\\
\cV_1&:=Q_{j_1}P_{k_1}V,\quad\cV_2:=Q_{j_2}P_{k_2}V,
\quad\cV_3:=Q_{j_3}P_{k_3}V_-.
\end{split}
\end{equation}
The proof of Lemma \ref{lem:cubic} will be separated into the following two parts in terms of the space-time locations:
outside of the cone and inside of the cone, respectively.
\begin{lemma}[Outside of cone]\label{lem:outcone}
Under the bootstrap assumption \eqref{thm1BA1}, it holds that for $\alpha\in(0,1/2]$ and $t\ge0$,
\begin{equation}\label{OutCone}
\sum_{\substack{j,j_1,j_2,j_3,k\ge-1,\\(k_1,k_2,k_3)\in\cY_k,\\
\max\{k_1,k_2\}\ge k_3-O(1)}}2^{j\alpha+N_1k}
\|I_{kk_1k_2k_3}^{jj_1j_2j_3}\id_{I_{out}}(t)\|_{L^2(\R)}
\ls\ve_1^3(1+t)^{-2\alpha},
\end{equation}
where $I_{out}:=\{t\ge0: \max\{j,j_1,j_2,j_3\}\ge\log_2(1+t)+O(1)\}$ and
\begin{equation}\label{charact:def}
\id_{I}(t):=\left\{
\begin{aligned}
&1,\qquad t\in I,\\
&0,\qquad t\not\in I.
\end{aligned}
\right.
\end{equation}
\end{lemma}

\begin{lemma}[Inside of cone]\label{lem:incone}
Under the bootstrap assumption \eqref{thm1BA1}, one has that for $\alpha\in(0,1/2]$ and $t\ge0$,
\begin{equation}\label{InCone}
\sum_{\substack{j,j_1,j_2,j_3,k\ge-1,\\(k_1,k_2,k_3)\in\cY_k,\\
\max\{k_1,k_2\}\ge k_3-O(1)}}2^{j\alpha+N_1k}
\|I_{kk_1k_2k_3}^{jj_1j_2j_3}\id_{I_{in}}(t)\|_{L^2(\R)}
\ls\ve_1^3(1+t)^{-2\alpha},
\end{equation}
where $I_{in}:=\{t\ge0: {\max\{j,j_1,j_2,j_3\}}\le\log_2(1+t)+O(1)\}$.
\end{lemma}
\vskip 0.2 true cm
It is obvious that Lemma \ref{lem:cubic} comes from Lemmas \ref{lem:outcone} and \ref{lem:incone} directly.

\begin{proof}[Proof of Lemma \ref{lem:outcone}]

According to the definition \eqref{QjPkCubic1}, we have
\begin{equation}\label{OutCone1}
\begin{split}
&I_{kk_1k_2k_3}^{jj_1j_2j_3}(t,x)
=(2\pi)^{-3}\psi_j(x)\iiint_{\R^3}K_1(x-x_1,x-x_2,x-x_3)\cV_1(t,x_1)\\
&\hspace{5cm}\times\cV_2(t,x_2)\cV_3(t,x_3)dx_1dx_2dx_3,\\
\end{split}
\end{equation}
where
\begin{equation}\label{OutCone1-0}
\begin{split}
&K_1(x-x_1,x-x_2,x-x_3)=\iiint_{\R^3}e^{i\Psi_1}
m_{++-}(\xi_1,\xi_2,\xi_3)\psi_k(\xi_1+\xi_2+\xi_3)\\
&\hspace{5cm}\times\psi_{[[k_1]]}(\xi_1)\psi_{[[k_2]]}(\xi_2)
\psi_{[[k_3]]}(\xi_3)d\xi_1d\xi_2d\xi_3,\\
&\Psi_1=t(-\Lambda(\xi_1+\xi_2+\xi_3)+\Lambda(\xi_1)+\Lambda(\xi_2)-\Lambda(\xi_3))\\
&\hspace{5cm}+\xi_1(x-x_1)+\xi_2(x-x_2)+\xi_3(x-x_3).
\end{split}
\end{equation}
If $x\in\supp\psi_j$, $x_l\in\supp\psi_l$ ($l=1,2,3$) and $\max\{j,j_1,j_2,j_3\}\ge\log_2(1+t)+O(1)$,
then the possible critical points of  phase $\Psi_1$ in \eqref{OutCone1-0} are contained in $\ds\max_{l=1,2,3}|j-j_l|\le O(1)$.
Based on this, the proof of \eqref{OutCone} will be separated into such two
cases: $\ds\max_{l=1,2,3}|j-j_l|\ge O(1)$ and $\ds\max_{l=1,2,3}|j-j_l|\le O(1)$.

\vskip 0.2cm

\noindent\textbf{Case 1.} $\ds\max_{l=1,2,3}|j-j_l|\ge O(1)$

Set
\begin{equation*}
\cL_1:=-i(|\p_{\xi_1}\Psi_1|^2+|\p_{\xi_2}\Psi_1|^2+|\p_{\xi_3}\Psi_1|^2)^{-1}
\sum_{l=1}^3\p_{\xi_l}\Psi_1\p_{\xi_l}.
\end{equation*}
Then $\cL_1e^{i\Psi_1}=e^{i\Psi_1}$.
In addition, the adjoint operator of $\cL_1$ is
\begin{equation*}
\cL_1^*:=i\sum_{l=1}^3\p_{\xi_l}\Big(\frac{\p_{\xi_l}\Psi_1\cdot}
{|\p_{\xi_1}\Psi_1|^2+|\p_{\xi_2}\Psi_1|^2+|\p_{\xi_3}\Psi_1|^2}\Big).
\end{equation*}
The conditions $\max\{j,j_1,j_2,j_3\}\ge\log_2(1+t)+O(1)$ and $\ds\max_{l=1,2,3}|j-j_l|\ge O(1)$ ensure that if $x\in\supp\psi_j$, $x_l\in\supp\psi_l$, $l=1,2,3$, then it holds that
\begin{equation*}
\begin{split}
|x-x_1|+|x-x_2|+|x-x_3|&\ge2^{O(1)}(1+t),\\
|x-x_1|+|x-x_2|+|x-x_3|&\gt2^{\max\{j,j_1,j_2,j_3\}}.
\end{split}
\end{equation*}
This, together with $|\Lambda'(y)|\le1$, yields
\begin{equation}\label{OutCone2}
\begin{split}
(|\p_{\xi_1}\Psi_1|^2+|\p_{\xi_2}\Psi_1|^2+|\p_{\xi_3}\Psi_1|^2)^{1/2}
&\gt|x-x_1|+|x-x_2|+|x-x_3|\\
&\gt\max\{1+t,2^{\max\{j,j_1,j_2,j_3\}}\}.
\end{split}
\end{equation}
On the other hand, for $l\ge2$, one obtains from \eqref{OutCone1-0} that
\begin{equation}\label{OutCone3}
|\p^l_{\xi_1,\xi_2,\xi_3}\Psi_1|\ls t.
\end{equation}
Without loss of generality, $\max\{k_1,k_2,k_3\}=k_1$ is assumed.
By the method of stationary phase and \eqref{OutCone2}, \eqref{OutCone3}, \eqref{symbol:m:bdd},
we arrive at
\begin{equation*}
\begin{split}
&\quad\;|K_1(x-x_1,x-x_2,x-x_3)|\\
&=\Big|\iiint_{\R^3}\cL_1^7(e^{i\Psi_1})m_{++-}(\xi_1,\xi_2,\xi_3)
\psi_k(\xi_1+\xi_2+\xi_3)\psi_{[[k_1]]}(\xi_1)\psi_{[[k_2]]}(\xi_2)
\psi_{[[k_3]]}(\xi_3)d\xi_1d\xi_2d\xi_3\Big|\\
&\ls\iiint_{\R^3}|(\cL_1^*)^7[m_{++-}(\xi_1,\xi_2,\xi_3)
\psi_k(\xi_1+\xi_2+\xi_3)\psi_{[[k_1]]}(\xi_1)\psi_{[[k_2]]}(\xi_2)
\psi_{[[k_3]]}(\xi_3)]|d\xi_1d\xi_2d\xi_3\\
&\ls2^{k_1+k_2+k_3+\max\{k_1,k_2,k_3\}}(1+|x-x_1|+|x-x_2|+|x-x_3|)^{-7}\\
&\ls2^{4\max\{k_1,k_2,k_3\}-\max\{j,j_1,j_2,j_3\}}
(1+t)^{-2}(1+|x-x_1|+|x-x_2|+|x-x_3|)^{-4}.
\end{split}
\end{equation*}
This, together with \eqref{thm1BA1}, \eqref{OutCone1}, the H\"{o}lder inequality \eqref{Holder}
and the Bernstein inequality, leads to
\begin{equation}\label{OutCone4}
\begin{split}
\|I_{kk_1k_2k_3}^{jj_1j_2j_3}\|_{L^2(\R)}
&\ls\|K_1(\cdot,\cdot,\cdot)\|_{L^1(\R^3)}\|\cV_1\|_{L^2}\|\cV_2\|_{L^\infty}\|\cV_3\|_{L^\infty}\\
&\ls2^{4\max\{k_1,k_2,k_3\}-\max\{j,j_1,j_2,j_3\}}(1+t)^{-2}\|\cV_1\|_{L^2}
\|\cV_2\|_{L^\infty}\|\cV_3\|_{L^\infty}\\
&\ls2^{4\max\{k_1,k_2,k_3\}-\max\{j,j_1,j_2,j_3\}}(1+t)^{-2}\|P_{k_1}V\|_{L^2}
\|P_{k_2}V\|_{L^\infty}\|P_{k_3}V_-\|_{L^\infty}\\
&\ls2^{4k_1+(k_2+k_3)/2-\max\{j,j_1,j_2,j_3\}}(1+t)^{-2}
\|P_{k_1}V\|_{L^2}\|P_{k_2}V\|_{L^2}\|P_{k_3}V_-\|_{L^2}\\
&\ls\ve_1^32^{(4-N)(k_1+k_2+k_3)-2j/3-(j_1+j_2+j_3)/9}(1+t)^{-2}.
\end{split}
\end{equation}
Combining \eqref{OutCone4} with $N\ge N_1+5$ implies
\begin{equation}\label{OutCone5}
\sum_{\substack{j,j_1,j_2,j_3,k\ge-1,\\ \max\limits_{l=1,2,3}|j-j_l|\ge O(1)}}
\sum_{\substack{(k_1,k_2,k_3)\in\cY_k,\\ \max\{k_1,k_2\}\ge k_3-O(1)}}
2^{j\alpha+N_1k}\|I_{kk_1k_2k_3}^{jj_1j_2j_3}\id_{I_{out}}(t)\|_{L^2(\R)}
\ls\ve_1^3(1+t)^{-2}.
\end{equation}

\vskip 0.2cm

\noindent\textbf{Case 2.} $\ds\max_{l=1,2,3}|j-j_l|\le O(1)$
\vskip 0.1 true cm

Without loss of generality, $\max\{k_1,k_2,k_3\}=k_1$ and $\med\{k_1,k_2,k_3\}=k_2$ are assumed.
Applying \eqref{trilin} to \eqref{QjPkCubic1} yields
\begin{equation}\label{OutCone6}
\|I_{kk_1k_2k_3}^{jj_1j_2j_3}\|_{L^2}
\ls2^{7k_2}\|\cV_1\|_{L^2}\|e^{it\Lambda}P_{[[k_2]]}\cV_2\|_{L^\infty}
\|e^{-it\Lambda}P_{[[k_3]]}\cV_3\|_{L^\infty}.
\end{equation}
By a similar argument of \eqref{energy3}, one can conclude from \eqref{loc:disp} with $\beta=\alpha$,
the assumption \eqref{thm1BA1}, \eqref{OutCone6} and the condition $N_1\ge9$ that
\begin{equation}\label{OutCone7}
\begin{split}
&\sum_{\substack{j,j_1,j_2,j_3,k\ge-1,\\ \max\limits_{l=1,2,3}|j-j_l|\le O(1)}}
\sum_{\substack{(k_1,k_2,k_3)\in\cY_k,\\ \max\{k_1,k_2\}\ge k_3-O(1)}}
2^{j\alpha+N_1k}\|I_{kk_1k_2k_3}^{jj_1j_2j_3}\id_{I_{out}}(t)\|_{L^2(\R)}\\
&\ls(1+t)^{-2\alpha}\sum_{\substack{j_1,j_2,j_3\ge-1,\\ k_1,k_2,k_3\ge-1}}
2^{(j_1+j_2+j_3)\alpha+N_1k_1+9(k_2+k_3)}\|Q_{j_1}P_{k_1}V\|_{L^2}
\|Q_{j_2}P_{k_2}V\|_{L^2}\|Q_{j_3}P_{k_3}V\|_{L^2}\\
&\ls\ve_1^3(1+t)^{-2\alpha}.
\end{split}
\end{equation}
Collecting \eqref{OutCone5} and \eqref{OutCone7} derives \eqref{OutCone}.
\end{proof}

\vskip 0.5cm

\begin{proof}[Proof of Lemma \ref{lem:incone}]
At first, we deal with the case of $t\ge1$. At this time, \eqref{OutCone1} can be reformulated as
\begin{equation}\label{InCone1}
\begin{split}
&I_{kk_1k_2k_3}^{jj_1j_2j_3}(t,x)=(2\pi)^{-3}\psi_j(x)
\iiint_{\R^3}K_2(x-x_1,x-x_2,x-x_3)\cV_1(t,x_1)\\
&\hspace{6cm}\times\cV_2(t,x_2)\cV_3(t,x_3)dx_1dx_2dx_3,\\
\end{split}
\end{equation}
where
\begin{equation}\label{InCone1-0}
\begin{split}
&K_2(x-x_1,x-x_2,x-x_3)=\iiint_{\R^3}e^{i\Psi_2}m_{++-}\psi_k(\xi)
\psi_{[[k_1]]}(\xi-\eta)\psi_{[[k_2]]}(\eta-\zeta)\\
&\hspace{6cm}\times\psi_{[[k_3]]}(\zeta)d\xi d\eta d\zeta,\\
&\Psi_2=t\Phi+\xi(x-x_1)+\eta(x_1-x_2)+\zeta(x_2-x_3),\\
&\Phi=\Phi(\xi,\eta,\zeta)=-\Lambda(\xi)+\Lambda(\xi-\eta)+\Lambda(\eta-\zeta)
-\Lambda(\zeta).
\end{split}
\end{equation}
The proof of \eqref{InCone} will be separated into two cases: $k_3-O(1)\le\max\{k_1,k_2\}\le k_3$ and $\max\{k_1,k_2\}\ge k_3$.
Due to the symmetry, it is convenient to assume $\max\{k_1,k_2\}=k_1$.

\vskip 0.2cm

\noindent\textbf{Case 1. $k_3-O(1)\le k_1\le k_3$}

\vskip 0.1cm

To control the factor $2^{j\alpha}$ in \eqref{InCone}, we will treat such two cases of $\ds j\le\max\{j_1,j_2,j_3\}+O(1)$ and $\ds j\ge\max\{j_1,j_2,j_3\}+O(1)$, respectively.
In addition, note that $2^k\ls2^{\max\{k_1,k_2,k_3\}}\ls2^{k_1}$ holds.

\vskip 0.1cm

\noindent\textbf{Case 1.1. $j\le\max\{j_1,j_2,j_3\}+O(1)$}

\vskip 0.1cm

For convenience, $\ds\max\{j_1,j_2,j_3\}=j_2$ is assumed.
By utilizing \eqref{trilin} as \eqref{OutCone6}, one can obtain
\begin{equation}\label{InCone2}
\begin{split}
\|I_{kk_1k_2k_3}^{jj_1j_2j_3}\|_{L^2(\R)}
&\ls\|T_{m_{++-}}(e^{it\Lambda}P_{[[k_1]]}\cV_1,
e^{it\Lambda}P_{[[k_2]]}\cV_2,e^{-it\Lambda}P_{[[k_3]]}\cV_3)\|_{L^2(\R)}\\
&\ls2^{7k_1}\|e^{it\Lambda}P_{[[k_1]]}\cV_1\|_{L^\infty}\|\cV_2\|_{L^2}
\|e^{-it\Lambda}P_{[[k_3]]}\cV_3\|_{L^\infty}.
\end{split}
\end{equation}
Therefore, it follows from \eqref{loc:disp} with $\beta=\alpha$ and $N_1\ge10$ that
\begin{equation}\label{InCone3}
\begin{split}
&\sum_{\substack{j,j_1,j_2,j_3,k\ge-1,\\j\le\max\{j_1,j_2,j_3\}+O(1)}}
\sum_{\substack{(k_1,k_2,k_3)\in\cY_k,\\k_3-O(1)\le k_1\le k_3}}
2^{j\alpha+N_1k}\|I_{kk_1k_2k_3}^{jj_1j_2j_3}\id_{I_{in}}(t)\|_{L^2(\R)}\\
&\ls\sum_{\substack{j_1,j_2,j_3\ge-1,\\k_1,k_2,k_3\ge-1,\\|k_1-k_3|\le O(1)}}
\sum_{j\le j_2+O(1)}2^{k_1(N_1+7)+j\alpha}
\|e^{it\Lambda}P_{[[k_1]]}\cV_1\|_{L^\infty}\|Q_{j_2}P_{k_2}V\|_{L^2}
\|e^{-it\Lambda}P_{[[k_3]]}\cV_3\|_{L^\infty}\\
&\ls(1+t)^{-2\alpha}
\sum_{\substack{j_1,j_2,j_3\ge-1,\\k_1,k_2,k_3\ge-1,\\|k_1-k_3|\le O(1)}}
2^{k_1(N_1+10)+(j_1+j_2+j_3)\alpha}\|Q_{j_1}P_{k_1}V\|_{L^2}\|Q_{j_2}P_{k_2}V\|_{L^2}
\|Q_{j_3}P_{k_3}V\|_{L^2}\\
&\ls\ve_1^3(1+t)^{-2\alpha}.
\end{split}
\end{equation}

\vskip 0.1cm

\noindent\textbf{Case 1.2. $j\ge\max\{j_1,j_2,j_3\}+O(1)$}
\vskip 0.1cm

At first, we discuss the possible critical points of the phase $\Psi_2$ in \eqref{InCone1}.
Note that
\begin{equation}\label{InCone4}
\begin{split}
\p_\xi\Phi&=\Lambda'(\xi-\eta)-\Lambda'(\xi)=-\eta\Lambda''(\xi-r_1\eta),
\quad r_1\in[0,1],\\
\p_\zeta\Phi&=\Lambda'(\zeta-\eta)-\Lambda'(\zeta)=-\eta\Lambda''(\zeta-r_2\eta),
\quad r_2\in[0,1],\\
\Lambda''(x)&=(1+x^2)^{-3/2}.
\end{split}
\end{equation}
By $|\xi|\approx2^k$, $|\xi-\eta|\approx2^{k_1}$, $|\eta-\zeta|\approx2^{k_2}$, $|\zeta|\approx2^{k_3}$ and
\begin{equation*}
|\xi-r_1\eta|=|r_1(\xi-\eta)+(1-r_1)\xi|\ls2^{\max\{k,k_1\}}\ls2^{k_1},
\end{equation*}
one has
\begin{equation}\label{InCone5}
2^{-3k_1}|\eta|\ls|\p_\xi\Phi|,|\p_\zeta\Phi|\ls|\eta|.
\end{equation}
On the other hand, direct computation shows
\begin{equation}\label{InCone6}
\p_\xi\Psi_2=t\p_\xi\Phi+x-x_1,\qquad\p_\zeta\Psi_2=t\p_\zeta\Phi+x_2-x_3.
\end{equation}
It is noticed that the condition $j\ge\max\{j_1,j_2,j_3\}+O(1)$ ensures $|x-x_1|\approx2^j$.
In view of \eqref{InCone5} and \eqref{InCone6}, in order to give a precise analysis on the
related Schwarz kernel $K_2$ in \eqref{InCone1}, one needs to discuss the scope of frequency $\eta$.
Note that when $2^{-3k_1}t|\eta|\gg2^j$, $|\p_\xi\Psi_2|\ge t|\p_\xi\Phi|-|x-x_1|$ has a lower bound;
when $t|\eta|\ll2^j$, $|\p_\xi\Psi_2|\ge|x-x_1|-t|\p_\xi\Phi|$ also has a lower bound.
Based on this, for a fixed and large enough number $M_1>0$, we now introduce
\begin{equation}\label{InCone7}
\begin{split}
\chi^I_{high}(\eta)&=\chi\Big(\frac{t|\eta|}{2^{j+3k_1+M_1}}\Big),
\qquad\qquad\chi^I_{low}(\eta)=1-\chi\Big(\frac{t|\eta|}{2^{j-M_1}}\Big),\\
\chi^I_{med}(\eta)&=(1-\chi^I_{high}(\eta))(1-\chi^I_{low}(\eta)),
\end{split}
\end{equation}
where the cut-off function $\chi$ with $\chi(s)\in C^\infty(\R)$ and $0\le\chi(s)\le1$ is defined as
\begin{equation}\label{cutoff:def}
\chi(s)=\left\{
\begin{aligned}
&0,\qquad s\le1,\\
&1,\qquad s\ge2.
\end{aligned}
\right.
\end{equation}
If $M_1\ge3$, one then easily knows
\begin{equation}\label{InCone8}
\begin{split}
&\supp\chi^I_{high}\subset\{t|\eta|\ge2^{j+3k_1+M_1}\},\qquad
\supp\chi^I_{low}\subset\{t|\eta|\le2^{j-M_1+1}\},\\
&\supp\chi^I_{high}\cap\supp\chi^I_{low}=\emptyset,\\
&\supp\chi^I_{med}\subset\{2^{j-M_1}\le t|\eta|\le2^{j+3k_1+M_1+1}\}.
\end{split}
\end{equation}
The remaining work is to deal with the case of the medium frequency mode $2^j\ls t|\eta|\ls2^{j+3k_1}$,
where the corresponding phase $\Psi_2$ may have critical points.
On $\supp\chi^I_{med}$, $\eta$ will be separated into the sub-high and sub-low modes according to
the property of $\p_\zeta\Psi_2=0$.
Note that $|x_2-x_3|$ has an upper bound $2^{\max\{j_2,j_3\}}$.
For the sub-high frequency mode $t|\eta|\ge2^{\max\{j_2,j_3\}+3k_1+M_1}$,
we see $|\p_\zeta\Psi_2|\ge t|\p_\zeta\Phi|-|x_2-x_3|$, which means that
there is no critical point for $\Psi_2$.
For the sub-low frequency mode $t|\eta|\le2^{\max\{j_2,j_3\}+3k_1+M_1}$, it follows from
the third line of \eqref{InCone8} that $j\le\max\{j_2,j_3\}+3k_1+2M_1$.
Based on this, the scope of $j$ in Case 1.2 will be separated into $j\le\max\{j_2,j_3\}+3k_1+2M_1$
and $j\ge\max\{j_2,j_3\}+3k_1+2M_1$.

\vskip 0.2cm

\noindent\textbf{Case 1.2.1. $j\le\max\{j_2,j_3\}+3k_1+2M_1$}

\vskip 0.1cm

Without loss of generality, $\max\{j_2,j_3\}=j_2$ is assumed.
Similarly to \eqref{InCone3}, one has that for $N_1\ge12$,
\begin{equation}\label{InCone9}
\begin{split}
&\sum_{\substack{j,j_1,j_2,j_3,k\ge-1,\\j\ge\max\{j_1,j_2,j_3\}+O(1),\\j\le\max\{j_2,j_3\}+3k_3+2M_1}}
\sum_{\substack{(k_1,k_2,k_3)\in\cY_k,\\k_3-O(1)\le k_1\le k_3}}
2^{j\alpha+N_1k}\|I_{kk_1k_2k_3}^{jj_1j_2j_3}\id_{I_{in}}(t)\|_{L^2(\R)}\\
&\ls\sum_{\substack{j_1,j_2,j_3\ge-1,\\k_1,k_2,k_3\ge-1,\\|k_1-k_3|\le O(1)}}
\sum_{j\le j_2+3k_1+2M_1}2^{k_1(N_1+7)+j\alpha}
\|e^{it\Lambda}P_{[[k_1]]}\cV_1\|_{L^\infty}\|Q_{j_2}P_{k_2}V\|_{L^2}
\|e^{-it\Lambda}P_{[[k_3]]}\cV_3\|_{L^\infty}\\
&\ls(1+t)^{-2\alpha}
\sum_{\substack{j_1,j_2,j_3\ge-1,\\k_1,k_2,k_3\ge-1,\\|k_1-k_3|\le O(1)}}
2^{k_1(N_1+12)+(j_1+j_2+j_3)\alpha}\|Q_{j_1}P_{k_1}V\|_{L^2}\|Q_{j_2}P_{k_2}V\|_{L^2}
\|Q_{j_3}P_{k_3}V\|_{L^2}\\
&\ls\ve_1^3(1+t)^{-2\alpha}.
\end{split}
\end{equation}

\vskip 0.2cm

\noindent\textbf{Case 1.2.2. $j\ge\max\{j_2,j_3\}+3k_1+2M_1$}

\vskip 0.1cm

In terms of
\begin{equation*}
\chi^I_{high}(\eta)+\chi^I_{low}(\eta)+\chi^I_{med}(\eta)=1,
\end{equation*}
the Schwartz kernel $K_2$ in \eqref{InCone1} can be separated as
\begin{equation}\label{InCone10}
\begin{split}
&K_2=K^I_{high}+K^I_{low}+K^I_{med},\\
&K^I_{\Xi}=\iiint_{\R^3}\chi^I_{\Xi}(\eta)e^{i\Psi_2}m_{++-}\psi_k(\xi)
\psi_{[[k_1]]}(\xi-\eta)\psi_{[[k_2]]}(\eta-\zeta)
\psi_{[[k_3]]}(\zeta)d\xi d\eta d\zeta,
\end{split}
\end{equation}
where $\Xi\in\{high,low,med\}$.

\vskip 0.2cm

\textbf{$(A_1)$ Estimates of $K^I_{high}$ and $K^I_{low}$}

\vskip 0.1cm
Set
\begin{equation}\label{cL2:def}
\cL_2=-i(\p_\xi\Psi_2)^{-1}\p_\xi,\qquad \cL_2^*=\p_\xi\Big(\frac{i\cdot}{\p_\xi\Psi_2}\Big).
\end{equation}
Then $\ds\cL_2e^{i\Psi_2}=e^{i\Psi_2}$.
Collecting \eqref{InCone5}, \eqref{InCone6}, \eqref{InCone8} with $M_1>0$ large enough yields
\begin{equation}\label{InCone11}
\begin{split}
|\p_\xi\Psi_2|&\gt\max\{2^{-3k_1}t|\eta|,2^j\},\hspace{1cm}\eta\in\supp\chi^I_{high},\\
|\p_\xi\Psi_2|&\gt2^j\gt t|\eta|,\hspace{2.6cm}\eta\in\supp\chi^I_{low}.
\end{split}
\end{equation}
On the other hand, for $l\ge2$, \eqref{InCone4} implies
\begin{equation}\label{InCone12}
|\p^l_\xi\Psi_2|=|t\p^l_\xi\Phi|=|t\eta\Lambda^{(l+1)}(\xi-\tilde r_1\eta)|\ls t|\eta|,
\quad\tilde r_1\in[0,1].
\end{equation}
Applying the method of stationary phase, we arrive at
\begin{equation}\label{InCone13}
\begin{split}
|K^I_{high}|&=\Big|\iiint_{\R^3}\chi^I_{high}(\eta)\cL_2^5(e^{i\Psi_2})m_{++-}
\psi_k(\xi)\psi_{[[k_1]]}(\xi-\eta)\psi_{[[k_2]]}(\eta-\zeta)
\psi_{[[k_3]]}(\zeta)d\xi d\eta d\zeta\Big|\\
&\ls\iiint_{\R^3}\chi^I_{high}(\eta)\Big|(\cL_2^*)^5[m_{++-}
\psi_k(\xi)\psi_{[[k_1]]}(\xi-\eta)\psi_{[[k_2]]}(\eta-\zeta)
\psi_{[[k_3]]}(\zeta)]\Big|d\xi d\eta d\zeta.
\end{split}
\end{equation}
In view of \eqref{symbol:m:bdd}, the worst term $(\cL_2^*)^5[\cdots]$ in \eqref{InCone13} can be
estimated by \eqref{InCone11} and \eqref{InCone12} as follows
\begin{equation}\label{InCone14}
\frac{|\p^2_\xi\Psi_2|^5}{(\p_\xi\Psi_2)^{10}}
\ls\frac{t^5|\eta|^5}{(\p_\xi\Psi_2)^{10}}
\ls2^{21k_1-3j}t^{-2}\eta^{-2},\qquad\eta\in\supp\chi^I_{high}.
\end{equation}
Note that $\chi^I_{high}(\eta)$ vanishes in a neighbourhood of the origin.
Then it follows from the integration by parts in $\eta$ and \eqref{cutoff:def}-\eqref{InCone8} that
\begin{equation}\label{InCone15}
\begin{split}
&\Big|\iiint_{\R^3}\psi_{[[k_1]]}(\xi-\eta)\psi_{[[k_2]]}(\eta-\zeta)
\psi_{[[k_3]]}(\zeta)\chi^I_{high}(\eta)\eta^{-2}d\xi d\eta d\zeta\Big|\\
=&\Big|\iiint_{\R^3}\psi_{[[k_1]]}(\tilde\xi)\psi_{[[k_2]]}(\eta-\zeta)
\psi_{[[k_3]]}(\zeta)\chi^I_{high}(\eta)\eta^{-2}d\tilde\xi d\eta d\zeta\Big|\\
\ls&2^{k_1}\Big|\iint_{\R^2}\psi_{[[k_2]]}(\eta-\zeta)\psi_{[[k_3]]}(\zeta)
\chi^I_{high}(\eta)d(-\eta^{-1})~d\zeta\Big|\\
=&2^{k_1}\Big|\iint_{\R^2}\p_\eta(\psi_{[[k_2]]}(\eta-\zeta)\chi^I_{high}(\eta))
\psi_{[[k_3]]}(\zeta)\eta^{-1}d\eta d\zeta\Big|\\
\ls&\frac{t}{2^{j+2k_1}}\Big\{2^{k_3}\int_{\R}|\p_\eta(\chi^I_{high}(\eta))|d\eta
+\iint_{\R^2}|\p_\eta(\psi_{[[k_2]]}(\eta-\zeta))|\psi_{[[k_3]]}(\zeta)d\eta d\zeta\Big\}\\
\ls&\frac{t}{2^{j+2k_1}}\Big\{t2^{k_3-j-3k_1}\int_{\R}|\chi'\Big(\frac{t|\eta|}{2^{j+3k_1+M_1}}\Big)|d\eta
+\iint_{\R^2}|\p_\eta(\psi_{[[k_2]]}(\eta-\zeta))|\psi_{[[k_3]]}(\zeta)d\eta d\zeta\Big\}\\
\ls&\frac{t2^{k_3}}{2^{j+2k_1}}.
\end{split}
\end{equation}
This, together with \eqref{InCone13}, \eqref{InCone14}, \eqref{symbol:m:bdd} and the condition $j\ge\max\{j_1,j_2,j_3\}+O(1)$,
yields
\begin{equation}\label{InCone16}
|K^I_{high}|\ls2^{21k_1-2j/3}t^{-1}(1+|x-x_1|+|x-x_2|+|x-x_3|)^{-10/3}.
\end{equation}
Next, we turn to the estimate of $K^I_{low}$.
For $\ds\eta\in\supp\chi^I_{low}$, one has $|\eta|\ls2^jt^{-1}$ and $\frac{|\p^2_\xi\Psi_2|^5}{(\p_\xi\Psi_2)^{10}}
\ls2^{-5j}$.
Thus, we can get from \eqref{InCone8} and \eqref{symbol:m:bdd} that
\begin{equation}\label{InCone16low}
\begin{split}
|K^I_{low}|&\ls\iiint_{\R^3}\chi^I_{low}(\eta)|(\cL_2^*)^5[m_{++-}
\psi_k(\xi)\psi_{[[k_1]]}(\xi-\eta)\psi_{[[k_2]]}(\eta-\zeta)
\psi_{[[k_3]]}(\zeta)]|d\xi d\eta d\zeta\\
&\ls2^{k_1-5j}\sum_{l=0}^{5}\iiint_{\R^3}|\p_\xi^l(\psi_k(\xi)\psi_{[[k_1]]}(\xi-\eta))|
\psi_{[[k_3]]}(\zeta)\chi^I_{low}(\eta)d\xi d\eta d\zeta\\
&\ls2^{3k_1-4j}t^{-1}\ls2^{3k_1-2j/3}t^{-1}(1+|x-x_1|+|x-x_2|+|x-x_3|)^{-10/3}.
\end{split}
\end{equation}

\vskip 0.2 true cm

\textbf{$(B_1)$ Estimate of $K^I_{med}$}

\vskip 0.1 true cm

Set
\begin{equation}\label{tildecL2:def}
\tilde\cL_2=-i(\p_\zeta\Psi_2)^{-1}\p_\zeta,
\qquad\tilde\cL_2^*=\p_\zeta\Big(\frac{i\cdot}{\p_\zeta\Psi_2}\Big).
\end{equation}
Then $\ds\tilde\cL_2e^{i\Psi_2}=e^{i\Psi_2}$.

The condition of $j\ge\max\{j_2,j_3\}+3k_1+2M_1$ and \eqref{InCone5}, \eqref{InCone6}, \eqref{InCone8}
with $M_1>0$ large enough ensure that
\begin{equation*}
|\p_\zeta\Psi_2|\gt2^{-3k_1}t|\eta|\gt2^{j-3k_1},\quad\eta\in\supp\chi^I_{med}.
\end{equation*}
Note that analogously to \eqref{InCone12}-\eqref{InCone16low}, one has
\begin{equation}\label{InCone17}
\begin{split}
&|\p^l_\zeta\Psi_2|=|t\eta\Lambda^{(l+1)}(\zeta-\tilde r_2\eta)|\ls t|\eta|,
\quad\tilde r_2\in[0,1],\quad l\ge2,\\
&\frac{|\p^2_\zeta\Psi_2|^l}{|\p_\zeta\Psi_2|^{l+5}}
\ls\frac{t^l|\eta|^l}{(\p_\zeta\Psi_2)^{l+5}}\ls2^{30k_1-5j},\quad l=0,\cdots,5,
\end{split}
\end{equation}
and
\begin{equation}\label{InCone18}
\begin{split}
|K^I_{med}|&\ls\iiint_{\R^3}\chi^I_{med}(\eta)|(\tilde\cL_2^*)^5[m_{++-}
\psi_k(\xi)\psi_{[[k_1]]}(\xi-\eta)\psi_{[[k_2]]}(\eta-\zeta)
\psi_{[[k_3]]}(\zeta)]|d\xi d\eta d\zeta\\
&\ls2^{31k_1-5j}\sum_{l=0}^{5}
\iiint_{\R^3}|\p_\zeta^l(\psi_{[[k_2]]}(\eta-\zeta)\psi_{[[k_3]]}(\zeta))|
\psi_{[[k_1]]}(\xi-\eta)\chi^I_{med}(\eta)d\xi d\eta d\zeta\\
&\ls2^{36k_1-2j/3}t^{-1}(1+|x-x_1|+|x-x_2|+|x-x_3|)^{-10/3}.
\end{split}
\end{equation}
Thus, combining \eqref{InCone10}, \eqref{InCone16}, \eqref{InCone16low}, \eqref{InCone18} with $2N\ge N_1+37$ implies
\begin{equation}\label{InCone19}
\begin{split}
&\sum_{\substack{j,j_1,j_2,j_3,k\ge-1,\\j\ge\max\{j_1,j_2,j_3\}+O(1),\\j\ge\max\{j_2,j_3\}+3k_3+2M_1}}
\sum_{\substack{(k_1,k_2,k_3)\in\cY_k,\\k_3-O(1)\le k_1\le k_3}}
2^{j\alpha+N_1k}\|I_{kk_1k_2k_3}^{jj_1j_2j_3}\id_{I_{in}}(t)\|_{L^2(\R)}\\
&\ls\sum_{\substack{j,k_1,k_2,k_3\ge-1,\\|k_1-k_3|\le O(1)}}
2^{k_1(N_1+36)-j/6}t^{-1}(2+j)^3\|P_{k_1}V\|_{L^2}\|P_{k_2}V\|_{L^2}
\|P_{k_3}V\|_{L^2}\\
&\ls\ve_1^3(1+t)^{-1}.
\end{split}
\end{equation}
Finally, collecting \eqref{InCone3}, \eqref{InCone9} and \eqref{InCone19} leads to
\begin{equation}\label{InCone20}
\begin{split}
\sum_{j,j_1,j_2,j_3,k\ge-1}
\sum_{\substack{(k_1,k_2,k_3)\in\cY_k,\\k_3-O(1)\le k_1\le k_3}}2^{j\alpha+N_1k}
\|I_{kk_1k_2k_3}^{jj_1j_2j_3}\id_{I_{in}}(t)\|_{L^2(\R)}
\ls\ve_1^3(1+t)^{-2\alpha},
\end{split}
\end{equation}
which finishes the proof of \eqref{InCone} for Case 1 and $t\ge1$.

\vskip 0.2cm

\noindent\textbf{Case 2. $k_1\ge k_3$}

\vskip 0.1cm

For $\max\{k_2,k_3\}\ge k_1-O(1)$, since the related treatment is analogous to that
in Case 1, the related details are omitted.

Next, we deal with the case of  $\max\{k_2,k_3\}\le k_1-O(1)$.
At this time, $|k-k_1|\le O(1)$ and $\med\{k_1,k_2,k_3\}=\max\{k_2,k_3\}$ hold.
Similarly to Case 1, we now analyze the critical points of $\Psi_2$ in \eqref{InCone1}.
If $j\ge j_1+O(1)$, then one has $|x-x_1|\approx2^j$.
On the other hand, it holds that
\begin{equation*}
|\eta|\le|\zeta-\eta|+|\zeta|\ls2^{\max\{k_2,k_3\}}\ll|\xi|\approx2^{k_1}.
\end{equation*}
This, together with \eqref{InCone4}, yields
\begin{equation}\label{InCone21}
|\p_\xi\Phi|\approx2^{-3k_1}|\eta|.
\end{equation}
In addition, \eqref{InCone4} and $|\zeta-r_2\eta|=|r_2(\zeta-\eta)+(1-r_2)\zeta|\ls2^{\max\{k_2,k_3\}}$ show that
\begin{equation}\label{InCone22}
2^{-3\max\{k_2,k_3\}}|\eta|\ls|\p_\zeta\Phi|\ls|\eta|.
\end{equation}
As in Case 1 with \eqref{InCone21} and \eqref{InCone22} instead of \eqref{InCone5},
we next discuss the frequency $\eta$ so that  the kernel $K_2$ in \eqref{InCone1} can be estimated.
For the low frequency mode $t|\eta|2^{-3k_1}\ll2^j$, one has $|\p_\xi\Psi_2|\ge|x-x_1|-t|\p_\xi\Phi|$,
which implies that there is no critical point for $\Psi_2$.
For the high frequency mode $t|\eta|2^{-3k_1}\gt2^j$, \eqref{InCone22} shows that the critical points of $\Psi_2$ are
contained in the scope of $\max\{j_2,j_3\}\ge j+3k_1-3\max\{k_2,k_3\}-O(1)$.
Based on this, we write
\begin{equation}\label{InCone23}
\begin{split}
&K_2=K^{II}_{high}+K^{II}_{low},\\
&K^{II}_{high}=\iiint_{\R^3}\chi^{II}_{high}(\eta)e^{i\Psi_2}m_{++-}\psi_k(\xi)
\psi_{[[k_1]]}(\xi-\eta)\psi_{[[k_2]]}(\eta-\zeta)
\psi_{[[k_3]]}(\zeta)d\xi d\eta d\zeta,\\
&K^{II}_{low}=\iiint_{\R^3}\chi^{II}_{low}(\eta)e^{i\Psi_2}m_{++-}\psi_k(\xi)
\psi_{[[k_1]]}(\xi-\eta)\psi_{[[k_2]]}(\eta-\zeta)
\psi_{[[k_3]]}(\zeta)d\xi d\eta d\zeta,
\end{split}
\end{equation}
where
\begin{equation*}
\chi^{II}_{high}(\eta)=\chi\Big(\frac{t|\eta|}{2^{j+3k_1-M_2}}\Big),\quad
\chi^{II}_{low}(\eta)=1-\chi\Big(\frac{t|\eta|}{2^{j+3k_1-M_2}}\Big),
\end{equation*}
$\chi$ is defined by \eqref{cutoff:def}, and $M_2>0$ is a fixed and large enough number.
Then one has
\begin{equation}\label{InCone24}
\supp\chi^{II}_{high}\subset\{t|\eta|\ge2^{j+3k_1-M_2}\},\quad
\supp\chi^{II}_{low}\subset\{t|\eta|\le2^{j+3k_1-M_2+1}\}.
\end{equation}

\vskip 0.2cm

\noindent\textbf{Case 2.1. $j\ge j_1+O(1)$ and
$\max\{j_2,j_3\}\le j+3k_1-3\max\{k_2,k_3\}-2M_2$}
\vskip 0.2cm

\textbf{$(A_2)$ Estimate of $K^{II}_{high}$}

\vskip 0.1cm

For $\eta\in\supp\chi^{II}_{high}$, the condition of $\max\{j_2,j_3\}\le j+3k_1-3\max\{k_2,k_3\}-2M_2$,
\eqref{InCone22} and \eqref{InCone24} ensure
\begin{equation*}
t|\p_\zeta\Phi|\gt2^{-3\max\{k_2,k_3\}}t|\eta|\gt2^{j+3k_1-3\max\{k_2,k_3\}-M_2}
\gt2^{\max\{j_2,j_3\}+M_2}.
\end{equation*}
This, together with \eqref{InCone6} and large $M_2>0$, leads to
\begin{equation}\label{InCone25}
|\p_\zeta\Psi_2|\gt t|\p_\zeta\Phi|
\gt\max\{2^{-3\max\{k_2,k_3\}}t|\eta|,2^{j+3k_1-3\max\{k_2,k_3\}}\},
\quad\eta\in\supp\chi^{II}_{high}.
\end{equation}
On the other hand, one has
\begin{equation}\label{InCone26}
1+|x-x_1|+|x-x_2|+|x-x_3|\ls2^{\max\{j,j_2,j_3\}}\ls2^{j+3k_1-3\max\{k_2,k_3\}}.
\end{equation}
It follows from the first line of \eqref{InCone17} and \eqref{InCone25} that
\begin{equation*}
\frac{|\p^2_\zeta\Psi_2|^5}{(\p_\zeta\Psi_2)^{10}}
\ls\frac{t^5|\eta|^5}{(\p_\zeta\Psi_2)^{10}}
\ls2^{30\max\{k_2,k_3\}-3j-9k_1}t^{-2}\eta^{-2}.
\end{equation*}
As in Case 1.2.2, we can achieve
\begin{equation}\label{InCone27}
\begin{split}
|K^{II}_{high}|&\ls\iiint_{\R^3}\chi^{II}_{high}(\eta)|(\tilde\cL_2^*)^5[m_{++-}
\psi_k(\xi)\psi_{[[k_1]]}(\xi-\eta)\psi_{[[k_2]]}(\eta-\zeta)
\psi_{[[k_3]]}(\zeta)]|d\xi d\eta d\zeta\\
&\ls2^{31\max\{k_2,k_3\}-3j-9k_1}t^{-2}\sum_{l=0}^{5}
\iiint_{\R^3}|\p_\zeta^l(\psi_{[[k_2]]}(\eta-\zeta)\psi_{[[k_3]]}(\zeta))|\\
&\hspace{6cm}\times\psi_{[[k_1]]}(\xi-\eta)\chi^{II}_{high}(\eta)
\eta^{-2}d\xi d\eta d\zeta,\\
&\ls2^{32\max\{k_2,k_3\}-4j-11k_1}t^{-1}\\
&\ls2^{21\max\{k_2,k_3\}-2j/3}t^{-1}(1+|x-x_1|+|x-x_2|+|x-x_3|)^{-10/3}.
\end{split}
\end{equation}
where $\tilde\cL_2$ is defined by \eqref{tildecL2:def} and \eqref{InCone26} is used.

\vskip 0.2 true cm

\textbf{$(B_2)$ Estimate of $K^{II}_{low}$}

\vskip 0.1 true cm

By \eqref{InCone6}, \eqref{InCone21} and \eqref{InCone24}, we have
\begin{equation}\label{YH-2}
|\p_\xi\Psi_2|\gt\max\{2^j,t|\eta|2^{-3k_1}\},\qquad\eta\in\supp\chi^{II}_{low}.
\end{equation}

In addition, one has from \eqref{InCone4} and \eqref{InCone24} that
\begin{equation}\label{YH-3}
|\p^l_\xi\Psi_2|=|t\p^l_\xi\Phi|=|t\eta\Lambda^{(l+1)}(\xi-\tilde r_1\eta)|
\ls2^{-(l+2)k_1}t|\eta|\ls|\p_\xi\Psi_2|,\quad l\ge2.
\end{equation}
Based on \eqref{YH-2}-\eqref{YH-3}, we conclude from \eqref{InCone24} that
\begin{equation}\label{InCone28}
\begin{split}
|K^{II}_{low}|&\ls\iiint_{\R^3}\chi^{II}_{low}(\eta)|(\cL_2^*)^5[m_{++-}
\psi_k(\xi)\psi_{[[k_1]]}(\xi-\eta)\psi_{[[k_2]]}(\eta-\zeta)
\psi_{[[k_3]]}(\zeta)]|d\xi d\eta d\zeta\\
&\ls2^{\max\{k_2,k_3\}-5j}\sum_{l=0}^{5}\iiint_{\R^3}|\p_\xi^l(\psi_k(\xi)\psi_{[[k_1]]}(\xi-\eta))|
\psi_{[[k_3]]}(\zeta)\chi^I_{low}(\eta)d\xi d\eta d\zeta\\
&\ls2^{2\max\{k_2,k_3\}+4k_1-4j}t^{-1}\\
&\ls2^{14k_1-2j/3}t^{-1}(1+|x-x_1|+|x-x_2|+|x-x_3|)^{-10/3},
\end{split}
\end{equation}
where $\cL_2$ is defined by \eqref{cL2:def} and \eqref{InCone26} is  used.
Combining \eqref{InCone27} and \eqref{InCone28} with $N\ge N_1+15$ yields
\begin{equation}\label{InCone29}
\begin{split}
&\sum_{\substack{j,j_1,j_2,j_3,k\ge-1,\\j\ge j_1+O(1),\\
\max\{j_2,j_3\}\le j+3k_1-3\max\{k_2,k_3\}-2M_2}}
\sum_{\substack{(k_1,k_2,k_3)\in\cY_k,\\k_1\ge k_3,\max\{k_2,k_3\}\le k_1-O(1)}}
2^{j\alpha+N_1k}\|I_{kk_1k_2k_3}^{jj_1j_2j_3}\id_{I_{in}}(t)\|_{L^2(\R)}\\
&\ls\sum_{j,k_1,k_2,k_3\ge-1}2^{k_1(N_1+14)+8\max\{k_2,k_3\}-j/6}
t^{-1}(5+j+k_1)^3\|P_{k_1}V\|_{L^2}\|P_{k_2}V\|_{L^2}\|P_{k_3}V\|_{L^2}\\
&\ls\ve_1^3(1+t)^{-1}.
\end{split}
\end{equation}

\vskip 0.2 true cm

\noindent\textbf{Case 2.2. $j\le j_1+O(1)$ or $\max\{j_2,j_3\}\ge j+3k_1-3\max\{k_2,k_3\}-2M_2$}
\vskip 0.2 true cm

\noindent\textbf{Case 2.2.1. $\max\{j_2,j_3\}\ge j+3k_1-3\max\{k_2,k_3\}-2M_2$}

\vskip 0.1 true cm

Without loss of generality, $\max\{j_2,j_3\}=j_2$ is assumed.
When $\alpha=1/2$, by the assumption \eqref{thm1BA1} of $\|Q_{j_2}P_{k_2}V\|_{L^2}$,
the produced factor $2^{-j_2/2}$ will provide the number $2^{-j/2}$ with an additional $2^{-3k_1/2}$
regularity. This can compensate the loss of regularity which is
caused by $\|e^{it\Lambda}P_{[[k_1]]}\cV_1\|_{L^\infty}$ and \eqref{loc:disp}.

Similarly to \eqref{InCone2} and \eqref{InCone3}, from \eqref{loc:disp} with $\beta=1/2$, \eqref{thm1BA1}
and \eqref{trilin} with $N_1\ge10$, one has
\begin{equation}\label{InCone30}
\begin{split}
&\sum_{\substack{j,j_1,j_2,j_3,k\ge-1,\\
\max\{j_2,j_3\}\ge j+3k_1-3\max\{k_2,k_3\}-2M_2}}
\sum_{\substack{(k_1,k_2,k_3)\in\cY_k,\\k_1\ge k_3,\max\{k_2,k_3\}\le k_1-O(1)}}
2^{j/2+N_1k}\|I_{kk_1k_2k_3}^{jj_1j_2j_3}\id_{I_{in}}(t)\|_{L^2(\R)}\\
&\ls\sum_{\substack{j_1,j_2,j_3\ge-1,\\k_1,k_2,k_3\ge-1}}
\sum_{j\le j_2-3k_1+3\max\{k_2,k_3\}+2M_2}2^{7\max\{k_2,k_3\}+k_1N_1+j/2}
\|e^{it\Lambda}P_{[[k_1]]}\cV_1\|_{L^\infty}\\
&\hspace{5cm}\times\|Q_{j_2}P_{k_2}V\|_{L^2}
\|e^{-it\Lambda}P_{[[k_3]]}\cV_3\|_{L^\infty}\\
&\ls\sum_{\substack{j_1,j_2,j_3\ge-1,\\k_1,k_2,k_3\ge-1}}
2^{17\max\{k_2,k_3\}/2+k_1(N_1-3/2)+j_2/2}
\|e^{it\Lambda}P_{[[k_1]]}\cV_1\|_{L^\infty}
\|e^{-it\Lambda}P_{[[k_3]]}\cV_3\|_{L^\infty}\|Q_{j_2}P_{k_2}V\|_{L^2}\\
&\ls\sum_{\substack{j_1,j_2,j_3\ge-1,\\k_1,k_2,k_3\ge-1}}
2^{k_1N_1+10\max\{k_2,k_3\}+(j_1+j_2+j_3)/2}(1+t)^{-1}\|Q_{j_1}P_{k_1}V\|_{L^2}
\|Q_{j_2}P_{k_2}V\|_{L^2}\|Q_{j_3}P_{k_3}V\|_{L^2}\\
&\ls\ve_1^3(1+t)^{-1}.
\end{split}
\end{equation}

When $\alpha\in(0,1/2)$, instead of \eqref{InCone30}, applying \eqref{loc:disp} to $P_{k_3}V_-$
with $\beta=\alpha$ , \eqref{disp:Lp} to $P_{k_1}V$ with $p=2/(1-2\alpha)$ and the Bernstein
inequality to $P_{[[k_2]]}Q_{j_2}P_{k_2}V$ leads to
\begin{equation}\label{InCone31}
\begin{split}
&\sum_{\substack{j,j_1,j_2,j_3,k\ge-1,\\
\max\{j_2,j_3\}\ge j+3k_1-3\max\{k_2,k_3\}-2M_2}}
\sum_{\substack{(k_1,k_2,k_3)\in\cY_k,\\k_1\ge k_3,\max\{k_2,k_3\}\le k_1-O(1)}}
2^{j\alpha+N_1k}\|I_{kk_1k_2k_3}^{jj_1j_2j_3}\id_{I_{in}}(t)\|_{L^2(\R)}\\
&\ls\sum_{\substack{j_1,j_2,j_3\ge-1,\\k_1,k_2,k_3\ge-1}}
\sum_{j\le j_2-3k_1+3\max\{k_2,k_3\}+2M_2}2^{7\max\{k_2,k_3\}+k_1N_1+j\alpha}
\|e^{it\Lambda}P_{[[k_1]]}Q_{j_1}P_{k_1}V\|_{L^{2/(1-2\alpha)}}\\
&\hspace{5cm}\times\|e^{it\Lambda}P_{[[k_2]]}Q_{j_2}P_{k_2}V\|_{L^{1/\alpha}}
\|e^{-it\Lambda}P_{[[k_3]]}Q_{j_3}P_{k_3}V_-\|_{L^\infty}\\
&\ls\sum_{\substack{j_1,j_2,j_3\ge-1,\\k_1,k_2,k_3\ge-1}}
2^{17\max\{k_2,k_3\}/2+k_1(N_1-3\alpha)+j_2\alpha+k_2/2}
\|e^{it\Lambda}P_{[[k_1]]}Q_{j_1}P_{k_1}V\|_{L^{2/(1-2\alpha)}}\\
&\hspace{5cm}\times\|e^{it\Lambda}P_{[[k_2]]}Q_{j_2}P_{k_2}V\|_{L^2}
\|e^{-it\Lambda}P_{[[k_3]]}Q_{j_3}P_{k_3}V_-\|_{L^\infty}\\
&\ls t^{-2\alpha}\sum_{\substack{j_1,j_2,j_3\ge-1,\\k_1,k_2,k_3\ge-1}}
2^{k_1N_1+11\max\{k_2,k_3\}+(j_1+j_2+j_3)\alpha}\|Q_{j_1}P_{k_1}V\|_{L^2}
\|Q_{j_2}P_{k_2}V\|_{L^2}\|Q_{j_3}P_{k_3}V\|_{L^2}\\
&\ls\ve_1^3(1+t)^{-2\alpha},
\end{split}
\end{equation}
where  \eqref{thm1BA1} is used.

\vskip 0.2 true cm

\noindent\textbf{Case 2.2.2. $j\le j_1+O(1)$}

\vskip 0.1 true cm

Analogously to Case 2.2.1, by utilizing \eqref{loc:disp} with $\beta=\alpha$, one can achieve
\begin{equation*}\label{InCone32}
\begin{split}
&\sum_{\substack{j,j_1,j_2,j_3,k\ge-1,\\j\le j_1+O(1)}}
\sum_{\substack{(k_1,k_2,k_3)\in\cY_k,\\k_1\ge k_3,\max\{k_2,k_3\}\le k_1-O(1)}}
2^{j\alpha+N_1k}\|I_{kk_1k_2k_3}^{jj_1j_2j_3}\id_{I_{in}}(t)\|_{L^2(\R)}\\
\end{split}
\end{equation*}

\begin{equation}\label{InCone32}
\begin{split}
&\ls\sum_{\substack{j_1,j_2,j_3\ge-1,\\k_1,k_2,k_3\ge-1}}
\sum_{j\le j_1+O(1)}2^{7\max\{k_2,k_3\}+k_1N_1+j\alpha}\|Q_{j_1}P_{k_1}V\|_{L^2}
\|e^{it\Lambda}P_{[[k_2]]}\cV_2\|_{L^\infty}
\|e^{-it\Lambda}P_{[[k_3]]}\cV_3\|_{L^\infty}\\
&\ls\sum_{\substack{j_1,j_2,j_3\ge-1,\\k_1,k_2,k_3\ge-1}}
2^{7\max\{k_2,k_3\}+k_1N_1+j_1\alpha}\|Q_{j_1}P_{k_1}V\|_{L^2}
\|e^{it\Lambda}P_{[[k_2]]}\cV_2\|_{L^\infty}
\|e^{-it\Lambda}P_{[[k_3]]}\cV_3\|_{L^\infty}\\
&\ls(1+t)^{-2\alpha}\sum_{\substack{j_1,j_2,j_3\ge-1,\\k_1,k_2,k_3\ge-1}}
2^{k_1N_1+10\max\{k_2,k_3\}+(j_1+j_2+j_3)\alpha}\|Q_{j_1}P_{k_1}V\|_{L^2}
\|Q_{j_2}P_{k_2}V\|_{L^2}\|Q_{j_3}P_{k_3}V\|_{L^2}\\
&\ls\ve_1^3(1+t)^{-2\alpha}.
\end{split}
\end{equation}
Collecting \eqref{InCone20} and \eqref{InCone29}-\eqref{InCone32} implies \eqref{InCone} for $t\ge1$.

\vskip 0.1cm

At last, we turn to the proof of \eqref{InCone} for $t\le1$.
For $t\le1$, note that $j\le\log_2(1+t)+O(1)\le O(1)$.
Then the related treatments are similar to those in Case 1.1 \eqref{InCone3} and Case 2.2.2 \eqref{InCone32}, respectively.
This completes the proof of \eqref{InCone}.

\end{proof}

\subsection{Estimates of the quartic and higher order nonlinearities}

\begin{lemma}\label{lem:quartic}
Under the bootstrap assumption \eqref{thm1BA1}, it holds that for $\alpha\in(0,1/2]$ and $t\ge0$,
\begin{equation}\label{quartic}
\sum_{j,k\ge-1}2^{j\alpha+N_1k}\Big(\|Q_j\cQ_k(t)\|_{L^2(\R)}+
\|Q_jP_ke^{-it\Lambda}\cN^I_4(U)\|_{L^2(\R)}\Big)
\ls\ve_1^4(1+t)^{-2\alpha}.
\end{equation}
\end{lemma}
\begin{proof}
Set
\begin{equation}\label{quartic1}
\cQ^I_k=\sum_{\substack{(k_1,k_2,k_3)\in\cY_k,\\
(\mu_1,\mu_2,\mu_3)\in A_\Phi^{good}}}
e^{-it\Lambda}P_kT_{\Phi_{\mu_1\mu_2\mu_3}^{-1}m_{\mu_1\mu_2\mu_3}}
(P_{k_1}U_{\mu_1},P_{k_2}U_{\mu_2},e^{it\mu_3\Lambda}P_{k_3}\p_tV_{\mu_3}),
\end{equation}
which comes from the third term in the expression of $\cQ_k$.

Substituting \eqref{profile:eqn5} into \eqref{quartic1} yields
\begin{equation}\label{quartic2}
\cQ^I_k=\cQ^{II}_k+\cN_{5,k}(U),
\end{equation}
where
\begin{equation*}\label{quartic3}
\begin{split}
\cQ^{II}_k&=\sum_{\substack{(k_1,k_2,k_3)\in\cY_k,\\
(\mu_1,\mu_2,\mu_3)\in A_\Phi^{good}}}
\sum_{\substack{(k_4,k_5)\in\cX_{k_3},\\ \nu_1,\nu_2=\pm}}
e^{-it\Lambda}P_kT_{\Phi_{\mu_1\mu_2\mu_3}^{-1}m_{\mu_1\mu_2\mu_3}}
(P_{k_1}U_{\mu_1},P_{k_2}U_{\mu_2},\\
&\hspace{6cm}P_{k_3}T_{a^I_{\mu_3\nu_1\nu_2}}(P_{k_4}U_{\nu_1},P_{k_5}U_{\nu_2}))\\
\end{split}
\end{equation*}

\begin{equation}\label{quartic3}
\begin{split}
&=\sum_{\substack{(k_1,k_2,k_3)\in\cY_k,\\
(\mu_1,\mu_2,\mu_3)\in A_\Phi^{good}}}
\sum_{\substack{(k_4,k_5)\in\cX_{k_3},\\ \nu_1,\nu_2=\pm}}\sum_{j_1,j_2,j_3,j_4\ge-1}
e^{-it\Lambda}P_kT_{\Phi_{\mu_1\mu_2\mu_3}^{-1}m_{\mu_1\mu_2\mu_3}}
(P_{[[k_1]]}e^{it\mu_1\Lambda}Q_{j_1}P_{k_1}V_{\mu_1},\\
&\quad P_{[[k_2]]}e^{it\mu_2\Lambda}Q_{j_2}P_{k_2}V_{\mu_2},
P_{k_3}T_{a^I_{\mu_3\nu_1\nu_2}}(P_{[[k_4]]}e^{it\nu_1\Lambda}Q_{j_3}P_{k_4}V_{\nu_1},
P_{[[k_5]]}e^{it\nu_2\Lambda}Q_{j_4}P_{k_5}V_{\nu_2})),
\end{split}
\end{equation}

\begin{equation}\label{quartic4}
\cN_{5,k}(U)=\sum_{\substack{(k_1,k_2,k_3)\in\cY_k,\\
(\mu_1,\mu_2,\mu_3)\in A_\Phi^{good}}}
e^{-it\Lambda}P_kT_{\Phi_{\mu_1\mu_2\mu_3}^{-1}m_{\mu_1\mu_2\mu_3}}
(P_{k_1}U_{\mu_1},P_{k_2}U_{\mu_2},P_{k_3}\cN_{3,\mu_3}(U)),
\end{equation}
and $\cN_{3,\mu_3}(U)$ is defined by \eqref{cN3:def}.
Let
\begin{equation}\label{quartic5}
\begin{split}
&\cQ_q:=Q_jP_ke^{-it\Lambda}T_{\Phi_{\mu_1\mu_2\mu_3}^{-1}m_{\mu_1\mu_2\mu_3}}
(P_{[[k_1]]}e^{it\mu_1\Lambda}\sV_1,P_{[[k_2]]}e^{it\mu_2\Lambda}\sV_2,\\
&\hspace{3cm}P_{k_3}T_{a^I_{\mu_3\nu_1\nu_2}}(P_{[[k_4]]}e^{it\nu_1\Lambda}\sV_3,
P_{[[k_5]]}e^{it\nu_2\Lambda}\sV_4)),\\
&\sV_1:=Q_{j_1}P_{k_1}V_{\mu_1},\sV_2:=Q_{j_2}P_{k_2}V_{\mu_2},
\sV_3:=Q_{j_3}P_{k_4}V_{\nu_1},\sV_4:=Q_{j_4}P_{k_5}V_{\nu_2}.
\end{split}
\end{equation}

Analogous to the estimates in Lemmas \ref{lem:outcone} and \ref{lem:incone} for the cubic nonlinearity $\cC_k(s)$,
the proof of \eqref{quartic} will be also separated into two cases.

\vskip 0.2cm

\noindent\textbf{Case 1.} ${\max\{j,j_1,j_2,j_3,j_4\}}\le\log_2(1+t)+O(1)$

\vskip 0.1cm

Comparing to Lemma \ref{lem:incone}, the appeared factor $2^{j\alpha}$ in this case can
be controlled by the additional $(1+t)^{-\alpha}$ decay, which is produced by the quartic nonlinearity.
In addition, due to $(k_1,k_2,k_3)\in\cY_k$ and $(k_4,k_5)\in\cX_{k_3}$, one can see that
$2^k\ls2^{\max\{k_1,k_2,k_3\}}$ and $2^{k_3}\ls2^{\max\{k_4,k_5\}}$ hold. Next
we treat $\cQ_q$ according to the differences of frequencies.

\vskip 0.2cm

\noindent\textbf{Case 1.1.} ${\max\{k_1,k_2,k_3\}}=k_1$

\vskip 0.1cm

In this case, $\med\{k_1,k_2,k_3\}=\max\{k_2,k_3\}$.
Applying \eqref{bilinear:b} and \eqref{trilin:good}, one then has
\begin{equation}\label{quartic6}
\begin{split}
\|\cQ_q\|_{L^2(\R)}&\ls2^{8\max\{k_2,k_3\}}\|\sV_1\|_{L^2}
\|P_{[[k_2]]}e^{it\mu_2\Lambda}\sV_2\|_{L^\infty}
\|T_{a^I_{\mu_3\nu_1\nu_2}}(P_{[[k_4]]}e^{it\nu_1\Lambda}\sV_3,
P_{[[k_5]]}e^{it\nu_2\Lambda}\sV_4)\|_{L^\infty}\\
&\ls2^{8\max\{k_2,k_3\}}\|\sV_1\|_{L^2}
\|P_{[[k_2]]}e^{it\mu_2\Lambda}\sV_2\|_{L^\infty}
\|P_{[[k_4]]}e^{it\nu_1\Lambda}\sV_3\|_{L^\infty}
\|P_{[[k_5]]}e^{it\nu_2\Lambda}\sV_4\|_{L^\infty}.
\end{split}
\end{equation}
Therefore, it can be deduced from \eqref{loc:disp} with $\beta=\alpha$, \eqref{thm1BA1} and \eqref{quartic6} that
\begin{equation}\label{quartic7}
\begin{split}
&\sum_{\substack{(k_1,k_2,k_3)\in\cY_k,\\(\mu_1,\mu_2,\mu_3)\in A_\Phi^{good},\\
{\max\{k_1,k_2,k_3\}}=k_1}}
\sum_{\substack{(k_4,k_5)\in\cX_{k_3},\\ \nu_1,\nu_2=\pm}}
\sum_{\substack{j_1,j_2,j_3,j_4\ge-1,\\j,k\ge-1}}
2^{j\alpha+N_1k}\|\cQ_q\id_{I_{in4}}(t)\|_{L^2(\R)}\\
&\ls\sum_{\substack{(k_1,k_2,k_3)\in\cY_k,\\(k_4,k_5)\in\cX_{k_3},\\
j_1,j_2,j_3,j_4\ge-1}}
\sum_{\substack{(\mu_1,\mu_2,\mu_3)\in A_\Phi^{good},\\ \nu_1,\nu_2=\pm}}
2^{k_1N_1}(1+t)^{\alpha}\|\cQ_q\id_{I_{in4}}(t)\|_{L^2(\R)}\\
&\ls(1+t)^{-2\alpha}\sum_{\substack{k_1,k_2,k_4,k_5\ge-1,\\j_1,j_2,j_3,j_4\ge-1}}
2^{k_1N_1+8\max\{k_2,k_4,k_5\}+2(k_2+k_4+k_5)+\alpha(j_2+j_3+j_4)}\\
&\hspace{3cm}\times\|Q_{j_1}P_{k_1}V\|_{L^2}\|Q_{j_2}P_{k_2}V\|_{L^2}
\|Q_{j_3}P_{k_4}V\|_{L^2}\|Q_{j_4}P_{k_5}V\|_{L^2}\\
&\ls\ve_1^4(1+t)^{-2\alpha},
\end{split}
\end{equation}
where $I_{in4}:=\{t\ge0: {\max\{j,j_1,j_2,j_3,j_4\}}\le\log_2(1+t)+O(1)\}$.

\vskip 0.1cm

\noindent\textbf{Case 1.2.} ${\max\{k_1,k_2,k_3\}}=k_2$

\vskip 0.1cm

Since the related treatment  is similar to that in Case 1.1,  the details are omitted here.

\vskip 0.1cm

\noindent\textbf{Case 1.3.} ${\max\{k_1,k_2,k_3\}}=k_3$

\vskip 0.1cm

In this case, $\med\{k_1,k_2,k_3\}=\max\{k_1,k_2\}$ holds.
For convenience, assume $\max\{k_4,k_5\}=k_5$.
Instead of \eqref{quartic6}, we have
\begin{equation*}
\begin{split}
\|\cQ_q\|_{L^2(\R)}&\ls2^{8\max\{k_1,k_2\}}\|\sV_1\|_{L^\infty}\|\sV_2\|_{L^\infty}
\|T_{a^I_{\mu_3\nu_1\nu_2}}(P_{[[k_4]]}e^{it\nu_1\Lambda}\sV_3,
P_{[[k_5]]}e^{it\nu_2\Lambda}\sV_4)\|_{L^2}\\
&\ls2^{8\max\{k_1,k_2\}}\|P_{[[k_1]]}e^{it\mu_1\Lambda}\sV_1\|_{L^\infty}
\|P_{[[k_2]]}e^{it\mu_2\Lambda}\sV_2\|_{L^\infty}
\|P_{[[k_4]]}e^{it\nu_1\Lambda}\sV_3\|_{L^\infty}\|\sV_4\|_{L^2}.
\end{split}
\end{equation*}
Analogously to \eqref{quartic7}, we can achieve
\begin{equation}\label{quartic8}
\begin{split}
&\sum_{\substack{(k_1,k_2,k_3)\in\cY_k,\\(\mu_1,\mu_2,\mu_3)\in A_\Phi^{good},\\
{\max\{k_1,k_2,k_3\}}=k_3}}
\sum_{\substack{(k_4,k_5)\in\cX_{k_3},\\ \nu_1,\nu_2=\pm}}
\sum_{\substack{j_1,j_2,j_3,j_4\ge-1,\\j,k\ge-1}}
2^{j\alpha+N_1k}\|\cQ_q\id_{I_{in4}}(t)\|_{L^2(\R)}\\
&\ls(1+t)^{-2\alpha}\sum_{\substack{k_1,k_2,k_4,k_5\ge-1,\\j_1,j_2,j_3,j_4\ge-1}}
2^{k_5N_1+8\max\{k_1,k_2,k_4\}+2(k_1+k_2+k_4)+\alpha(j_1+j_2+j_3)}\\
&\hspace{3cm}\times\|Q_{j_1}P_{k_1}V\|_{L^2}\|Q_{j_2}P_{k_2}V\|_{L^2}
\|Q_{j_3}P_{k_4}V\|_{L^2}\|Q_{j_4}P_{k_5}V\|_{L^2}\\
&\ls\ve_1^4(1+t)^{-2\alpha}.
\end{split}
\end{equation}
Collecting \eqref{quartic7} and \eqref{quartic8} yields
\begin{equation}\label{quartic9}
\begin{split}
\sum_{\substack{(k_1,k_2,k_3)\in\cY_k,\\(\mu_1,\mu_2,\mu_3)\in A_\Phi^{good}}}
\sum_{\substack{(k_4,k_5)\in\cX_{k_3},\\ \nu_1,\nu_2=\pm}}
\sum_{\substack{j_1,j_2,j_3,j_4\ge-1,\\j,k\ge-1}}
2^{j\alpha+N_1k}\|\cQ_q\id_{I_{in4}}(t)\|_{L^2(\R)}
\ls\ve_1^4(1+t)^{-2\alpha}.
\end{split}
\end{equation}

\vskip 0.2cm

\noindent\textbf{Case 2.} ${\max\{j,j_1,j_2,j_3,j_4\}}\ge\log_2(1+t)+O(1)$
\vskip 0.1cm

As in Lemma \ref{lem:outcone}, the related treatments will be separated into the following two cases.

\vskip 0.2cm

\noindent\textbf{Case 2.1.} $\ds\max_{l=1,2,3,4}|j-j_l|\le O(1)$
\vskip 0.1cm

In this case, one can take the treatment as in Case 1, where the only difference is
that the appeared factor $2^{j\alpha}$ can be absorbed by $2^{j_1\alpha}$ in \eqref{quartic7} or $2^{j_4\alpha}$ in \eqref{quartic8}.
Then we arrive at
\begin{equation}\label{quartic10}
\sum_{\substack{(k_1,k_2,k_3)\in\cY_k,\\(\mu_1,\mu_2,\mu_3)\in A_\Phi^{good}}}
\sum_{\substack{(k_4,k_5)\in\cX_{k_3},\\ \nu_1,\nu_2=\pm}}
\sum_{\substack{j_1,j_2,j_3,j_4\ge-1,\\ \max_{l=1,2,3,4}|j-j_l|\le O(1)}}
\sum_{j,k\ge-1}2^{j\alpha+N_1k}\|\cQ_q\id_{I_{out4}}(t)\|_{L^2(\R)}
\ls\ve_1^4(1+t)^{-2\alpha},
\end{equation}
where $I_{out4}:=\{t\ge0: {\max\{j,j_1,j_2,j_3,j_4\}}\ge\log_2(1+t)+O(1)\}$.

\vskip 0.2cm

\noindent\textbf{Case 2.2.} $\ds\max_{l=1,2,3,4}|j-j_l|\ge O(1)$
\vskip 0.1cm

Analogously to \eqref{OutCone1}, $I_4$ can be rewritten as
\begin{equation}\label{quartic11}
\begin{split}
&\cQ_q(t,x)=(2\pi)^{-4}\psi_j(x)\int_{\R^4}K_4(x-x_1,x-x_2,x-x_3,x-x_4)
\sV_1(t,x_1)\sV_2(t,x_2)\\
&\hspace{6cm}\times\sV_3(t,x_3)\sV_3(t,x_4)dx_1dx_2dx_3dx_4,\\
\end{split}
\end{equation}
where
\begin{equation}\label{YH-4}
\begin{split}
&K_4(x-x_1,x-x_2,x-x_3,x-x_4):=\int_{\R^4}e^{i\Psi_4}m_4(\xi_1,\xi_2,\xi_3,\xi_4)
d\xi_1d\xi_2d\xi_3d\xi_4,\\
&\Psi_4:=t(-\Lambda(\xi_1+\xi_2+\xi_3+\xi_4)+\mu_1\Lambda(\xi_1)+\mu_2\Lambda(\xi_2)
+\nu_1\Lambda(\xi_3)+\nu_2\Lambda(\xi_4))\\
&\hspace{0.8cm}+\xi_1(x-x_1)+\xi_2(x-x_2)+\xi_3(x-x_3)+\xi_4(x-x_4),\\
&m_4(\xi_1,\xi_2,\xi_3,\xi_4):=(\Phi_{\mu_1\mu_2\mu_3}^{-1}m_{\mu_1\mu_2\mu_3})
(\xi_1,\xi_2,\xi_3+\xi_4)a^I_{\mu_3\nu_1\nu_2}(\xi_3,\xi_4)
\psi_k(\xi_1+\xi_2+\xi_3+\xi_4)\\
&\hspace{3.5cm}\times\psi_{k_3}(\xi_3+\xi_4)\psi_{[[k_1]]}(\xi_1)
\psi_{[[k_2]]}(\xi_2)\psi_{[[k_4]]}(\xi_3)\psi_{[[k_5]]}(\xi_4).
\end{split}
\end{equation}
Denote
\begin{equation*}
\cL_4:=-i(\sum_{l=1}^4|\p_{\xi_l}\Psi_4|^2)^{-1}
\sum_{l=1}^4\p_{\xi_l}\Psi_4\p_{\xi_l}.
\end{equation*}
Then $\cL_4e^{i\Psi_4}=e^{i\Psi_4}$ holds and its adjoint operator $\cL_4^*$ is
\begin{equation*}
\cL_4^*:=i\sum_{l=1}^4\p_{\xi_l}\Big(\frac{\p_{\xi_l}\Psi_4~\cdot}
{\sum_{l=1}^4|\p_{\xi_l}\Psi_4|^2}\Big).
\end{equation*}
The conditions $\max\{j,j_1,j_2,j_3,j_4\}\ge\log_2(1+t)+O(1)$ and $\ds\max_{l=1,2,3,4}|j-j_l|\ge O(1)$ show that when $x\in\supp\psi_j$, $x_l\in\supp\psi_l$, $l=1,2,3,4$, it holds that
\begin{equation*}
\begin{split}
|x-x_1|+|x-x_2|+|x-x_3|+|x-x_4|&\ge2^{O(1)}(1+t),\\
|x-x_1|+|x-x_2|+|x-x_3|+|x-x_4|&\gt2^{\max\{j,j_1,j_2,j_3,j_4\}}.
\end{split}
\end{equation*}
This, together with $|\Lambda'(y)|\le1$, leads to
\begin{equation}\label{quartic12}
\begin{split}
(\sum_{l=1}^4|\p_{\xi_l}\Psi_4|^2)^{1/2}
&\gt|x-x_1|+|x-x_2|+|x-x_3|+|x-x_4|\\
&\gt\max\{1+t,2^{\max\{j,j_1,j_2,j_3,j_4\}}\}.
\end{split}
\end{equation}
On the other hand, one obtains from \eqref{3phase:bdd} and \eqref{quartic11} that for $(\mu_1,\mu_2,\mu_3)\in A_\Phi^{good}$,
\begin{equation}\label{quartic13}
\begin{split}
|\p^l_{\xi_1,\xi_2,\xi_3,\xi_4}\Phi_{\mu_1\mu_2\mu_3}^{-1}(\xi_1,\xi_2,\xi_3+\xi_4)|
&\ls2^{(l+1)\max\{k_1,k_2,k_4,k_5\}},\quad l\ge0,\\
|\p^l_{\xi_1,\xi_2,\xi_3,\xi_4}\Psi_4|&\ls t,\quad l\ge2,
\end{split}
\end{equation}
where $|\xi_1|\approx2^{k_1}$, $|\xi_2|\approx2^{k_2}$, $|\xi_3|\approx2^{k_4}$ and $|\xi_4|\approx2^{k_5}$.

Without loss of generality,  $\max\{k_1,k_2,k_4,k_5\}=k_1$ is assumed.
By the method of stationary phase and \eqref{quartic11}--\eqref{quartic13}, \eqref{symbol:m:bdd}, we have
\begin{equation*}
\begin{split}
&\quad\;|K_4(x-x_1,x-x_2,x-x_3,x-x_4)|\\
&=\Big|\int_{\R^4}\cL_4^8(e^{i\Psi_4})m_4(\xi_1,\xi_2,\xi_3,\xi_4)
d\xi_1d\xi_2d\xi_3d\xi_4\Big|\\
&\ls\int_{\R^4}|(\cL_4^*)^8m_4(\xi_1,\xi_2,\xi_3,\xi_4)|d\xi_1d\xi_2d\xi_3d\xi_4\\
&\ls2^{k_1+k_2+k_4+k_5+10\max\{k_1,k_2,k_4,k_5\}}\Big(1+\sum_{i=1}^4|x-x_i|\Big)^{-8}\\
&\ls2^{11k_1+k_2+k_4+k_5-\max\{j,j_1,j_2,j_3,j_4\}}
(1+t)^{-2}\Big(1+\sum_{i=1}^4|x-x_i|\Big)^{-5}.
\end{split}
\end{equation*}
Similarly to \eqref{OutCone4},
\begin{equation*}
\|\cQ_q(t)\|_{L^2(\R)}
\ls\ve_1^42^{(11-N)(k_1+k_2+k_4+k_5)-5j/9-(j_1+j_2+j_3+j_4)/9}(1+t)^{-2}.
\end{equation*}
This, together with the condition $N\ge N_1+12$, yields
\begin{equation}\label{quartic14}
\sum_{\substack{(k_1,k_2,k_3)\in\cY_k,\\(\mu_1,\mu_2,\mu_3)\in A_\Phi^{good}}}
\sum_{\substack{(k_4,k_5)\in\cX_{k_3},\\ \nu_1,\nu_2=\pm}}
\sum_{\substack{j_1,j_2,j_3,j_4\ge-1,\\ \max_{l=1,2,3,4}|j-j_l|\le O(1)}}
\sum_{j,k\ge-1}2^{j\alpha+N_1k}\|\cQ_q\id_{I_{out4}}(t)\|_{L^2(\R)}
\ls\ve_1^4(1+t)^{-2}.
\end{equation}
Combining \eqref{quartic3}, \eqref{quartic5}, \eqref{quartic9}, \eqref{quartic10} and \eqref{quartic14} leads to
\begin{equation}\label{quartic15}
\sum_{j,k\ge-1}2^{j\alpha+N_1k}\|Q_j\cQ_k^{II}(t)\|_{L^2(\R)}
\ls\ve_1^4(1+t)^{-2\alpha}.
\end{equation}
Note that the estimate \eqref{quartic15} also holds for $\cN_{5,k}(U)$ defined by \eqref{quartic4}
with the first inequality of \eqref{trilin}, here we omit the details.
Thus, we achieve
\begin{equation}\label{quartic16}
\sum_{j,k\ge-1}2^{j\alpha+N_1k}\|Q_j\cQ_k^I(t)\|_{L^2(\R)}
\ls\ve_1^4(1+t)^{-2\alpha}.
\end{equation}
With \eqref{trilinear++-}, one can get the estimate \eqref{quartic16} for the other terms in $\cQ_k$.
The estimate for $P_ke^{-it\Lambda}\cN^I_4(U)$ defined by \eqref{cN4:def} is the same.
Therefore, the proof of \eqref{quartic} is completed.

\end{proof}

\subsection{Estimates of the boundary term $\cB_k$}

\begin{lemma}\label{lem:bdry}
Under the bootstrap assumption \eqref{thm1BA1}, it holds that for $\alpha\in(0,1/2]$ and $t\ge0$,
\begin{equation}\label{bdry}
\sum_{j,k\ge-1}2^{j\alpha+N_1k}\|Q_j\cB_k\|_{L^2(\R)}\ls\ve_1^2.
\end{equation}
\end{lemma}
\begin{proof}
Denote
\begin{equation}\label{bdry1}
\begin{split}
\cB_k^I&:=-i\sum_{\mu_1,\mu_2=\pm}e^{-is\Lambda}P_k
T_{\Phi_{\mu_1\mu_2}^{-1}a_{\mu_1\mu_2}}(U_{\mu_1},U_{\mu_2})\Big|_{s=0}^t,\\
\cB_k^{II}&:=-i\sum_{(\mu_1,\mu_2,\mu_3)\in A_\Phi^{good}}\sum_{(k_1,k_2,k_3)\in\cY_k}
e^{-is\Lambda}P_kT_{\Phi_{\mu_1\mu_2\mu_3}^{-1}m_{\mu_1\mu_2\mu_3}}
(P_{k_1}U_{\mu_1},P_{k_2}U_{\mu_2},P_{k_3}U_{\mu_3})\Big|_{s=0}^t\\
&\qquad-i\sum_{\substack{(k_1,k_2,k_3)\in\cY_k,\\ \max\{k_1,k_2\}\le k_3-O(1)}}
e^{-is\Lambda}P_kT_{\Phi_{++-}^{-1}m_{++-}}(P_{k_1}U,P_{k_2}U,P_{k_3}U_-)\Big|_{s=0}^t.
\end{split}
\end{equation}
Then $\cB_k=\cB_k^I+\cB_k^{II}$.
Next we prove
\begin{equation}\label{bdry2}
\sum_{j,k\ge-1}2^{j\alpha+N_1k}\|Q_j\cB_k^I\|_{L^2(\R)}\ls\ve_1^2.
\end{equation}
By  virtue of \eqref{proj:proj}, one can find that
\begin{equation}\label{bdry3}
\begin{split}
&Q_j\cB_k^I=-i\sum_{j_1,j_2\ge-1}\sum_{(k_1,k_2)\in\cX_k}\cB_{kk_1k_2}^{jj_1j_2},\\
&\cB_{kk_1k_2}^{jj_1j_2}:=Q_jP_ke^{-it\Lambda}T_{\Phi_{\mu\nu}^{-1}a_{\mu\nu}}
(e^{it\mu\Lambda}P_{[[k_1]]}Q_{j_1}P_{k_1}V_\mu,e^{it\nu\Lambda}P_{[[k_2]]}Q_{j_2}P_{k_2}V_\nu).
\end{split}
\end{equation}
The proof of \eqref{bdry2} will be separated into two cases as in Lemma \ref{lem:quartic} and $k_1\ge k_2$ is assumed.

\vskip 0.2cm

\noindent\textbf{Case 1. $\max\{j,j_1,j_2\}\le\log_2(1+t)+O(1)$}

\vskip 0.1cm

It can be concluded from \eqref{loc:disp}, \eqref{thm1BA1} and \eqref{bilinear:a} that
\begin{equation}\label{bdry4}
\begin{split}
&\sum_{j,k\ge-1}2^{j\alpha+N_1k}\|\sum_{\substack{j_1,j_2\ge-1,\\(k_1,k_2)\in\cX_k}}
\sum_{\substack{\max\{j,j_1,j_2\}\le\log_2(1+t)+O(1),\\
\max\{|j-j_1|,|j-j_2|\}\le O(1)}}\cB_{kk_1k_2}^{jj_1j_2}\|_{L^2}\\
\ls& \sum_{j_1,j_2,k_1,k_2\ge-1}2^{N_1k_1+5k_2}(1+t)^{\alpha}
\|Q_{j_1}P_{k_1}V\|_{L^2}\|e^{it\nu\Lambda}P_{[[k_2]]}Q_{j_2}P_{k_2}V_\nu\|_{L^\infty}\\
\ls& \sum_{j_1,j_2,k_1,k_2\ge-1}2^{N_1k_1+13k_2/2+j_2\alpha}\|Q_{j_1}P_{k_1}V\|_{L^2}
\|Q_{j_2}P_{k_2}V\|_{L^2}\\
\ls& \ve_1^2.
\end{split}
\end{equation}

\vskip 0.2cm

\noindent\textbf{Case 2. $\max\{j,j_1,j_2\}\ge\log_2(1+t)+O(1)$}

\vskip 0.2cm

\noindent\textbf{Case 2.1. $\max\{|j-j_1|,|j-j_2|\}\le O(1)$}
\vskip 0.1cm

By the Bernstein inequality, \eqref{thm1BA1} and \eqref{bilinear:a}, one has that
\begin{equation}\label{bdry5}
\begin{split}
&\sum_{j,k\ge-1}2^{j\alpha+N_1k}\|\sum_{\substack{j_1,j_2\ge-1,\\(k_1,k_2)\in\cX_k}}
\sum_{\substack{\max\{j,j_1,j_2\}\ge\log_2(1+t)+O(1),\\
\max\{|j-j_1|,|j-j_2|\}\le O(1)}}\cB_{kk_1k_2}^{jj_1j_2}\|_{L^2}\\
\ls& \sum_{j_1,k_1,k_2\ge-1}2^{j_1\alpha+N_1k_1+5k_2}
\|Q_{j_1}P_{k_1}V_\mu\|_{L^2}\|e^{it\nu\Lambda}P_{[[k_2]]}Q_{j_2}P_{k_2}V_\nu\|_{L^\infty}\\
\ls& \ve_1^2.
\end{split}
\end{equation}

\vskip 0.2cm

\noindent\textbf{Case 2.2. $\max\{|j-j_1|,|j-j_2|\}\geq O(1)$}

\vskip 0.2cm
It is noted that $\cB_{kk_1k_2}^{jj_1j_2}$ can be rewritten as
\begin{equation*}
\begin{split}
&\cB_{kk_1k_2}^{jj_1j_2}(t,x)
=(2\pi)^{-2}\psi_j(x)\iint_{\R^2}K_5(x-x_1,x-x_2)Q_{j_1}P_{k_1}V_\mu(t,x_1)
Q_{j_2}P_{k_2}V_\nu(t,x_2)dx_1dx_2,\\
&K_5(x-x_1,x-x_2):=\iint_{\R^2}e^{i\Psi_5}(\Phi_{\mu\nu}^{-1}a_{\mu\nu})(\xi_1,\xi_2)
\psi_k(\xi_1+\xi_2)\psi_{[[k_1]]}(\xi_1)\psi_{[[k_2]]}(\xi_2)d\xi_1d\xi_2,\\
&\Psi_5:=t(-\Lambda(\xi_1+\xi_2)+\mu\Lambda(\xi_1)+\nu\Lambda(\xi_2))
+\xi_1(x-x_1)+\xi_2(x-x_2).
\end{split}
\end{equation*}
By \eqref{2phase:bdd1}, \eqref{2phase:bdd2} and \eqref{symbol:a}, we have
\begin{equation*}
|\p_{\xi_1,\xi_2}^l(\Phi_{\mu\nu}^{-1}a_{\mu\nu})|\ls2^{k_2},\qquad l\ge0,
\end{equation*}
where $|\xi_1|\approx2^{k_1}$ and $|\xi_2|\approx2^{k_2}$.
When $\max\{j,j_1,j_2\}\ge\log_2(1+t)+O(1)$ and $\max\{|j-j_1|,|j-j_2|\}\ge O(1)$, for $x\in\supp\psi_j$, $x_1\in\supp\psi_{j_1}$ and $x_2\in\supp\psi_{j_2}$, one can see that
\begin{equation*}
|x-x_1|+|x-x_2|\ge2^{O(1)}(1+t),\qquad|x-x_1|+|x-x_2|\gt2^{\max\{j,j_1,j_2\}}.
\end{equation*}
This ensures
\begin{equation*}
|\p_{\xi_1}\Psi_5|+|\p_{\xi_2}\Psi_5|\gt|x-x_1|+|x-x_2|
\gt\max\{1+t,2^{\max\{j,j_1,j_2\}}\}.
\end{equation*}
Let
\begin{equation*}
\begin{split}
\cL_5&:=-i(|\p_{\xi_1}\Psi_5|^2+|\p_{\xi_2}\Psi_5|^2)^{-1}
(\p_{\xi_1}\Psi_5\p_{\xi_1}+\p_{\xi_2}\Psi_5\p_{\xi_2}),\\
\cL_5^*&:=i\sum_{l=1}^2\p_{\xi_l}\Big(\frac{\p_{\xi_l}\Psi_5~\cdot}
{|\p_{\xi_1}\Psi_5|^2+|\p_{\xi_2}\Psi_5|^2}\Big).
\end{split}
\end{equation*}
Then $L_5e^{i\Psi_5}=e^{i\Psi_5}$. It follows from the method of stationary phase that
\begin{equation*}
\begin{split}
&\quad\;|K_5(x-x_1,x-x_2)|\\
&=\Big|\iint_{\R^2}\cL_5^4(e^{i\Psi_5})(\Phi_{\mu\nu}^{-1}a_{\mu\nu})(\xi_1,\xi_2)
\psi_k(\xi_1+\xi_2)\psi_{[[k_1]]}(\xi_1)\psi_{[[k_2]]}(\xi_2)d\xi_1d\xi_2\Big|\\
&\ls\iint_{\R^2}|(\cL_5^*)^4[(\Phi_{\mu\nu}^{-1}a_{\mu\nu})(\xi_1,\xi_2)
\psi_k(\xi_1+\xi_2)\psi_{[[k_1]]}(\xi_1)\psi_{[[k_2]]}(\xi_2)]|d\xi_1d\xi_2\\
&\ls2^{k_1+2k_2-\max\{j,j_1,j_2\}}(1+|x-x_1|+|x-x_2|)^{-3}.
\end{split}
\end{equation*}
This, together with the H\"{o}lder inequality \eqref{Holder}, the Bernstein inequality and \eqref{thm1BA1}, leads to
\begin{equation*}
\begin{split}
\|\cB_{kk_1k_2}^{jj_1j_2}(t)\|_{L^2}&\ls2^{k_1+2k_2-\max\{j,j_1,j_2\}}
\|P_{k_1}V_\mu\|_{L^2}\|P_{k_2}V_\nu\|_{L^\infty}\\
&\ls2^{(k_1+k_2)(3-N)-\max\{j,j_1,j_2\}}\ve_1^2.
\end{split}
\end{equation*}
Therefore,
\begin{equation}\label{bdry6}
\sum_{j,k\ge-1}2^{j\alpha+N_1k}\|\sum_{\substack{j_1,j_2\ge-1,\\(k_1,k_2)\in\cX_k}}
\sum_{\substack{\max\{j,j_1,j_2\}\ge\log_2(1+t)+O(1),\\
\max\{|j-j_1|,|j-j_2|\}\ge O(1)}}\cB_{kk_1k_2}^{jj_1j_2}\|_{L^2}
\ls\ve_1^2.
\end{equation}
Substituting \eqref{bdry4}--\eqref{bdry6} into \eqref{bdry3} derives \eqref{bdry2}.
The estimate \eqref{bdry2} also holds for $\cB_k^{II}$. Thus, \eqref{bdry} is proved.
\end{proof}

\section{Proofs of Theorem \ref{thm1} and Corollaries \ref{coro1} and \ref{coro2}}

\begin{proof}[Proof of Theorem \ref{thm1}]
Suppose that the bootstrap assumption \eqref{thm1BA} holds for $\alpha\in(0,1/2]$ and $t\in[0,T_{\alpha,\ve}]$.
Next we show that the upper bound $\ve_1$ can be improved to $\frac34\ve_1$ in \eqref{thm1BA}.

At first, we deal with $\|V(t)\|_{H^N(\R)}=\|U(t)\|_{H^N(\R)}$.
It can be concluded from \eqref{loc:disp} with $\beta=\alpha$ and \eqref{thm1BA1} that
\begin{equation*}
\begin{split}
&\quad~\|U(s)\|_{W^{1,\infty}}+\sum_{k\ge-1}2^{k(7+1/4)}\|P_kU(s)\|_{L^\infty}\\
&\ls\sum_{j,k\ge-1}2^{k(7+1/4)}\|P_{[k-1,k+1]}e^{-is\Lambda}Q_jP_kV(s)\|_{L^\infty}\\
&\ls(1+s)^{-\alpha}\sum_{j,k\ge-1}2^{k(8+3/4)+\alpha j}\|Q_jP_kV(s)\|_{L^2}\\
&\ls\ve_1(1+s)^{-\alpha}.
\end{split}
\end{equation*}
This, together with \eqref{initial:data}, \eqref{energy} and \eqref{thm1BA1}, yields that for $t\in[0,T_{\alpha,\ve}]$,
\begin{equation*}
\|U(t)\|_{H^N(\R)}\ls
\left\{
\begin{aligned}
&\ve+\ve_1^2+\ve_1^3\ln(1+t),\qquad&&\alpha=1/2,\\
&\ve+\ve_1^2+\ve_1^3t^{1-2\alpha},&&\alpha\in(0,1/2).
\end{aligned}
\right.
\end{equation*}
We now turn to the estimate of $\|V(t)\|_{Z_\alpha}$.
Note that for $t\in[0,T_{\alpha,\ve}]$, \eqref{initial:data}, \eqref{QjPk},
Lemmas \ref{lem:cubic}, \ref{lem:quartic} and \ref{lem:bdry} show
\begin{equation*}
\|V(t)\|_{Z_\alpha}\ls
\left\{
\begin{aligned}
&\ve+\ve_1^2+\ve_1^3\ln(1+t)_,&&\alpha=1/2,\\
&\ve+\ve_1^2+\ve_1^3t^{1-2\alpha},\qquad&&\alpha\in(0,1/2).
\end{aligned}
\right.
\end{equation*}
Thus, there is a constant $C_1\ge1$ such that for $t\in[0,T_{\alpha,\ve}]$,
\begin{equation}\label{thm1pf1}
\begin{split}
\|V(t)\|_{H^N(\R)}+\|V(t)\|_{Z_\alpha}\le
\left\{
\begin{aligned}
&C_1(\ve+\ve_1^2+\ve_1^3\ln(1+t)),&&\alpha=1/2,\\
&C_1(\ve+\ve_1^2+\ve_1^3t^{1-2\alpha}),\qquad&&\alpha\in(0,1/2).
\end{aligned}
\right.
\end{split}
\end{equation}
Choosing $\ve_1=4C_1\ve$, $\ve_0=\frac{1}{16C_1^2}$ and
\begin{equation*}
\kappa_0=
\left\{
\begin{aligned}
&\frac{1}{64C_1^3},&&\alpha=1/2,\\
&\frac{1}{(64C_1^3)^\frac{1}{1-2\alpha}},\qquad&&\alpha\in(0,1/2),
\end{aligned}
\right.
\end{equation*}
then \eqref{thm1pf1} shows that for $t\in[0,T_{\alpha,\ve}]$,
\begin{equation}\label{thm1pf2}
\|V(t)\|_{H^N(\R)}+\|V(t)\|_{Z_\alpha}\le\frac14\ve_1+\frac14\ve_1+\frac14\ve_1
=\frac34\ve_1.
\end{equation}
This, together with the local existence of classical solution to \eqref{KG} and Proposition \ref{Prop:Cont},
yields that \eqref{KG} admits a unique classical solution $u\in C([0,T_{\alpha,\ve}],H^{N+1}(\R))\cap C^1([0,T_{\alpha,\ve}],H^N(\R))$.

Moreover, \eqref{thm1:disp} is a result of \eqref{loc:disp}, \eqref{profile:def} and \eqref{thm1pf2}.
\end{proof}

\begin{proof}[Proof of Corollary \ref{coro1}]
At first, we consider the case of $\beta\in(1/2,1]$ and compute $\|(\Lambda u_0,u_1)\|_{Z_{1/2}}$.
For any $\beta\in(1/2,1]$ and function $f$, one obtains from \eqref{Znorm:def} that
\begin{equation*}
\begin{split}
\|f\|_{Z_{1/2}}&=\sum_{j,k\ge-1}2^{j(1/2-\beta)}2^{j\beta+12k}\|Q_jP_kf\|_{L^2}\\
&\ls\sum_{k\ge-1}2^{12k}\Big(\sum_{j\ge-1}2^{j(1-2\beta)}\Big)^{1/2}\|2^{j\beta}\|Q_jP_kf\|_{L^2}\|_{\ell_j^2}.
\end{split}
\end{equation*}
The fact of $\|2^{j\beta}\|Q_jg\|_{L^2}\|_{\ell_j^2}\approx \|\w{x}^\beta g\|_{L^2}$ leads to
\begin{equation}\label{coro1pf1}
\begin{split}
\|f\|_{Z_{1/2}}&\ls\frac{1}{\sqrt{1-2^{1-2\beta}}}\sum_{k\ge-1}2^{12k}
\|\w{x}^\beta P_kf\|_{L^2}\\
&\ls\frac{1}{\sqrt{2\beta-1}}\sum_{k\ge-1}2^{12k}\|\w{x}^\beta P_k\Lambda^{-14}\Lambda^{14}f\|_{L^2}.
\end{split}
\end{equation}
Note that
\begin{equation}\label{coro1pf2}
\begin{split}
(P_k\Lambda^{-14}g)(x)&=\int_{\R}\cK(x-y)g(y)dy,\\
\cK(x-y)&=\frac{1}{2\pi}\int_{\R}e^{i\xi(x-y)}\frac{\psi_k(\xi)}{(1+\xi^2)^7}d\xi.
\end{split}
\end{equation}
It follows from the stationary method that
\begin{equation*}
|\cK(x-y)|\ls2^{-13k}(1+2^k|x-y|)^{-3}.
\end{equation*}
This, together with \eqref{coro1pf1}, \eqref{coro1pf2} and Young's inequality, derives that
\begin{equation*}
\begin{split}
\|f\|_{Z_{1/2}}&\ls\frac{1}{\sqrt{2\beta-1}}\sum_{k\ge-1}2^{12k}
\Big\|\int_{\R}\w{x-y}^\beta|\cK(x-y)|\w{y}^\beta|(\Lambda^{14}f)(y)|dy\Big\|_{L^2_x}\\
\ls&\frac{1}{\sqrt{2\beta-1}}\sum_{k\ge-1}2^{12k}\|\w{\cdot}^\beta\cK(\cdot)\|_{L^1(\R)}
\|\w{x}^\beta\Lambda^{14}f\|_{L^2_x}\\
\ls&\frac{1}{\sqrt{2\beta-1}}\|\w{x}^\beta\Lambda^{14}f\|_{L^2_x}.
\end{split}
\end{equation*}
Hence, there is a positive constant $C_2>0$ such that
\begin{equation*}
\ve=\|u_0\|_{H^{N+1}(\R)}+\|u_1\|_{H^N(\R)}+\|(\Lambda u_0,u_1)\|_{Z_{1/2}}
\le\frac{C_2\eps}{\sqrt{2\beta-1}},
\end{equation*}
which yields
\begin{equation*}
T_{1/2,\ve}=e^{\kappa_0/\ve^2}-1\ge e^{\frac{\kappa_0(2\beta-1)}{C_2^2\eps^2}}-1.
\end{equation*}
Choosing $\eps_1=\frac{\ve_0\sqrt{2\beta-1}}{C_2}$ and $\kappa_1=\frac{\kappa_0(2\beta-1)}{C_2^2}$.
For $\eps\le\eps_1$, \eqref{KG} admits a unique classical solution $u\in C([0,e^{\kappa_1/\eps^2}-1],H^{N+1}(\R))\cap C^1([0,e^{\kappa_1/\eps^2}-1],H^N(\R))$.

If $\beta>1$, one can find that $\|\w{x}\Lambda^{14}f\|_{L^2}\le\|\w{x}^\beta\Lambda^{14}f\|_{L^2}$
and further Corollary \ref{coro1}  holds.
\end{proof}

\begin{proof}[Proof of Corollary \ref{coro2}]
Similarly to the proof of \eqref{coro1pf1}, it holds that for any $\beta\in(0,1/2)$,
\begin{equation*}
\|f\|_{Z_\beta}\ls\frac{1}{\sqrt{1-2\beta}}\|\w{x}^\frac12\Lambda^{14}f\|_{L^2}.
\end{equation*}
Note that there is a positive constant $C_3$ such that
\begin{equation*}
\ve=\|u_0\|_{H^{N+1}(\R)}+\|u_1\|_{H^N(\R)}+\|(\Lambda u_0,u_1)\|_{Z_\beta}
\le\frac{C_3\eps}{\sqrt{1-2\beta}},
\end{equation*}
which yields
\begin{equation*}
T_{\beta,\ve}=\frac{\kappa_0}{\ve^{\frac{2}{1-2\beta}}}
\ge\frac{\kappa_0(1-2\beta)^\frac{1}{1-2\beta}}{(C_3\eps)^{\frac{2}{1-2\beta}}}.
\end{equation*}
Since there exists $\beta\in(0,1/2)$ such that $\beta\ge1/2-\frac{1}{M+1}$,
then by the choice of $\eps_2=\min\{\frac{\ve_0\sqrt{1-2\beta}}{C_3},\frac{\kappa_0(1-2\beta)^\frac{1}{1-2\beta}}{(C_3)^{\frac{2}{1-2\beta}}}\}$
and for $\eps\le\eps_2$, \eqref{KG} admits a unique classical solution $u\in C([0,\eps^{-M}],H^{N+1}(\R))\cap C^1([0,\eps^{-M}],H^N(\R))$.
\end{proof}

\appendix
\setcounter{equation}{1}

\section{Estimates of multi-linear Fourier multipliers}

\begin{lemma}\label{lem:bilinear}
Suppose that $T_{m_2}(f,g)$ is defined by \eqref{m-linear:def} with functions $f,g$ on $\R$.
For any $k_1,k_2\ge-1$ and $p,q,r\in[1,\infty]$ satisfying $1/p=1/q+1/r$, it holds that
\addtocounter{equation}{1}
\begin{align}
&\|T_{\Phi^{-1}_{\mu_1\mu_2}a_{\mu_1\mu_2}}(P_{k_1}f,P_{k_2}g)\|_{L^p(\R)}
\ls2^{5\min\{k_1,k_2\}}\|P_{k_1}f\|_{L^q(\R)}\|P_{k_2}g\|_{L^r(\R)},
\tag{\theequation a}\label{bilinear:a}\\
&\|T_{a_{\mu_1\mu_2}}(P_{k_1}f,P_{k_2}g)\|_{L^p(\R)}
+\|T_{a_{\sigma\mu_1\mu_2}}(P_{k_1}f,P_{k_2}g)\|_{L^p(\R)}
\ls\|P_{k_1}f\|_{L^q(\R)}\|P_{k_2}g\|_{L^r(\R)},
\tag{\theequation b}\label{bilinear:b}
\end{align}
where $\Phi_{\mu_1\mu_2}$, $a_{\mu_1\mu_2}$ and $a_{\sigma\mu_1\mu_2}$ are defined by \eqref{phase:def},
\eqref{symbol:a} and \eqref{symbol:a'}, respectively.
\end{lemma}
\begin{proof}
According to \eqref{proj:proj} and the definition of the multi-linear pseudoproduct operator \eqref{m-linear:def},
we have
\begin{equation}\label{bilinear1}
\begin{split}
&T_{m_2}(P_{k_1}f,P_{k_2}g)(x)=(2\pi)^{-2}\iint_{\R^2}\cK(x-y,x-z)P_{k_1}f(y)P_{k_2}g(z)dydz,\\
&\cK(y,z)=\iint_{\R^2}e^{i(y\xi_1+z\xi_2)}m_2(\xi_1,\xi_2)\psi_{k_1k_2}(\xi_1,\xi_2)
d\xi_1d\xi_2,\\
&\psi_{k_1k_2}(\xi_1,\xi_2):=\psi_{[k_1-1,k_1+1]}(\xi_1)\psi_{[k_2-1,k_2-1]}(\xi_2).
\end{split}
\end{equation}
As in Lemma 3.3 of \cite{DIP17}, the $L^1$ norm of the Schwartz kernel $\cK(y,z)$ can be bounded by
\begin{equation}\label{bilinear2}
\begin{split}
&\|\cK(y,z)\|_{L^1(\R^2)}\ls\|(1+|2^{k_1}y|+|2^{k_2}z|)^2\cK(y,z)\|_{L^2(\R^2)}
\|(1+|2^{k_1}y|+|2^{k_2}z|)^{-2}\|_{L^2(\R^2)}\\
&\ls\sum_{l=0}^2(2^{lk_1}
\|\psi_{k_1k_2}(\xi_1,\xi_2)\p_{\xi_1}^lm_2(\xi_1,\xi_2)\|_{L^\infty}
+2^{lk_2}\|\psi_{k_1k_2}(\xi_1,\xi_2)\p_{\xi_2}^lm_2(\xi_1,\xi_2)\|_{L^\infty}).
\end{split}
\end{equation}
Inspired by Lemma 4.5 in \cite{Zheng19}, we next show
\begin{equation}\label{bilinear3}
(1+|\xi_1|)^l|\p^l_{\xi_1}\Phi^{-1}_{\mu_1\mu_2}(\xi_1,\xi_2)|
+(1+|\xi_2|)^l|\p^l_{\xi_2}\Phi^{-1}_{\mu_1\mu_2}(\xi_1,\xi_2)|
\ls(1+\min\{|\xi_1|,|\xi_2|\})^{2l+1},l\ge0,
\end{equation}
which yields
\begin{equation}\label{bilinear4}
\sum_{l=0}^2(2^{lk_1}|\psi_{k_1k_2}(\xi_1,\xi_2)\p^l_{\xi_1}\Phi^{-1}_{\mu_1\mu_2}(\xi_1,\xi_2)|
+2^{lk_2}|\psi_{k_1k_2}(\xi_1,\xi_2)\p^l_{\xi_2}\Phi^{-1}_{\mu_1\mu_2}(\xi_1,\xi_2)|)
\ls2^{5\min\{k_1,k_2\}}.
\end{equation}
It is pointed out that the analogous result to \eqref{bilinear4} has been obtained in
\cite{HY23} for space dimensions $d\ge2$.
However, we require the more precise estimate \eqref{bilinear3} for 1D case, which will be utilized in the next lemma.

Note that \eqref{symbol:a} and \eqref{symbol:a'} imply
\begin{equation}\label{bilinear5}
\begin{split}
&\sum_{l=0}^2(2^{lk_1}|\psi_{k_1k_2}(\xi_1,\xi_2)\p_{\xi_1}^la_{\mu_1\mu_2}(\xi_1,\xi_2)|
+2^{lk_2}|\psi_{k_1k_2}(\xi_1,\xi_2)\p_{\xi_2}^la_{\mu_1\mu_2}(\xi_1,\xi_2)|)\ls1,\\
&\sum_{l=0}^2(2^{lk_1}|\psi_{k_1k_2}(\xi_1,\xi_2)\p_{\xi_1}^la_{\sigma\mu_1\mu_2}(\xi_1,\xi_2)|
+2^{lk_2}|\psi_{k_1k_2}(\xi_1,\xi_2)\p_{\xi_2}^la_{\sigma\mu_1\mu_2}(\xi_1,\xi_2)|)\ls1.
\end{split}
\end{equation}
On the other hand, if \eqref{bilinear3} has been proved, then it follows from
\eqref{bilinear1}, \eqref{bilinear2}, \eqref{bilinear4}, \eqref{bilinear5} and the H\"{o}lder inequality \eqref{Holder}
that \eqref{bilinear:a} and \eqref{bilinear:b} hold.

Without loss of generality, $|\xi_1|\le|\xi_2|$ is assumed since the case of $|\xi_1|\ge|\xi_2|$ can be treated analogously.

The estimate on the first term of left hand side in \eqref{bilinear3} follows from $|\p_{\xi_1}^l\Phi^{-1}_{\mu_1\mu_2}(\xi_1,\xi_2)|\ls|\Phi^{-1}_{\mu_1\mu_2}(\xi_1,\xi_2)|$
$\ls1+|\xi_1|$ due to \eqref{2phase:bdd2}.
In addition, the second term of left hand side in \eqref{bilinear3} can be easily shown for the case of $|\xi_1|\ge2^{-10}|\xi_2|$.
We next deal with the second term in \eqref{bilinear3} for $|\xi_1|\le2^{-10}|\xi_2|$ and $|\xi_2|\ge1$.

For $\p_{\xi_2}^l\Phi_{\mu+}$ with $l\ge1$, there is some $r\in[0,1]$ such that
\begin{equation*}
|\p_{\xi_2}^l\Phi_{\mu+}(\xi_1,\xi_2)|=|-\Lambda^{(l)}(\xi_1+\xi_2)+\Lambda^{(l)}(\xi_2)|
=|\xi_1\Lambda^{(l+1)}(r\xi_1+\xi_2)|\ls|\xi_1|(1+|\xi_2|)^{-l},
\end{equation*}
which derives $(1+|\xi_2|)^l|\p_{\xi_2}^l\Phi_{\mu+}(\xi_1,\xi_2)|\ls1+|\xi_1|$.
By \eqref{2phase:bdd1} and Leibnitz's rules, one has
\begin{equation*}
(1+|\xi_2|)^l|\p_{\xi_2}^l\Phi^{-1}_{\mu+}(\xi_1,\xi_2)|\ls(1+|\xi_1|)^{2l+1},
\quad l\ge0.
\end{equation*}
This yields \eqref{bilinear3} and \eqref{bilinear4} for $\mu_2=+$.

For $\p_{\xi_2}^l\Phi_{\mu-}$, according to the definition \eqref{phase:def}, it is known that
there is a positive constant $C>0$ such that
\begin{equation*}
-\Phi_{\mu-}(\xi_1,\xi_2)=\Lambda(\xi_1+\xi_2)-\mu\Lambda(\xi_1)+\Lambda(\xi_2)
\ge\Lambda(\xi_1+\xi_2)\ge C|\xi_2|.
\end{equation*}
When $l\ge1$,
$|\p_{\xi_2}^l\Phi_{\mu-}(\xi_1,\xi_2)|=|\Lambda^{(l)}(\xi_1+\xi_2)+\Lambda^{(l)}(\xi_2)|
\lesssim |\xi_2|^{1-l}$ holds.
Analogously, for $l\ge0$, one has
$|\p_{\xi_2}^l\Phi_{\mu-}^{-1}(\xi_1,\xi_2)|\ls|\xi_2|^{-1-l}$,
which implies \eqref{bilinear3} for $\mu_2=-$.
\end{proof}

\begin{lemma}\label{lem:trilin}
Suppose that $T_{m_3}(f,g,h)$ is defined by \eqref{m-linear:def} with functions $f,g,h$ on $\R$.
For any $k_1,k_2,k_3\ge-1$ and $p,q_1,q_2,q_3\in[1,\infty]$ satisfying $1/p=1/q_1+1/q_2+1/q_3$, it holds that
\begin{equation}\label{trilin}
\begin{split}
\|T_{b_{\mu_1\mu_2\mu_3}}(P_{k_1}f,P_{k_2}g,P_{k_3}h)\|_{L^p(\R)}
&\ls\|P_{k_1}f\|_{L^{q_1}(\R)}\|P_{k_2}g\|_{L^{q_2}(\R)}\|P_{k_3}h\|_{L^{q_3}(\R)},\\
\|T_{m_{\mu_1\mu_2\mu_3}}(P_{k_1}f,P_{k_2}g,P_{k_3}h)\|_{L^p(\R)}
&\ls2^{7\med\{k_1,k_2,k_3\}}\|P_{k_1}f\|_{L^{q_1}(\R)}\\
&\qquad\times\|P_{k_2}g\|_{L^{q_2}(\R)}\|P_{k_3}h\|_{L^{q_3}(\R)},
\end{split}
\end{equation}
where $b_{\mu_1\mu_2\mu_3}$ and $m_{\mu_1\mu_2\mu_3}$ are defined by \eqref{symbol:b} and \eqref{symbol:m}, respectively.
For $(\mu_1,\mu_2,\mu_3)\in\{(+++),(+--),(---)\}$, one has
\begin{equation}\label{trilin:good}
\begin{split}
\|T_{\Phi^{-1}_{\mu_1\mu_2\mu_3}m_{\mu_1\mu_2\mu_3}}
(P_{k_1}f,P_{k_2}g,P_{k_3}h)\|_{L^p(\R)}
&\ls2^{8\med\{k_1,k_2,k_3\}}\|P_{k_1}f\|_{L^{q_1}(\R)}\\
&\qquad\times\|P_{k_2}g\|_{L^{q_2}(\R)}\|P_{k_3}h\|_{L^{q_3}(\R)},
\end{split}
\end{equation}
where $\Phi_{\mu_1\mu_2\mu_3}$ is defined by \eqref{phase:def}.
\end{lemma}
\begin{proof}
Similarly to \eqref{bilinear1} and \eqref{bilinear2}, we have
\begin{equation}\label{trilinear1}
\begin{split}
&T_{m_3}(P_{k_1}f,P_{k_2}g,P_{k_3}h)(x)=(2\pi)^{-3}\iiint_{\R^3}
\cK(x-x_1,x-x_2,x-x_3)P_{k_1}f(x_1)\\
&\hspace{6cm}\times P_{k_2}g(x_2)P_{k_3}h(x_3)dx_1dx_2dx_3,\\
&\cK(x_1,x_2,x_3)=\iiint_{\R^3}e^{i(x_1\xi_1+x_2\xi_2+x_3\xi_3)}
m_3(\xi_1,\xi_2,\xi_3)\psi_{k_1k_2k_3}(\xi_1,\xi_2,\xi_3)d\xi_1d\xi_2d\xi_3,\\
&\psi_{k_1k_2k_3}(\xi_1,\xi_2,\xi_3):=\psi_{[k_1-1,k_1+1]}(\xi_1)
\psi_{[k_2-1,k_2-1]}(\xi_2)\psi_{[k_3-1,k_3-1]}(\xi_3)
\end{split}
\end{equation}
and
\begin{equation}\label{trilinear2}
\begin{split}
&\quad\;\|\cK(x_1,x_2,x_3)\|_{L^1(\R^3)}\\
&\ls\|(1+|2^{k_1}x_1|+|2^{k_2}x_2|+|2^{k_3}x_3|)^2\cK\|_{L^2(\R^3)}
\|(1+|2^{k_1}x_1|+|2^{k_2}x_2|+|2^{k_3}x_3|)^{-2}\|_{L^2(\R^3)}\\
&\ls\sum_{l=0}^2\sum_{\iota=1}^32^{lk_{\iota}}
\|\psi_{k_1k_2k_3}(\xi_1,\xi_2,\xi_3)\p_{\xi_{\iota}}^lm_3(\xi_1,\xi_2,\xi_3)\|_{L^\infty}.
\end{split}
\end{equation}
According to the definition \eqref{symbol:b}, one has
\begin{equation*}
\sum_{l=0}^2\sum_{\iota=1}^32^{lk_{\iota}}\|\psi_{k_1k_2k_3}(\xi_1,\xi_2,\xi_3)
\p_{\xi_{\iota}}^lb_{\mu_1\mu_2\mu_3}(\xi_1,\xi_2,\xi_3)\|_{L^\infty}\ls1.
\end{equation*}
This, together with \eqref{trilinear1} and \eqref{trilinear2}, yields the first inequality of \eqref{trilin}.

In the remaining part, we focus on the proof for the second inequality of \eqref{trilin} and \eqref{trilin:good}.
For $l\ge0$, one can calculate from \eqref{2phase:bdd2} and the definition \eqref{symbol:m} to obtain
\begin{equation}\label{symbol:m:bdd}
\begin{split}
&\quad\;|\p_{\xi_1,\xi_2,\xi_3}^lm_{\mu_1\mu_2\mu_3}(\xi_1,\xi_2,\xi_3)|\\
&\ls1+\min\{|\xi_1|,|\xi_2+\xi_3|\}+\min\{|\xi_2|,|\xi_1+\xi_3|\}
+\min\{|\xi_3|,|\xi_1+\xi_2|\}\\
&\ls2^{\med\{k_1,k_2,k_3\}}.
\end{split}
\end{equation}
If $\med\{k_1,k_2,k_3\}\ge\max\{k_1,k_2,k_3\}-O(1)$, then it is deduced from \eqref{symbol:m:bdd} that
\begin{equation}\label{trilinear3}
\begin{split}
&\quad\sum_{l=0}^2\sum_{\iota=1}^32^{lk_{\iota}}
\|\psi_{k_1k_2k_3}(\xi_1,\xi_2,\xi_3)\p_{\xi_{\iota}}^l
m_{\mu_1\mu_2\mu_3}(\xi_1,\xi_2,\xi_3)\|_{L^\infty}\\
&\ls2^{2\max\{k_1,k_2,k_3\}}\max_{\iota=1,2,3}\sum_{l=0}^2
\|\p_{\xi_{\iota}}^lm_{\mu_1\mu_2\mu_3}(\xi_1,\xi_2,\xi_3)\|_{L^\infty}\\
&\ls2^{3\med\{k_1,k_2,k_3\}}.
\end{split}
\end{equation}
For $l\ge1$, $|\Lambda^{(l)}(y)|\ls1$ and further $|\p^l_{\xi_1,\xi_2,\xi_3}\Phi_{\mu_1\mu_2\mu_3}|\ls1$ hold.
For $(\mu_1,\mu_2,\mu_3)\in\{(+++),(+--),(---)\}$, it follows from \eqref{3phase:bdd} that
\begin{equation}\label{trilinear4}
|\p^l_{\xi_1,\xi_2,\xi_3}\Phi^{-1}_{\mu_1\mu_2\mu_3}|
\ls\sum_{l_1=1}^l(|\Phi_{\mu_1\mu_2\mu_3}|)^{-1-l_1}
\ls2^{(l+1)\min\{k_1,k_2,k_3\}}.
\end{equation}
Therefore, \eqref{trilinear1}-\eqref{trilinear4} together with the H\"{o}lder inequality imply the second inequality of \eqref{trilin} and \eqref{trilin:good} for the case of $\med\{k_1,k_2,k_3\}\ge\max\{k_1,k_2,k_3\}-O(1)$.

Next, we turn to the proof of the second inequality in \eqref{trilin} and \eqref{trilin:good} for the case of $\med\{k_1,k_2,k_3\}\le\max\{k_1,k_2,k_3\}-O(1)$. To this end,
we are devoted to establishing the following estimate
\begin{equation}\label{trilinear5}
\sum_{\iota=1}^32^{lk_{\iota}}\|\psi_{k_1k_2k_3}(\xi_1,\xi_2,\xi_3)
\p_{\xi_{\iota}}^lm_{\mu_1\mu_2\mu_3}(\xi_1,\xi_2,\xi_3)\|_{L^\infty}
\ls2^{(3l+1)\med\{k_1,k_2,k_3\}},\quad l\ge0.
\end{equation}
This, together with \eqref{trilinear1}, \eqref{trilinear2} and the H\"{o}lder inequality, will  imply
the second inequality in \eqref{trilin} for the case of $\med\{k_1,k_2,k_3\}\le\max\{k_1,k_2,k_3\}-O(1)$.

Note that by the definition \eqref{symbol:m}, $m^{II}_{\mu_1\mu_2\mu_3}(\xi_1,\xi_2,\xi_3)$ is a linear combination of the products of \eqref{symbol:b} and  one then has
\begin{equation}\label{trilinear6}
\sum_{\iota=1}^32^{lk_{\iota}}\|\psi_{k_1k_2k_3}(\xi_1,\xi_2,\xi_3)
\p_{\xi_{\iota}}^lm^{II}_{\mu_1\mu_2\mu_3}(\xi_1,\xi_2,\xi_3)\|_{L^\infty}
\ls1,\quad l\ge0.
\end{equation}
Meanwhile, $m^I_{\mu_1\mu_2\mu_3}(\xi_1,\xi_2,\xi_3)$ is a linear combination of trinomial products of $a_{\sigma_1\sigma_2}$, $\tilde a_{\nu_1\nu_2\nu_3}$ and
\begin{equation}\label{trilinear7}
\Phi^{-1}_{\mu\nu}(\xi_1,\xi_2+\xi_3),\Phi^{-1}_{\mu\nu}(\xi_2,\xi_1+\xi_3),
\Phi^{-1}_{\mu\nu}(\xi_3,\xi_1+\xi_2).
\end{equation}
Based on \eqref{bilinear3}, we now show
\begin{equation}\label{trilinear8}
\sum_{\iota=1}^32^{lk_{\iota}}\|\psi_{k_1k_2k_3}(\xi_1,\xi_2,\xi_3)
\p_{\xi_{\iota}}^l(\Phi^{-1}_{\mu\nu}(\xi_1,\xi_2+\xi_3))\|_{L^\infty}
\ls2^{(3l+1)\med\{k_1,k_2,k_3\}},\quad l\ge0.
\end{equation}
Denote
\begin{equation*}
\tilde\Phi(\xi_1,\xi_2,\xi_3)=\Phi^{-1}_{\mu\nu}(\xi_1,\xi_2+\xi_3).
\end{equation*}
If $\max\{k_1,k_2,k_3\}=k_1$, one then has $|\xi_2+\xi_3|\ls|\xi_1|$, $|\xi_2+\xi_3|\ls2^{\max\{k_2,k_3\}}$ and $\max\{k_2,k_3\}=\med\{k_1,k_2,k_3\}$.
Therefore, it follows from \eqref{bilinear3} that
\begin{equation}\label{trilinear9}
\begin{split}
(1+|\xi_1|)^l|\p_{\xi_1}^l\tilde\Phi(\xi_1,\xi_2,\xi_3)|
&=(1+|\xi_1|)^l|\p_{\xi_1}^l\Phi^{-1}_{\mu\nu}(\xi_1,\xi_2+\xi_3)|\\
&\ls(1+|\xi_2+\xi_3|)^{2l+1},\\
&\ls2^{(2l+1)\med\{k_1,k_2,k_3\}}.
\end{split}
\end{equation}
On the other hand, we have
\begin{equation*}
\p_{\xi_2}^l\tilde\Phi(\xi_1,\xi_2,\xi_3)
=\p_{\xi_3}^l\tilde\Phi(\xi_1,\xi_2,\xi_3)
=\p_{\xi_2}^l\Phi^{-1}_{\mu\nu}(\xi_1,\xi_2+\xi_3),
\end{equation*}
which yields
\begin{equation}\label{trilinear10}
\begin{split}
&\quad\;(1+|\xi_2|)^l|\p_{\xi_2}^l\tilde\Phi(\xi_1,\xi_2,\xi_3)|
+(1+|\xi_3|)^l|\p_{\xi_3}^l\tilde\Phi(\xi_1,\xi_2,\xi_3)|\\
&\ls2^{l\max\{k_2,k_3\}}|\p_{\xi_2}^l\Phi^{-1}_{\mu\nu}(\xi_1,\xi_2+\xi_3)|\\
&\ls2^{(3l+1)\med\{k_1,k_2,k_3\}}.
\end{split}
\end{equation}
If $\max\{k_1,k_2,k_3\}=k_2$, by $\med\{k_1,k_2,k_3\}\le\max\{k_1,k_2,k_3\}-O(1)$,
one then has $k_3\le k_2-O(1)$.
Hence, $|\xi_2+\xi_3|\approx|\xi_2|\gt|\xi_1|$.
Similarly to \eqref{trilinear9} and \eqref{trilinear10}, we can obtain
\begin{equation}\label{trilinear11}
\begin{split}
(1+|\xi_1|)^l|\p_{\xi_1}^l\tilde\Phi(\xi_1,\xi_2,\xi_3)|
&=(1+|\xi_1|)^l|\p_{\xi_1}^l\Phi^{-1}_{\mu\nu}(\xi_1,\xi_2+\xi_3)|\\
&\ls(1+|\xi_1|)^{2l+1},\\
&\ls2^{(2l+1)\med\{k_1,k_2,k_3\}}
\end{split}
\end{equation}
and
\begin{equation}\label{trilinear12}
\begin{split}
&\quad\;(1+|\xi_2|)^l|\p_{\xi_2}^l\tilde\Phi(\xi_1,\xi_2,\xi_3)|
+(1+|\xi_3|)^l|\p_{\xi_3}^l\tilde\Phi(\xi_1,\xi_2,\xi_3)|\\
&\ls(1+|\xi_2+\xi_3|)^l|\p_{\xi_2}^l\Phi^{-1}_{\mu\nu}(\xi_1,\xi_2+\xi_3)|\\
&\ls2^{(2l+1)\med\{k_1,k_2,k_3\}}.
\end{split}
\end{equation}
For $\max\{k_1,k_2,k_3\}=k_3$, \eqref{trilinear11} and \eqref{trilinear12} still hold
by the analogous proof for the case of $\max\{k_1,k_2,$ $k_3\}=k_2$.

Collecting \eqref{trilinear9}-\eqref{trilinear12} yields \eqref{trilinear8}.
With the same argument, \eqref{trilinear8} also holds for the other two terms in \eqref{trilinear7}.
Thus, \eqref{trilinear5} is achieved by \eqref{trilinear6} and \eqref{trilinear8}.

At last, we prove \eqref{trilin:good} for the case of $\med\{k_1,k_2,k_3\}\le\max\{k_1,k_2,k_3\}-O(1)$.
For this purpose, it requires to establish the following estimates
\begin{equation}\label{trilinear13}
\sum_{\iota=1}^32^{lk_{\iota}}\|\psi_{k_1k_2k_3}(\xi_1,\xi_2,\xi_3)
\p_{\xi_{\iota}}^l\Phi^{-1}_{\mu_1\mu_2\mu_3}(\xi_1,\xi_2,\xi_3)\|_{L^\infty}
\ls2^{(2l+1)\med\{k_1,k_2,k_3\}},
\end{equation}
where $(\mu_1,\mu_2,\mu_3)\in\{(+++),(+--),(---)\}$ and $\med\{k_1,k_2,k_3\}\le\max\{k_1,k_2,k_3\}-O(1)$.

Combining \eqref{trilinear5} and \eqref{trilinear13} leads to
\begin{equation*}
\sum_{l=0}^2\sum_{\iota=1}^32^{lk_{\iota}}\|\psi_{k_1k_2k_3}(\xi_1,\xi_2,\xi_3)
\p_{\xi_{\iota}}^l(\Phi^{-1}_{\mu_1\mu_2\mu_2}m_{\mu_1\mu_2\mu_2})
(\xi_1,\xi_2,\xi_3)\|_{L^\infty}\ls2^{8\med\{k_1,k_2,k_3\}},
\end{equation*}
which yields \eqref{trilin:good} for the case of $\med\{k_1,k_2,k_3\}\le\max\{k_1,k_2,k_3\}-O(1)$.

If $\max\{k_1,k_2,k_3\}=k_1$, one then has $|\xi_2|,|\xi_3|\ll|\xi_1|$.
Similarly to Lemma \ref{lem:bilinear}, for $\p^l_{\xi_1}\Phi^{-1}_{+\mu_2\mu_3}$ with $l\ge1$, there is
some $r\in[0,1]$ such that
\begin{equation*}
\begin{split}
|\p^l_{\xi_1}\Phi_{+\mu_2\mu_3}(\xi_1,\xi_2,\xi_3)|
&=|\Lambda^{(l)}(\xi_1)-\Lambda^{(l)}(\xi_1+\xi_2+\xi_3)|\\
&=|(\xi_2+\xi_3)\Lambda^{(l+1)}(\xi_1+r(\xi_2+\xi_3))|\\
&\ls2^{\med\{k_1,k_2,k_3\}}(1+|\xi_1|)^{-l}.
\end{split}
\end{equation*}
This together with \eqref{3phase:bdd} derives
\begin{equation}\label{trilinear14}
(1+|\xi_1|)^l|\p^l_{\xi_1}\Phi^{-1}_{+\mu_2\mu_3}(\xi_1,\xi_2,\xi_3)|
\ls2^{(2l+1)\med\{k_1,k_2,k_3\}}.
\end{equation}
For $\p^l_{\xi_1}\Phi^{-1}_{-\mu_2\mu_3}$, we have
\begin{equation*}
\begin{split}
-\Phi_{-\mu_2\mu_3}(\xi_1,\xi_2,\xi_3)
&=\Lambda(\xi_1+\xi_2+\xi_3)+\Lambda(\xi_1)-\mu_2\Lambda(\xi_2)-\mu_3\Lambda(\xi_3)\\
&\ge\Lambda(\xi_1)\gt1+|\xi_1|
\end{split}
\end{equation*}
and
\begin{equation*}
|\p^l_{\xi_1}\Phi_{-\mu_2\mu_3}(\xi_1,\xi_2,\xi_3)|
=|\Lambda^{(l)}(\xi_1+\xi_2+\xi_3)+\Lambda^{(l)}(\xi_1)|
\ls(1+|\xi_1|)^{1-l},\quad l\ge1.
\end{equation*}
Thereby,
\begin{equation*}
|\p^l_{\xi_1}\Phi^{-1}_{-\mu_2\mu_3}|\ls(1+|\xi_1|)^{-1-l}.
\end{equation*}
Together with \eqref{trilinear14}, we can achieve
\begin{equation}\label{trilinear15}
(1+|\xi_1|)^l|\p^l_{\xi_1}\Phi^{-1}_{\mu_1\mu_2\mu_3}(\xi_1,\xi_2,\xi_3)|
\ls2^{(2l+1)\med\{k_1,k_2,k_3\}}.
\end{equation}
On the other hand, \eqref{trilinear4} implies
\begin{equation}\label{trilinear16}
(1+|\xi_2|)^l|\p^l_{\xi_2}\Phi^{-1}_{\mu_1\mu_2\mu_3}(\xi_1,\xi_2,\xi_3)|
+(1+|\xi_3|)^l|\p^l_{\xi_3}\Phi^{-1}_{\mu_1\mu_2\mu_3}(\xi_1,\xi_2,\xi_3)|
\ls2^{(2l+1)\med\{k_1,k_2,k_3\}}.
\end{equation}
Collecting \eqref{trilinear15} and \eqref{trilinear16} derives \eqref{trilinear13} for the case of $\max\{k_1,k_2,k_3\}=k_1$.
The proof of \eqref{trilinear13} for the case of $\max\{k_1,k_2,k_3\}=k_2$ or $\max\{k_1,k_2,k_3\}=k_3$
can be completed analogously.
\end{proof}

\begin{lemma}
Suppose that $T_{m_3}(f,g,h)$ is defined by \eqref{m-linear:def} with functions $f,g,h$ on $\R$.
For any $k_1,k_2,k_3\ge-1$ and $p,q_1,q_2,q_3\in[1,\infty]$ satisfying $\max\{k_1,k_2\}\le k_3-O(1)$, $1/p=1/q_1+1/q_2+1/q_3$,
it holds that
\begin{equation}\label{trilinear++-}
\begin{split}
\|T_{\Phi^{-1}_{++-}m_{++-}}(P_{k_1}f,P_{k_2}g,P_{k_3}h)\|_{L^p(\R)}
&\ls2^{7\max\{k_1,k_2\}}\|P_{k_1}f\|_{L^{q_1}(\R)}\\
&\qquad\times\|P_{k_2}g\|_{L^{q_2}(\R)}\|P_{k_3}h\|_{L^{q_3}(\R)}.
\end{split}
\end{equation}
\end{lemma}
\begin{proof}
It follows from a direct computation that for $\iota=1,2,3$,
\begin{equation*}
\begin{split}
|\p_{\xi_\iota}\Phi^{-1}_{++-}|&\ls|\p\Phi_{++-}||\Phi_{++-}|^{-2}\ls2^{-2k_3},\\
|\p^2_{\xi_\iota}\Phi^{-1}_{++-}|&\ls|\p^2\Phi_{++-}||\Phi_{++-}|^{-2}
+|\p\Phi_{++-}|^2|\Phi_{++-}|^{-3}\ls2^{-2k_3},
\end{split}
\end{equation*}
where we have used \eqref{badphase:bdd} and the fact of $|\p^l_{\xi_1,\xi_2,\xi_3}\Phi_{++-}|\ls1$ with $l\ge1$.
Thus, one can obtain
\begin{equation*}
\sum_{l=0}^2\sum_{\iota=1}^3(1+|\xi_\iota|)^l
|\p^l_{\xi_\iota}\Phi^{-1}_{++-}(\xi_1,\xi_2,\xi_3)|\ls1.
\end{equation*}
This, together with \eqref{trilinear1}, \eqref{trilinear2} and \eqref{trilinear5}, leads to \eqref{trilinear++-}.
\end{proof}

\end{document}